\documentclass[12pt, reqno, a4paper]{amsart}

\usepackage{ amssymb, amsmath, enumerate, amsfonts, amsthm, mathrsfs, url, bm}

\usepackage{xcolor}  	
\usepackage[backref]{hyperref}
\hypersetup{
	colorlinks,
    linkcolor={blue!60!black},
    citecolor={blue!60!black},
    urlcolor={red!60!black}
}

\usepackage{color}

\advance \topmargin by -\headheight
\advance \topmargin by -\headsep
\evensidemargin \oddsidemargin
\numberwithin{equation}{section}

\newtheorem{theorem}{Theorem}[section]
\newtheorem{lemma}[theorem]{Lemma}    
\newtheorem{corollary}[theorem]{Corollary}
\newtheorem{proposition}[theorem]{Proposition}
\newtheorem{conjecture}[theorem]{Conjecture}

\newtheorem*{bl}{Bergelson--Leibman}
\newtheorem*{lef}{Lefmann's criterion}

  \theoremstyle{remark}

\theoremstyle{definition}
\newtheorem{definition}[theorem]{Definition}
\newtheorem{remark}[theorem]{Remark}

\newtheorem*{notation}{Notation}

\theoremstyle{claim}
\newtheorem*{claim}{Claim}

\newtheorem*{question*}{Question}

 \newcommand{\set}[1]{\left\{#1\right\}}
\newcommand{\bigset}[1]{\bigl\{ #1 \bigr\}}
\newcommand{\Bigset}[1]{\Bigl\{ #1 \Bigr\}}

\newcommand{\abs}[1]{\left| #1\right|}
\newcommand{\bigabs}[1]{\bigl| #1 \bigr|}
\newcommand{\Bigabs}[1]{\Bigl| #1 \Bigr|}
\newcommand{\biggabs}[1]{\biggl| #1 \biggr|}
\newcommand{\Biggabs}[1]{\Biggl| #1 \Biggr|}
\newcommand{\sqbrac}[1]{\left[ #1 \right]}

\newcommand{\bigfloor}[1]{\bigl\lfloor #1 \bigr\rfloor}
\newcommand{\floor}[1]{\left\lfloor #1 \right\rfloor}
\newcommand{\brac}[1]{\left( #1 \right)}
\newcommand{\bigbrac}[1]{\bigl( #1 \bigr)}
\newcommand{\Bigbrac}[1]{\Bigl( #1 \Bigr)}

\newcommand{\norm}[1]{\left\| #1\right\|}
\newcommand{\bignorm}[1]{\big\| #1 \big\|}

\newcommand{\biggnorm}[1]{\bigg\| #1 \biggr\|}

\newcommand{\recip}[1]{\frac{1}{#1}}
\newcommand{\trecip}[1]{\tfrac{1}{#1}}

\newcommand{\vx}{\mathbf{x}}

\newcommand{\vc}{\mathbf{c}}

\newcommand{\N}{\mathbb{N}}
\newcommand{\Z}{\mathbb{Z}}

\newcommand{\R}{\mathbb{R}}
\newcommand{\C}{\mathbb{C}}
\newcommand{\T}{\mathbb{T}}

\newcommand{\meas}{\mathrm{meas}}

\newcommand{\intd}{\mathrm{d}}

\newcommand{\eps}{\varepsilon}
\newcommand{\hash}{\#}

\makeatletter
\let\@@pmod\pmod
\DeclareRobustCommand{\pmod}{\@ifstar\@pmods\@@pmod}
\def\@pmods#1{\mkern4mu({\operator@font mod}\mkern 6mu#1)}
\makeatother

\newcommand{\cR}{\mathcal{R}}
\renewcommand{\le}{\leqslant}
\renewcommand{\leq}{\leqslant}
\renewcommand{\geq}{\geqslant}

\renewcommand{\ge}{\geqslant}

\def \bT {{\mathbb T}}
\def \bN {{\mathbb N}}
\def \bZ {{\mathbb Z}}
\def \bC {{\mathbb C}}
\def \bR {{\mathbb R}}

\def \alp {{\alpha}}
\def \bet {{\beta}}
\def \gam {{\gamma}}
\def \del {{\delta}}
\def \kap {{\kappa}}
\def \tet {{\theta}}
\def \lam {{\lambda}}
\def \kap {{\kappa}}

\def \d {{\mathrm{d}}}

\def \by {{\mathbf{y}}}

\def \fM {{\mathfrak M}}
\def \fm {{\mathfrak m}}
\def \fN {{\mathfrak N}}

\def \cN {{\mathcal N}}

\def \cW {{\mathcal W}}

\newcommand{\mmod}[1]{{\,\,\mathrm{mod}\,\,#1}}

\def \bx {{\mathbf x}}

\def \cS {{\mathcal S}}

\def \fS {{\mathfrak S}}
\def \fJ {{\mathfrak J}}

\def \bc {{\mathbf{c}}}

\def \by {{\mathbf{y}}}
\def \bm {{\mathbf{m}}}

\def \fn {{\mathfrak n}}

\begin{document}

\title{Rado's criterion over squares and higher powers}

\author{Sam Chow}
\address{Department of Mathematics\\
University of York\\ Heslington \\ York YO10 5DD\\ 
United Kingdom}
\email{sam.chow42@gmail.com}

\author{Sofia Lindqvist}
\address{Mathematical Institute\\
University of Oxford\\ 
UK}
\email{lindqvist.sofia@gmail.com}

\author{Sean Prendiville}
\address{School of Mathematics\\ University of Manchester
\\ UK}
\email{sean.prendiville@manchester.ac.uk}


\subjclass[2010]{11B30, 11D72, 11L15}
\keywords{Arithmetic combinatorics, arithmetic Ramsey theory, Weyl sums, smooth numbers, restriction theory}

\date{\today}

\begin{abstract}
We establish partition regularity of the generalised Pythagorean equation in five or more variables.  Furthermore, we show how Rado's characterisation of a partition regular equation remains valid over the set of positive $k$th powers, provided the equation has at least $(1+o(1))k\log k$ variables.  We thus completely describe which diagonal forms are partition regular and which are not, given sufficiently many variables.  In addition, we prove a supersaturated version of Rado's theorem for a linear equation restricted either to squares minus one or to logarithmically-smooth numbers.
\end{abstract}

\maketitle

\setcounter{tocdepth}{1}
\tableofcontents

\section{Introduction}\label{intro}

Schur's theorem \cite{schur} is a foundational result in Ramsey theory, asserting that in any finite colouring of the positive integers there exists a monochromatic solution to the equation $x+y = z$ (a solution in which each variable receives the same colour).  A notorious question of Erd\H{o}s and Graham 
asks if the same is true for the Pythagorean equation $x^2 + y^2 = z^2$, offering \$250 for an answer \cite{graham07, graham}. The computer-aided verification \cite{computer} of the two colour case of this problem is reported to be the largest mathematical proof in existence, consuming 200 terabytes \cite{nature}.  
We provide an affirmative answer to the analogue of the Erd\H{o}s--Graham question for generalised Pythagorean equations in five or more variables.

\begin{theorem}[Schur-type theorem in the squares]\label{infinitary square schur}
In any finite colouring of the positive integers there exists a monochromatic solution to the equation
\begin{equation}\label{generalised pythag}
x_1^2 + x_2^2 + x_3^2 + x_4^2 = x_5^2.
\end{equation}
\end{theorem}  
This is a consequence of a more general phenomenon. Given enough variables, we completely describe which diagonal forms have the above property and which do not.
\begin{definition}[Partition regular]  Given a polynomial $P \in \Z[x_1, \dots, x_s]$ and a set $S$ call the equation $P(x) = 0$ \emph{partition regular over $S$} if, in any finite colouring of $S$, there exists a solution $x  \in S^s$ whose coordinates all receive the same colour.  We say that the equation is \emph{non-trivially partition regular} if every finite colouring of $S$ has a monochromatic solution in which each variable is distinct.
\end{definition}

Rado \cite{rado} established an elegant algebraic characterisation of partition regular homogeneous linear equations.
\newtheorem*{rado}{Rado's criterion for one equation}
\begin{rado}
Let $c_1, \dots, c_s \in \Z\setminus\set{0}$, where $s \ge 3$. Then the equation $\sum_{i=1}^s c_i x_i = 0$ is (non-trivially) partition regular over the positive integers if and only if there exists a non-empty set $I\subset [s]$ such that $\sum_{i\in I} c_i = 0$.
\end{rado}

A number of authors \cite{bergelson, bergyoutube, graham, DiNasso} have sought algebraic characterisations of partition regularity within families of non-linear Diophantine equations.  The example of the Fermat equation shows that one cannot hope for something as simple as Rado's criterion for diagonal forms.  Nevertheless, provided that the number of variables $s$ is sufficiently large in terms of the degree $k$, we establish that the same criterion characterises partition regularity for equations in $k$th powers. 
 
\begin{theorem}[Rado over $k$th powers]\label{intro main theorem}
There exists $s_0(k) \in \bN$ such that for $s \geq s_0(k)$ and $c_1, \dots, c_s \in \Z\setminus\set{0}$ the following holds. The equation 
\begin{equation}\label{kth power rado eqn}
\sum_{i=1}^s c_i x_i^k = 0
\end{equation} is (non-trivially) partition regular over the positive integers if and only if there exists a non-empty set $I\subset [s]$ such that $\sum_{i\in I} c_i = 0$.  Moreover, we may take $s_0(2) = 5$, $s_0(3) = 8$ and 
\begin{equation}\label{s0 size}
s_0(k) = k\brac{\log k + \log\log k + 2 + O(\log\log k/ \log k)}.
\end{equation}
\end{theorem}

Notice that Rado's criterion for a linear equation shows that the condition $\sum_{i\in I}c_i = 0$ is necessary for \eqref{kth power rado eqn} to be partition regular.  The content of Theorem \ref{intro main theorem} is that this condition is also sufficient.

For higher-degree equations one cannot avoid the assumption of some lower bound on the number of variables, as the example of the Fermat equation demonstrates.  Given current knowledge on the solubility of diagonal Diophantine equations \cite{Woo1992}, the bound \eqref{s0 size} is at the cutting edge of present technology.  Indeed, it is unlikely that one could improve this condition without making an analogous breakthrough in Waring's problem, since partition regularity implies the existence of a non-trivial integer solution to the equation \eqref{kth power rado eqn}.  

We remark that one could use the methods of this paper to establish the weaker but explicit bound 
$$
s_0(k) \leq k^2 +1.
$$
This follows by utilising the work of Bourgain--Demeter--Guth \cite{bdg} on Vinogradov's mean value theorem, eschewing smooth numbers, as in \cite{chow}.

We are also able to establish the sufficiency of Rado's criterion for other sparse arithmetic sets of interest, such as logarithmically-smooth numbers and shifted squares.  For these sets we  avoid certain local issues which must be surmounted for perfect powers, and thereby prove  stronger quantitative variants of partition regularity, analogous to work of Frankl, Graham and R\"odl \cite{FGR} counting monochromatic solutions to a linear equation.

\begin{theorem}[Supersaturation\footnote{The term `supersaturation', from extremal combinatorics, describes when we wish to ``determine the minimum number of copies of a particular
substructure in a combinatorial object of prescribed size'' \cite{NSS18}. For us, the substructure is defined by a Diophantine equation.} in squares minus one]\label{shifted squares}
Let $c_1, \dots, c_s \in \Z\setminus \set{0}$ with $s \geq 5$ and suppose that $\sum_{i \in I} c_i = 0$ for some non-empty $I$.  Define the set of shifted squares by
$$
S:=\set{x^2 - 1 : x \in \Z}.
$$
For any $r \in \N$ there exist $c_0 > 0$ and $N_0 \in \N$ such that for any $N \geq N_0$ if we have an $r$-colouring of $S$ then
\begin{multline}\label{shifted square count}
\hash\Bigset{x \in (S\cap [N])^s : \sum_i c_i x_i = 0 \text{ and x is monochromatic}} \\ \geq c_0|S\cap[N]|^s N^{-1}.
\end{multline}
\end{theorem}

\begin{remark} 
For the set of squares minus one, the upper bound
$$
\hash\Bigset{x \in (S\cap [N])^s : \sum_i c_i x_i = 0 } \ll |S\cap[N]|^s N^{-1}
$$
follows from an application of the Hardy--Littlewood circle method \cite{vaughan97}. Hence, the number of monochromatic solutions is within a constant (depending only on the number of colours) of the maximum possible.
\end{remark}

We prove Theorem \ref{shifted squares} in Part \ref{super sat part} together with an analogous result for logarithmically-smooth numbers.   
\begin{definition}[$R$-smooth numbers]\label{smooth def}
A number is \emph{$R$-smooth} if all of its prime factors are at most $R$.  Denote the set of $R$-smooths in $[N]$ by 
\begin{equation*}
S(N; R):=\set{x \in [N]: p \mid x \implies p \leq R}.
\end{equation*}
\end{definition}
When $R$ is logarithmic in $N$, of the form $R = \log^K N$, then 
$$
|S(N; \log^K N)| \sim N^{1 - K^{-1}+ o(1)} \qquad (N \to \infty),
$$
so logarithmically-smooth numbers constitute a polynomially sparse arithmetic set \cite{Granville}.

A recent breakthrough of Harper \cite{harper} gives a count of the number of solutions to an additive equation in logarithmically-smooth numbers.  We are able to extend this count to finite colourings as follows. 
 \begin{theorem}[Supersaturation in the smooths]\label{super smooths}
Let $c_1, \dots, c_s \in \Z\setminus \set{0}$, and suppose that $\sum_{i \in I} c_i = 0$ for some non-empty $I$. Then for any $r \in \N$ there exist $c_0 > 0$ and $C, N_0 \in \N$ such that if $N \geq N_0$, $R \geq \log^C N$ and $S(N; R)$ is $r$-coloured then
\begin{multline}\label{smooth count}
\hash\Bigset{x \in S(N; R)^s : \sum_i c_i x_i = 0 \text{ and x is monochromatic}}\\ \geq c_0|S(N; R)|^s N^{-1}.
\end{multline}
\end{theorem}
As for shifted squares, we emphasise that the corresponding upper bound in \eqref{smooth count} follows (when $s \geq 3$) from the methods of Harper \cite{harper}.  

\subsection{Non-triviality}  It may be that \eqref{kth power rado eqn} possesses a wealth of monochromatic solutions for `trivial' reasons.  For instance, if $c_1 + \dots + c_s = 0$ then taking $x_1 = \dots = x_s$ yields many uninteresting solutions.  We have delineated between partition regularity and non-trivial partition regularity to ensure that Rado's criterion still has content in such a situation.  However, since Rado's criterion is necessary for `trivial' partition regularity, the two notions are in fact equivalent.

\subsection{Previous work}

To the knowledge of the authors, work on non-linear partition regularity begins with papers of Furstenberg and S\'ark\"ozy \cite{furstenberg, sarkozy}, independently resolving a conjecture of Lov\'asz---a line of investigation which culminates in the polynomial Szemer\'edi theorem of Bergelson--Leibman \cite{BergelsonLeibman}, proved using ergodic methods.  Such methods have also established colouring results for which no density analogue exists, such as partition regularity of the equation $x - y = z^2$ \cite[p.53]{bergelson}.  Interestingly, the story is more complicated for the superficially similar equation $x+y = z^2$ studied in \cite{ks06, cgs, greenlindqvist, pach}.

A recent breakthrough of Moreira \cite{moreira} resolves a longstanding conjecture of Hindman \cite{hindman}, proving partition regularity of  the equation $x + y^2 = yz$.  More intuitively: in any finite colouring of the positive integers there exists a monochromatic configuration of the form $\set{a, a+b, ab}$.  This result is a consequence of a general theorem which also yields partition regularity of equations of the form $x_0 = c_1 x_1^2 + \dots + c_s x_s^2$, subject to the condition that $c_1 + \dots + c_s = 0$.

Notice that all of the above results involve an equation with at least one linear term.  There are fewer results in the literature concerning genuinely non-linear equations such as \eqref{kth power rado eqn}.  Certain diagonal quadrics are dealt with in Lefmann \cite[Fact 2.8]{lefmann}, using Rado's theorem to locate a long monochromatic progression whose common difference possesses a (well-chosen) multiple of the same colour.  This results in the following sufficient condition for partition regularity.

\begin{lef}
Let $c_1, \dots, c_s \in \Z\setminus\set{0}$, and suppose that $\sum_{i \in I} c_i = 0$ with $I \neq \emptyset$.  Moreover, suppose that the auxiliary system 
\begin{equation}\label{lefmann aux}
\begin{split}
\Bigbrac{\sum_{i \notin I} c_i}x_0^2 +\sum_{i \in I} c_i x_i^2   &= 0,\\
\sum_{i \in I} c_i x_i & = 0
\end{split}
\end{equation}
possesses a rational solution with $x_0 \ne 0$. Then the equation
\begin{equation}\label{diag quadric}
c_1 x_1^2 + \dots + c_sx_s^2 = 0
\end{equation}
is partition regular.
\end{lef}
This result reduces the combinatorial problem of establishing partition regularity of \eqref{diag quadric} to a task in number theory: find a rational point of a certain form on a variety determined by a diagonal quadric and linear equation.  In Appendix \ref{lefmann} we derive general algebraic criteria guaranteeing such a rational point using the Hardy--Littlewood circle method.
\begin{theorem}[Lefmann + Hardy--Littlewood circle method]\label{lefmann theorem}
Let $c_1, \dots, c_s \in \Z\setminus\set{0}$, and suppose that $\sum_{i \in I} c_i = 0$ with $I \neq \emptyset$.  Suppose in addition that $|I| \geq 6$ and at least two $c_i$ are positive and at least two are negative. Then
\begin{equation}\label{gensq}
c_1 x_1^2 + \dots + c_sx_s^2 = 0
\end{equation}
is partition regular.
\end{theorem}
This result does not encompass all equations amenable to Lefmann's criterion: fewer variables may suffice, for instance
$$
x^2 + 9y^2 = 2z^2 + 8w^2\qquad \text{or} \qquad 4x^2 + y^2 = 2z^2 + 2w^2.
$$
We emphasise that Lefmann's criterion cannot hope to be a necessary condition for partition regularity, as there are partition regular equations for which the auxiliary Lefmann system \eqref{lefmann aux} has no rational point of the required form. Such equations include the generalised Pythagorean equation \eqref{generalised pythag}, as well as the `convex' equation
\begin{equation}\label{convex equation}
x_1^2 + x_2^2 + x_3^2 + x_4^2 = 4 x_5^2
\end{equation} 
addressed in \cite{densesquares}. 

In the same article, Lefmann \cite[Theorem 2.6]{lefmann} established Rado's criterion for reciprocals.

\begin{theorem} [Lefmann] 
Let $c_1,\ldots, c_s \in \bZ \setminus \{ 0\}$. Then
\[
\sum_{i=1}^s c_i x_i^{-1} = 0
\]
is partition regular over $\bN$ if and only if $\sum_{i \in I} c_i = 0$ for some non-empty $I \subset [s]$.
\end{theorem}

\noindent This demonstrates the partition regularity of 
\[
\frac1x + \frac1y = \frac1z,
\]
answering a question of Erd\H{o}s and Graham.

If one is prepared to relax the definition of partition regularity, so that certain variables are not constrained to receive the same colour as the remainder, then specific homogeneous 
 equations of arbitrary degree are dealt with in Frantzikinakis--Host \cite{FranHost}.  For instance, one consequence of their methods is that in any finite colouring of the positive integers there exist distinct $x,y$ of the same colour, along with $\lambda$ (possibly of a different colour) such that
 \begin{equation}\label{fran host}
 9x^2 + 16y^2 = \lambda^2.
 \end{equation}
 However for these techniques to succeed, not only must one variable of \eqref{fran host} be free to take on any colour, but it is also necessary for the solution set to possess a well-factorable parametrisation, allowing for the theory of multiplicative functions to come into play.

When the coefficients of \eqref{kth power rado eqn} sum to zero, partition regularity follows easily, since any element of the diagonal constitutes a monochromatic solution.  However, there are results in the literature which also guarantee \emph{non-trivial} partition regularity in this situation, provided that $s \geq k^2 +1$.  This was first established for quadrics in 
\cite{densesquares} and for general $k$ in 
\cite{chow}.  In fact in 
\cite{chow} 
it is established that, under these assumptions, dense subsets of the \emph{primes} contain many solutions to \eqref{kth power rado eqn}.  Density results were obtained for non-diagonal quadratic forms in at least 9 variables by Zhao \cite{zhao}, subject to the condition that the corresponding matrix has columns which sum to zero. 

We believe that when the solution set of a given equation contains the diagonal it is more robust with respect to certain local issues---indeed one expects dense sets (such as congruence classes) to contain solutions under this assumption.  As a consequence, the local issues for such equations are easier to handle using elementary devices, such as passing to a well-chosen subprogression.  The novelty in our methods is that for general equations, instead of tackling the somewhat thorny local problem head on, we show how we may assume our  colouring possesses a certain homogeneous structure, and this structure allows the same devices available in the dense regime to come into play.

We remark that it appears to be a challenging problem to decrease $s_0(k)$ substantially below $k^2 +1$ for the density analogue of Theorem \ref{intro main theorem}.  In order to show that $s_0(k) = (1+o(1))k\log k$ is admissible in our partition result we make heavy use of the fact that a colouring of the positive integers induces a colouring of the smooth positive integers, and we obtain a monochromatic solution to our equation in the smooths.  Sets of positive density, however, may not contain any smooth numbers.  We are therefore in the curious situation where we can prove that relatively dense sets of smooth numbers possess solutions to certain diagonal equations, but cannot say the same for dense sets of integers.

It is interesting to compare our results with partition regularity results over the primes. Here congruence obstructions mean that one cannot hope to establish a Rado-type criterion.  For example, a parity obstruction prohibits Schur's equation from being partition regular over the primes. The situation is markedly different if one considers modifications of the primes with no local obstructions, such as the set of primes minus one.  Partition regularity of the Schur equation over this set was established by Li--Pan \cite{LiPan}, then generalised to the full Rado criterion for systems of linear equations by L\^e \cite{le}.  This latter result utilised the full strength of Green and Tao's asymptotic for linear equations in primes \cite{LinearEquationsPrimes}, together with a characterisation of so called `large' sets due to Deuber \cite{deuber}.  Neither of these tools are available, or reasonable to expect, for $k$th powers.

The argument of Li--Pan for Schur's theorem in primes minus one is a direct application of the Fourier-analytic transference principle pioneered by Green \cite{greenprimes}, elucidated by the same author in the context of partition regularity in a comment\footnote{\url{https://goo.gl/Yjookp}} on MathOverflow.  This approach cannot hope to succeed for perfect powers, at least when the coefficients of the equation do not sum to zero, since one can no longer pass to the same (affine) subprogression in all of the variables.  The introduction of homogeneous sets (Definition \ref{homogeneous definition}) allows us to circumvent these difficulties.  However, for squares minus one, or smooth numbers, one need only pass to projective subprogressions when enacting the transference principle.  The methods of Part \ref{super sat part} therefore use a direct form of the transference principle analogous to Li--Pan.  We include the argument to illustrate the subtleties which must be overcome for perfect powers.

\subsection{Notation}\label{absolute constants}
We adopt the convention that $\eps$ denotes an arbitrarily small positive real number, so its value may differ between instances. We shall use Vinogradov and Bachmann--Landau notation: for functions $f$ and positive-valued functions $g$, write $f \ll g$ or $f = O(g)$ if there exists a constant $C$ such that $|f(x)| \le C g(x)$ for all $x$. At times we opt for a more explicit approach, using $C$ to denote a large absolute constant (whose value may change from line to line), and $c$ to denote a small positive absolute constant. The notation $f \asymp g$ is the same as $f\ll g \ll f$. For $Y \ge 1$, let $[Y] = \{ 1,2, \ldots, \lfloor Y \rfloor \}$. We write $\T$ for the torus $\R / \Z$.  For $x \in \R$ and $q \in \N$, put \mbox{$e(x) = e^{2 \pi i x}$} and $e_q(x) = e^{2 \pi i x / q}$.  If $S$ is a set, we denote the cardinality of $S$ by $|S|$ or $\# S$. 

Throughout we use counting measure on $\Z^d$ and Haar probability measure on the dual $\T^d := \R^d/\Z^d$.  So if $f,g :\Z^d \to \C$ have finite support then
$$
\norm{f}_p :=\begin{cases} \brac{ \sum_x |f(x)|^p}^{1/p} & \text{if } p < \infty\\
						\mathrm{max}_x |f(x)| & \text{if } p = \infty.\end{cases}
$$
Define the Fourier transform of $f $ by
$$
\hat{f}(\alpha) := \sum_x f(x) e(\alpha \cdot x).
$$
We endow $\T^d$ with the metric $(\alp, \bet) \mapsto \norm{\alpha - \beta}$, where
$$
\norm{\alpha} := \sum_{i=1}^d \min_{n \in \Z} |\alpha_i - n|.
$$

\subsection*{Funding}

SC was supported by EPSRC Programme Grant \mbox{\texttt{EP/J018260/1}.} SL is supported by Ben Green's Simons Investigator Grant, number \texttt{376201}. A portion of this project was completed whilst the authors were in residence in Berkeley, California during the Spring 2017 semester: SC and SL at the Mathematical Sciences Research Institute, and SP at the Simons Institute for the Theory of Computing.  During this time SC was supported by the National Science Foundation under Grant No. \texttt{DMS-1440140}.

\subsection*{Acknowledgements} 

The authors have benefited greatly from Trevor Wooley's knowledge of discrete restriction theory over smooth numbers. Thanks to Ben Barber for enquiring about Rado's criterion over squares during a seminar at the University of Bristol, and for recording this question on MathOverflow\footnote{\url{https://goo.gl/Yjookp}}.  Thanks to Xuancheng Shao for the idea of colouring smooth numbers.  We are indebted to Jonathan Chapman for corrections, clarifications and numerous useful comments on an earlier version of this manuscript.

\section{Methods}\label{intro methods}
All of the essential ideas required for Theorem \ref{intro main theorem} are contained in the proof of the following finitary analogue of Theorem \ref{infinitary square schur}, whose deduction is the focus of this section.

\begin{theorem}[Finitary Schur-type theorem in the squares]\label{finitary main theorem}
For any $r \in \N$ there exists $N_0 = N_0(r)$ such that for any $N \geq N_0$ the following is true.  Given an $r$-colouring of $[N]$ there exists a monochromatic solution to the equation
$
x_1^2 - x_2^2 = x_3^2 + x_4^2 + x_5^2.
$
\end{theorem}

Inspired by work of Cwalina--Schoen \cite{CwalinaSchoen} and Green--Sanders \cite{greensanders}, we derive Theorem \ref{finitary main theorem} in \S \ref{colouring section} by an induction on the number of colours, in combination with a density result concerning what we have termed \emph{homogeneous sets}. 
\begin{definition}[Homogeneous set]\label{homogeneous definition}  Call a set $B$ of positive integers \emph{$M$-homogeneous} if for any $q\in \N$  we have
\begin{equation}\label{homogeneous inequality}
\begin{split}
B\cap q \cdot [M] \neq \emptyset.
\end{split}
\end{equation}
Given a set $S \subset \N$, we say that $B$ is $M$-\emph{homogeneous in $S$} if \eqref{homogeneous inequality} holds for all homogeneous progressions $q \cdot [M]$ contained in $S$.  Notice that the latter does not require that $B \subset S$.
\end{definition}
Chapman \cite{chapman} has observed that this is a quantitative variant of what it means to be \emph{multiplicatively syndetic} (see Bergelson--Glasscock \cite{glasscock}), and that such sets appear to have a number of interesting properties in regard to the partition regularity of homogeneous systems of polynomial equations.

We leave it as an exercise for the reader to verify that if $B$ is an $M$-homogeneous set then $|B\cap [N]| \gg_{M} N$ for $N$ sufficiently large in terms of $M$, so homogeneous sets are dense (see Lemma \ref{homogeneous density lemma}).  In fact they are dense on all sufficiently long homogeneous arithmetic progressions.

We demonstrate the utility of this definition by giving a proof of Schur's theorem.  The argument is prototypical for that employed in the proof of Theorem \ref{finitary main theorem}.
 
\begin{proof}[Proof of Schur's theorem]
We induct on the number of colours $r$ to show that there exists $N_r \in \bN$ such that however $[N_r]$ is $r$-coloured there exist $x_1, x_2, x_3 \in [N_r]$ all of the same colour with $x_1+x_2 = x_3$.

The base case of 1-colourings follows on taking $N_1 = 2$, so we may assume that $r \geq 2$.  Let $N$ be a large positive integer, whose size (depending on $r$) is to be determined, and fix an $r$-colouring
$$
[N] = C_1 \cup \dots \cup C_r.
$$
Set $M := N_{r-1}$ and consider two possibilities.

\noindent\textbf{The inhomogeneous case: Some colour class $C_i$ is not $M$-homogeneous in $[N]$.}  From the definition of homogeneity it follows that there exists a positive integer $q$ such that $q\cdot [M]\subset [N]$ and $q \cdot [M] \cap C_i = \emptyset$.  On setting $C_j' := \set{x \in [M] : qx \in C_j}$ we induce an $(r-1)$-colouring
$$
[M] = \bigcup_{j \neq i} C_j'.
$$
Since $M = N_{r-1}$ it follows from our induction hypothesis that there exist $x_1', x_2', x_3' \in C_j'$ such that $x_1' + x_2' = x_3'$.  Schur's theorem follows in this case on setting $x_t := qx_t'$ for $t =1, 2, 3$.

\noindent\textbf{The homogeneous case: All colour classes are $M$-homogeneous in $[N]$.}  In this case it turns out that every colour class contains a solution to the Schur equation, provided that $N$ is sufficiently large in terms of $r$.  To prove this we invoke the following.\begin{claim}
For any $\delta > 0$ and $M \in \N$ there exists $N_0 = N_0(\delta, M)$ such that for any $N \geq N_0$ if $A \subset [N]$ has $|A| \geq \delta N$ and $B\subset [N]$ is $M$-homogeneous in $[N]$ then there exist $x, x' \in A$ and $y \in B$ such that $x - x' = y$.
\end{claim}
The claim settles the homogeneous case of Schur's theorem on taking $A =B$ to be any colour class, since $M$-homogeneous sets have density at least $M^{-2} + o(1)$ in $[N]$ (see Lemma \ref{homogeneous density lemma};  one could have alternatively taken the largest colour class).  

To prove the claim we invoke Szemer\'edi's theorem!\footnote{The claim itself is not deep, for instance it is readily obtained from \cite[Theorem 4]{CRS07}. Our proof is designed to set the stage for the general setting of Part \ref{higher part}, when we will invoke the multidimensional polynomial Szemer\'edi theorem of Bergelson and Leibman \cite{BergelsonLeibman}.} This yields $N_0 = N_0(\delta, M)$ such that for any $N \geq N_0$ if $A \subset [N]$ with $|A| \geq \delta N$ then $A$ contains an arithmetic progression of length $M+1$, so that there exist $x$ and $q > 0$ for which
$$
x,\quad x+q,\quad x + 2q, \quad \dots,\quad x + Mq\quad \in  A.
$$
Notice that $q \cdot [M]  \subset [N]$, so $M$-homogeneity of $B$ implies that there exists $y \in q\cdot [M] \cap B$.  Taking $x' = x + y$ establishes the claim and completes our proof of Schur's theorem.
\end{proof}

It may seem excessive to employ a density result in the proof of a colouring result, since (typically) density results lie deeper and require more work to prove.\footnote{One can give an alternative argument for Schur's theorem based on these ideas, replacing Szemer\'edi's theorem with van der Waerden's.  However, this approach does not seem to generalise to the non-linear situation.} We have described this approach to motivate our proof of Theorem \ref{infinitary square schur}, which uses an analogous non-linear density result.  We also believe the proof offers an alternative reason for \emph{why} Schur's theorem is true: there is always a long homogeneous arithmetic progression on which one of the colour classes is multiplicatively syndetic. This exemplifies a well-used philosophy in Ramsey theory that underlying every partition result there is some notion of largeness.


To prove partition regularity of the generalised Pythagorean equation we induct on the number of colours  as in our proof of Schur's theorem.  The inhomogeneous case follows with minimal change to the argument.  In the remaining case we may assume that all colour classes are homogeneous.  In this situation we are able to show that every colour class contains many solutions to our non-linear equation by employing the following density result.
 \begin{theorem}[Non-linear homogeneous S\'ark\"ozy]\label{non linear sarkozy}
For any $\delta > 0$ and $M \in \N$ there exist $N_0$ and $c_0>0$ such that for any $N \geq N_0$ the following holds.  Let $A \subset [N]$ have density at least $\delta$ in $[N]$, and let $B$ be an $M$-homogeneous subset of the positive integers.  Then 
\begin{equation*}
\begin{split}
\hash\set{(\bx, \by) \in A^2 \times B^3:x_1^2 - x_2^2 = y_1^2 + y_2^2 +y_3^2} \geq c_0 N^{3}.
\end{split}
\end{equation*}
\end{theorem}

Using Green's Fourier-analytic transference principle \cite{greenprimes}, as elucidated for squares in \cite{densesquares, FourVariants}, the deduction of Theorem \ref{non linear sarkozy} is reduced (in \S\S\ref{pseud sark}--\ref{square W trick})  to a linear analogue in which the squares have been removed from the dense variables. This can be thought of as a generalisation of the Furstenberg--S\'ark\"ozy theorem \cite{furstenberg, sarkozy}, extended to homogeneous sets. 

\begin{theorem}[Supersaturated homogeneous S\'ark\"ozy]\label{super saturated sarkozy}
For any $\delta > 0$ and $M \in \N$ there exist $N_0, c_0>0$ such that for any $N \geq N_0$ the following holds.  Let $A\subset[N]$ have density at least $\delta$ in $[N]$ and let $B$ be an $M$-homogeneous subset of the positive integers.  Then there are at least $c_0 N^{\frac{5}{2}}$ tuples $(x, y) \in A^2 \times B^3$ satisfying the equation
\begin{equation}\label{linear sarkozy equation}
\begin{split}
x_1 - x_2 = y_1^2 + y_2^2 + y_3^2.
\end{split}
\end{equation}
\end{theorem}

Our ability to remove the squares from the dense variables is intrinsically linked to the fact that the coefficients corresponding to these variables sum to zero.  One consequence of this is that we may restrict all of the dense variables to lie in the same congruence class, without destroying solutions to the equation in the process.

Theorem \ref{super saturated sarkozy} is ultimately derived (in \S \ref{square varnavides}) from the following result, which is both more general and at the same time slightly weaker than Theorem \ref{super saturated sarkozy}.  It is weaker in that it yields only one solution to \eqref{linear sarkozy equation}, yet it applies to the more general context of multidimensional sets of integers.  The increase in dimension allows us to deduce a supersaturation result for \eqref{linear sarkozy equation} by bootstrapping the existence of  a single solution to the existence of many solutions, using an averaging argument first implemented by Varnavides \cite{varnavides}. 
\begin{theorem}[Multidimensional homogeneous S\'ark\"ozy]\label{multidimensional sarkozy}
For any $\delta > 0$ and $d, M \in \N$ there exists $N_0$ such that for any $N \geq N_0$ the following holds.
If $A \subset [N]^d$ is at least $\delta$-dense in $[N]^d$ and $B_1, \dots, B_d$ are $M$-homogeneous sets of positive integers, then there exist $x, x' \in A$ and $y_1 \in B_1, \dots, y_d \in B_d$ such that
\begin{equation}\label{2nd multidim eqn}
x - x' = (y_1^2, \dots, y_d^2).
\end{equation}
\end{theorem}
In \S\ref{square multidim section} this theorem is proved using the Fourier-analytic density increment strategy pioneered by Roth \cite{roth}, a proof which yields quantitative bounds on $N_0$.  One can deduce the qualitative statement in a few lines from the multidimensional polynomial Szemer\'edi theorem of Bergelson and Leibman \cite{BergelsonLeibman}, see Corollary \ref{smooth multidimensional sarkozy generalisation}.  The general Rado criterion of Theorem \ref{intro main theorem} requires a more complicated density result for which Fourier analysis does not appear sufficient and which therefore necessitates the invocation of this deep result.

\section{Open problems}

\subsection{The supersaturation result}

Frankl, Graham and R\"odl \cite{FGR} establish that for any $r$-colouring of $[N]$, a linear equation $\sum_{i=1}^s c_i x_i= 0$ satisfying Rado's criterion has $\gg_r N^{s-1}$ monochromatic solutions.  Our methods do not yield the analogous supersaturation result for equation \eqref{kth power rado eqn}.  We instead find that if $N$ is sufficiently large in terms of $M$ then $[N]$ contains a homogeneous arithmetic progression of length $M$ which possesses at least $\gg_r M^{s-k}$ monochromatic solutions to \eqref{kth power rado eqn}.  This deficiency is an artefact of our method where, to avoid tackling certain local issues, we iteratively pass to a well-chosen homogeneous subprogression.

It may be possible to establish a supersaturation result if one is prepared to replace the homogeneous arithmetic progressions appearing in this paper with quadratic Bohr sets.  Informally, let us call a set \emph{quadratic Bohr homogeneous} if it has large intersection with all quadratic Bohr sets (centred at zero).  Then our methods reduce to showing that if $A$ is a dense subset of a quadratic Bohr set and if $B$ is quadratic Bohr homogeneous, then there are many solutions to the equation
$$
x_1^2-  x_2^2  = y_1^2 + y_2^2 + y_3^2
$$
with $x_i \in A$ and $y_i \in B$.  A promising strategy for obtaining such a result proceeds by decomposing $1_A$ according to a variant of the arithmetic regularity lemma developed by Green and Tao \cite{greentaoregularity}.  It is in fact this strategy which informs the simpler approach developed in this paper.

\subsection{Quantitative bounds}

Define the \emph{Rado number} (see \cite[p.103]{ramseytheory}) of the equation \eqref{kth power rado eqn} to be the smallest positive integer $R_{\vc,k}(r)$ such that any $r$-colouring of the interval $\set{1, 2, \dots, R_{\vc,k}(r)}$ results in at least one monochromatic tuple $(x_1, \dots, x_s)$ satisfying \eqref{kth power rado eqn} with all $x_i$ distinct. For linear equations, this quantity has been extensively studied by Cwalina and Schoen \cite{CwalinaSchoen}, with near optimal bounds extracted for certain choices of coefficients. In \cite{densesquares} it is shown that when $k = 2$,  $c_1+\dots + c_s = 0$ and $s \geq 5$ then there exists a constant $C_{\vc}$ such that
\begin{equation}\label{rado bound}
R_{\vc, 2}(r) \leq \exp\exp\exp(C_{\vc} r).
\end{equation}

It is feasible that the methods of this paper lead to quantitative bounds for the Rado number of the equation \eqref{kth power rado eqn} provided that there exist coefficients with $c_i = -c_j$.  In this situation, all of the results we employ in our argument can be proved using Fourier-analytic methods, where the quantitative machinery is well-developed.  However, these bounds are sure to be of worse quality than \eqref{rado bound} due to our induction on the number of colours, a feature of the argument not present in \cite{densesquares}.

If there are no coefficients satisfying $c_i = -c_j$, then any hope of extracting quantitative bounds on $R_{\vc, k}(r)$ is diminished, since the methods of this paper invoke the multidimensional (polynomial) Szemer\'edi theorem, a result for which there are no quantitative bounds presently known.  It would be interesting if one could avoid calling on such a deep result.

\subsection{Systems of equations}

Rado \cite{rado} characterised when \emph{systems} of linear equations are partition regular.  This criterion says that a system $A x = 0$ is partition regular if and only if the integer matrix $A$ satisfies the so-called \emph{columns condition} (see \cite[p.73]{ramseytheory}).  We conjecture that the columns condition is sufficient for systems of equations in $k$th powers, provided that the number of variables is sufficiently large in terms of the degree and the number of equations, and that the matrix of coefficients is sufficiently generic.  For instance, in analogy with results of Cook \cite{cook} we posit the following.
\begin{conjecture}
Let $a_1, \ldots, a_s, b_1, \ldots, b_s \in \bZ \setminus \{0\}$. Then the system of equations
\begin{align*}
a_1 x_1^2 + \dots + a_s x_s^2 & = 0\\
b_1 x_1^2 + \dots + b_s x_s^2 & = 0
\end{align*}
is non-trivially partition regular, provided that 
\begin{itemize}
\item[(i)] $s \geq 9$;
\item[(ii)] the matrix $A:= \begin{pmatrix} a_1 & \dots & a_s\\ b_1 & \dots & b_s\end{pmatrix}$ satisfies the columns condition;
\item[(iii)] for any real numbers $\lambda, \mu$ that are not both zero, the vector $(\lambda, \mu) A$ has at least five non-zero entries, not all of which have the same sign.
\end{itemize}
\end{conjecture}
Condition (ii) is certainly necessary for partition regularity, by Rado's criterion.  Weakening conditions (i) and (iii) would presumably require improvements in circle method technology.

\subsection{Roth with logarithmically-smooth common difference}

Using the arguments of \S\ref{smooth multidim section} one can prove the following (see Remark \ref{smooth roth remark}).
\begin{theorem}\label{smooth roth theorem} If $A\subset [N]$ lacks a three-term arithmetic progression with $R$-smooth common difference, where $10 \le R \le N$, then 
\begin{equation}\label{smooth roth bound}
|A| \ll N\frac{(\log\log R)^4}{\log R}.
\end{equation}
\end{theorem}
When $R = \log^K N$ for some fixed absolute constant $K$, the set of $R$-smooth numbers in $[N]$ has cardinality $N^{1- K^{-1}+ o(1)}$.  Common differences arising from such a set are therefore polynomially sparse, and Theorem \ref{smooth roth theorem} results in a density bound of the form $(\log\log N)^{-1 + o(1)}$.

The argument for Theorem \ref{smooth roth theorem} really only uses the fact that the $R$-smooths contain the interval $[R]$, and that $A$ must be dense on a translate of this set, so we are in fact locating a `short' arithmetic progression.  Since smooth arithmetic progressions are much more abundant than short arithmetic progressions, it would be interesting if one could obtain a better density bound by exploiting this.

The only other bound known for Roth's theorem with common difference arising from a polynomially sparse arithmetic set can be found in \cite{quantitativeBL}, which deals with perfect $k$th powers.  This also results in a double logarithmic bound, of the form $(\log\log N)^{-c_k}$ for some small $c_k >0$.  Breaking the double logarithmic barrier for the smooth Roth problem may be a tractable intermediate step towards improving bounds in the polynomial Roth theorem.

\part{The generalised Pythagorean equation}\label{pythag part}
In this part we establish partition regularity of the 5-variable Pythagorean equation $x_1^2 + x_2^2 + x_3^2 + x_4^2 = x_5^2$.  The proof contains all of the essential ideas required for Theorem \ref{intro main theorem} but is more transparent, avoiding notational complexities and the need for smooth number technology.  Unlike the general case, we show that all requisite steps can be established using Fourier analysis, avoiding recourse to deeper results involving higher-order uniformity and the multidimensional Szemer\'edi theorem.  This may be of use to those interested in quantitative bounds and supersaturation.

Throughout this part we assume familiarity with the high-level schematic outlined in \S \ref{intro methods}.
\section{Induction on colours}\label{colouring section}

We first derive Theorem \ref{finitary main theorem} from Theorem \ref{non linear sarkozy} by induction on the number of colours.  We deduce Theorem \ref{non linear sarkozy} from Theorem \ref{super saturated sarkozy} in \S\S\ref{pseud sark}--\ref{square W trick}, and prove Theorem \ref{super saturated sarkozy} in \S\S\ref{square multidim section}--\ref{square varnavides}.

\subsection{The inductive base: one colour}

\begin{definition}[$T$ counting operator]  Given functions $f_1, \dots, f_s : \Z \to \C$ with finite support, define the counting operator
\begin{align*}
T(f_1, \dots, f_5) := \sum_{x_1^2 -x_2^2 = x_3^2 + x_4^2 + x_5^2}f_1(x_1) f_2(x_2) f_3(x_3) f_4(x_4) f_5(x_5).
\end{align*}
We write $T(f)$ for $T(f, f, \dots, f)$.
\end{definition}

By Theorem \ref{unrestricted lower bound}, there exist $N_1 \in \N$ and $c_1 > 0$ such that for $N \geq N_1$ we have
$$
T(1_{[N]}) \geq c_1 N^{3}.
$$
Since the latter quantity is positive, Theorem \ref{finitary main theorem} follows for 1-colourings (the base case of our induction).

\subsection{The inductive step}

Let $[N] = C_1\cup \dots \cup C_r$ be an $r$-colouring.  
We split our proof into two cases depending on the homogeneity of the $C_i$.
\subsubsection{The inhomogeneous case}  Let $M := N_0(r-1)$ be the quantity whose existence is guaranteed by our inductive hypothesis.  We first suppose that some $C_i$ is not $M$-homogeneous in $[N]$ (see Definition \ref{homogeneous definition}).  Consequently there exists $q \in \N$  such that
\begin{equation}\label{inhomogeneous assumption}
\begin{split}
q \cdot [M] \subset [N] \quad \text{and} \quad C_i \cap q \cdot [M]= \emptyset.
\end{split}
\end{equation}
For $j\neq i$ let us define
$$
C_j' := \set{x \in [M] : qx \in C_j}.
$$
Then it follows from \eqref{inhomogeneous assumption} that $\bigcup_{j \neq i} C_j'  = [M]$.  By the induction hypothesis, there exist $y_k \in C_j'$ for some $j \neq i$ such that $y_1^2 - y_2^2 = y_3^2 + y_4^2 + y_5^2$.  Setting $x_k := qy_k$ we obtain elements of $C_j$ which solve the generalised Pythagorean equation.
\subsubsection{The homogeneous case}
In this case every colour class is $M$-homogeneous in $[N]$.  We claim that Theorem \ref{non linear sarkozy} then implies that each $C_i$ contains a solution to the generalised Pythagorean equation.  First we observe that each colour class is dense.
\begin{lemma}[Homogeneous sets are dense]\label{homogeneous density lemma}
If $B \subset [N]$ is $M$-homogeneous in $[N]$ then
$$
|B| \geq \recip{M}\floor{\frac{N}{M}}.
$$
\end{lemma}

\begin{proof}
We proceed by a variant of Varnavides averaging \cite{varnavides}.  For each $q \le N/M$ the definition of homogeneity gives
$$
B \cap  q \cdot [M] \neq \emptyset.
$$
Summing over $q$ then yields
$$
\sum_{q \le N/M} |B\cap q \cdot [M]| \geq \floor{N/M}.
$$
Interchanging the order of summation, we see that
\begin{align*}
\sum_{x \in B} \hash\set{(q, m) \in [N/M]\times [M] : x = qm} \geq \floor{N/M}.
\end{align*}
The result follows on noting that 
$$
\hash\set{(q, m) \in [N/M]\times [M] : x = qm} \le M.
$$
\end{proof}
Setting $A = B =  C_i$ in Theorem \ref{non linear sarkozy}  
we deduce that if $N \geq N_0(M)$ then
$$
T(1_{C_r})\geq c_0(M) N^{s-k}.
$$
Since the latter quantity is positive the induction step follows, completing the proof of Theorem \ref{finitary main theorem}.  Note that a quantity dependent on $M = N_0(r-1)$ is ultimately dependent only on $r$.

\section{A pseudorandom Furstenberg--S\'ark\"ozy theorem}\label{pseud sark}

In \S\ref{colouring section} we reduced partition regularity of the generalised Pythagorean equation \eqref{generalised pythag} to Theorem \ref{non linear sarkozy}.  In \S\ref{square W trick} we deduce the latter result from Theorem \ref{super saturated sarkozy}.  To prepare the ground for this deduction, we first modify Theorem \ref{super saturated sarkozy} to accommodate sets which are relatively dense in a suitably pseudorandom set.  The goal is to find the weakest possible pseudorandomness conditions required for such a result to hold.  Our primary quantity of interest is the following.
\begin{definition}[$T_1$ counting operator]  Given functions $f_1, f_2 : \Z \to \C$ with finite support and $B \subset \Z$, define 
\begin{align*}
T_1(f_1, f_2; B) := \sum_{\substack{x_1 - x_2 = y_1^2 + y_2^2 + y_3^2}}f_1(x_1) f_2(x_2)1_B(y_1) 1_B(y_2) 1_B(y_3).
\end{align*}
We write $T_1(f; B)$ for $T_1(f, f; B)$ and $T_1(A;B)$ for $T_1(1_A; B)$.
\end{definition}

We begin by showing how Theorem \ref{super saturated sarkozy} implies a result in which the indicator function $1_{A}$ can be replaced by a function $f : [N] \to [0,1]$ with sufficiently large average.
\begin{lemma}[Functional S\'ark\"ozy]\label{functional sarkozy}
For any $\delta > 0$ and $M \in \N$ there exists $N_0 \in \N$ and $c_0 > 0$ such that for any $N \geq N_0$ the following holds.  Let $f: [N] \to [0, 1]$ with $\norm{f}_1 \geq \delta N$ and let $B$ be an $M$-homogeneous subset of the positive integers.  Then 
\begin{equation*}
\begin{split}
T_1(f; B) \geq c_0 N^{\frac{5}{2}}.
\end{split}
\end{equation*}
\end{lemma}

\begin{proof} Let $A = \{ x \in [N]: f(x) \ge \del/2 \}$. As $\|f\|_1 \ge \del N$ and $f \leq 1$, we have $|A| \ge \del N / 2$. Since $f \ge \del 1_A / 2$, we deduce that
\[
T_1(f;B) \ge (\del/2)^2 T_1(A;B),
\]
and an application of Theorem \ref{super saturated sarkozy} completes the proof.
\end{proof}

Our next step is to weaken the assumptions of Theorem \ref{super saturated sarkozy} even further, replacing bounded functions with unbounded functions which are sufficiently pseudorandom.  The pseudorandomness we enforce posits the existence of a `random-like' majorising function $\nu$, whose properties are given in the following two definitions.
\begin{definition}[Fourier decay]\label{decay def} We say that $\nu : [N] \to [0, \infty)$ has \emph{Fourier decay of level $\theta$} (with respect to $1_{[N]}$) if
$$
\biggnorm{\frac{\hat{\nu}}{\norm{\nu}_1} - \frac{\hat{1}_{[N]}}{\norm{1_{[N]}}_1}}_\infty \le \theta .
$$
\end{definition} 
\begin{definition}[$p$-restriction]\label{restriction def} We say that $\nu : [N] \to [0, \infty)$ satisfies a \emph{$p$-restriction estimate with constant $K$} if 
$$
\sup_{|\phi| \le \nu} \int_{\T} \abs{\hat{\phi}(\alpha)}^p \intd \alpha \le K \norm{\nu}_1^p N^{-1}.
$$
\end{definition}

\begin{theorem}[Pseudorandom S\'ark\"ozy] \label{pseudorandom}
For any $\delta> 0$ and $K, M \in \N$ there exist $N_0, c_0, \theta > 0$ such that for any $N \geq N_0$ the following holds.  Let $B$ be an $M$-homogeneous set of positive integers. Let $\nu: [N] \to [0, \infty)$ satisfy a $4.995$-restriction estimate with constant $K$, and have Fourier decay of level $\theta$.  Then for any $f: [N] \to [0, \infty)$ with $f \le \nu$ and $\norm{f}_1 \geq \delta \norm{\nu}_1$ we have
\begin{equation*}
\begin{split}
T_1(f;B) \geq c_0\norm{\nu}^2_1 N^{\frac{1}{2}}. 
\end{split}
\end{equation*}
\end{theorem}

\begin{proof}
Since $\nu$ has Fourier decay of level $\theta$, we may apply the dense model lemma recorded in \cite[Theorem 5.1]{FourVariants}, rescaling as appropriate, to conclude the existence of $g: \bZ \to \bC$ satisfying $0 \le g \le 1_{[N]}$ and
\begin{equation}\label{dense model approx quad}
\biggnorm{\frac{\hat{f}}{\norm{\nu}_1}-\frac{\hat{g}}{N}}_\infty \ll  \log(\theta^{-1})^{-3/2}.
\end{equation}
Provided that $\theta \le \exp(-C\delta^{-1})$ with $C$ a large positive constant, we can compare Fourier coefficients at 0 to deduce that $\norm{g}_1 \gg \delta N$.  Applying Lemma \ref{functional sarkozy} then gives
\begin{equation}\label{g lower bound quad}
T_1( g;B) \gg_{\delta, M} N^{\frac{5}{2}}.
\end{equation}

Let $h$ denote the indicator function of the set $\{x^2 : x \in B\cap [\sqrt{N}]\}$.  Then for functions $h_1, h_2 : [N] \to \C$ we have
\begin{equation}\label{T_1 as a linear sum quad}
T_1(h_1, h_2; B) = \sum_{x_1 - x_2 = x_3+x_4+x_5} h_1(x_1)h_2(x_2) h(x_3)h(x_4)h(x_5).
\end{equation}
The function $h$ is majorised by the indicator function of the set 
\[
\{x^2 : x \in [\sqrt{N}]\}
\]
which, by Lemma \ref{BabyRestriction}, satisfies a $4.995$-restriction estimate with constant $O(1)$.

The function $g$ is majorised by $1_{[N]}$, which satisfies a 4.995-restriction estimate with constant $O(1)$.  Employing the generalised von Neumann lemma (Lemma \ref{gen von neu}), together with \eqref{dense model approx quad} and \eqref{T_1 as a linear sum quad}, we deduce that 
$$
\abs{\norm{\nu}_1^{-2}T_1(f;B) - N^{-2}T_1(g;B)} \ll K N^{1/2}\log(\theta^{-1})^{-3/400}.
$$
Combining this with \eqref{g lower bound quad} and choosing $\theta \le \tet_0(\del, M, K)$ completes the proof. 
\end{proof}

\section{The $W$-trick for squares: a simplified treatment}\label{square W trick}\label{W trick for squares}
In this section we deduce our non-linear density result (Theorem \ref{non linear sarkozy}) from its pseudorandom analogue (Theorem \ref{pseudorandom}).  The heart of the matter is massaging the set of squares to appear suitably pseudorandom.  This is accomplished using a version of the $W$-trick for squares, simplified from that developed in Browning--Prendiville \cite{densesquares}.  

It is useful to have a non-linear version of the operator $T_1$ introduced in \S\ref{pseud sark}.
\begin{definition}[$T_2$ counting operator]  Given functions $f_1, f_2 : \Z \to \C$ with finite support and $B \subset \Z$, define 
\begin{align*}
T_2(f_1, f_2; B) := \sum_{\substack{x_1^2- x_2^2 = y_1^2 + y_2^2 + y_3^2}}f_1(x_1) f_2(x_2)1_B(y_1) 1_B(y_2) 1_B(y_3).
\end{align*}
We write $T_2(f; B)$ for $T_2(f, f; B)$ and $T_2(A;B)$ for $T_2(1_A; B)$.
\end{definition}

Assuming the notation and premises of Theorem \ref{non linear sarkozy}, our objective is to obtain a lower bound for $T_2(A; B)$ by relating it to an estimate for $T_1(f;B)$, where $f$ is a function bounded above by a pseudorandom majorant $\nu$, as in Theorem \ref{pseudorandom}.

Let
\begin{equation}\label{Wsquaredef}
W = 2 \prod_{p \le w} p^2,
\end{equation}
where $w = w(\delta, M)$ is a constant to be determined, and the product is over primes. 
By Lemma \ref{greedy}, applied with $S = [N]$, there exists a $w$-smooth positive integer $\zeta \ll_{\delta, w} 1$, and $\xi \in [W]$ with $(\xi, W) = 1$, such that
\begin{equation*}
|\{ x \in \Z: \zeta(\xi + Wx) \in A \}| \geq \trecip{2} \delta |\{ x \in \Z: \zeta(\xi + Wx) \in [N] \}|.
\end{equation*}
Set
$$
A_1 := \{  \trecip{2}W x^2 + \xi x : \zeta(\xi + Wx) \in A \}\setminus\set{0}
$$
and, noting that $(2W)^{1/2}$ is a positive integer, set
\[
B_1 := \bigset{y \in \bN: \zeta (2W)^{1/2}y \in B}.
\]
One may check that $B_1$ is $M$-homogeneous, and that there exists an absolute constant $C$ such that if $N \geq C(\delta \zeta W)^{-1}$ then
\begin{equation}\label{A1 large}
 |A_1| \geq \frac{\delta N}{8\zeta W}.
\end{equation}
By the binomial theorem
\begin{equation}\label{AA_1}
T_2(A; B) \geq T_1(A_1; B_1).
\end{equation}

We note that although the squares are not equidistributed in arithmetic progressions with small modulus, the same cannot be said of the set 
\begin{equation}\label{more pseudo}
\set{\trecip{2} W x^2 + \xi x : x \in \N}.
\end{equation}  This is the reason for our passage from $A$ to $A_1$; the latter is a subset of the more pseudorandom set \eqref{more pseudo}.  Unfortunately, the (truncated) Fourier transform of \eqref{more pseudo} still does not behave sufficiently like that of an interval: they decay differently around the zero frequency, reflecting the growing gaps between consecutive elements of \eqref{more pseudo}.  To compensate for this, we must work with a weighted indicator function of $A_1$ that counteracts this increasing sparsity.

We first observe that $A_1$ is contained in the interval $[X]$, where
$$
X :=\frac{1}{2W} \cdot \brac{\frac{N}{\zeta}}^2.
$$
Define a weight function $\nu : [X]  \to [0, \infty)$ by 
\begin{equation}\label{square majorant}
\nu(n) = \begin{cases} Wx+ \xi, & \text{if } n = \trecip{2}W x^2 + \xi x \text{ for some } x \in [N/\zeta] \\
0, & \text{otherwise.}\end{cases}
\end{equation}
Since the results we are about to invoke are independent of the normalisation of $\nu$, we note that we could replace the weight $Wx+ \xi$ in the above definition by $x$, or even by $\sqrt{n}$.  We have chosen to incorporate the more complicated weight in order to make calculations a little cleaner. The weight $\nu(\cdot)$ has average value 1, since
\begin{equation} \label{normalisation}
\sum_{n\in [X]} \nu(n) =  \sum_{x \le \frac{N}{\zeta W}} Wx+ O(N/\zeta)  = X + O\brac{W^{1/2}X^{1/2}}.
\end{equation}


\begin{lemma} [Density transfer] For $N$ large in terms of $w$ and $\del$ we have
$$
\sum_{n \in A_1} \nu(n) \geq \tfrac{\delta^2}{256} \norm{\nu}_1.
$$
\end{lemma}
\begin{proof}
For $N$ sufficiently large in terms of $\delta$ and $w$ the estimate \eqref{A1 large} holds so, with $Z > 0$ a parameter, we have
$$
\sum_{\substack{\recip{2}Wx^2 + \xi x \in A_1 \\ x > Z}}1 \ge |A_1| - Z \geq  \tfrac{\del N}{8\zeta W}   - Z.
$$
Therefore
\[
\sum_{n \in A_1} \nu(n) \ge \sum_{\recip{2}Wx^2 + \xi x \in A_1} Wx  \ge WZ\brac{\tfrac{\delta N}{8\zeta W} - Z} = \tfrac{\delta ZN}{8\zeta} - WZ^2 .
\]
Choosing $Z =  \tfrac{\delta N}{16\zeta W}$ gives
\[
\sum_{n \in A_1} \nu(n) \geq \tfrac{\delta^2 N^2}{256\zeta^2 W} = \tfrac{\delta^2}{128} X.
\]
An application of \eqref{normalisation} completes the proof.
\end{proof}

The following two ingredients are established in Appendices \ref{AppendixB} and \ref{AppendixC}.

\begin{lemma}[Fourier decay] \label{decay}
We have 
$$
\|\hat \nu - \hat{1}_{[X]} \|_\infty \ll  Xw^{-1/2}.
$$
\end{lemma}

\begin{lemma}[Restriction estimate] \label{restriction} For any real number $p > 4$ we have
\begin{equation*}
\begin{split}
\sup_{|\phi| \le \nu} \int_{\T} \abs{\hat{\phi}(\alpha)}^p \intd\alpha \ll_p X^{p-1}.
\end{split}
\end{equation*}
\end{lemma}

\begin{proof}[Proof of Theorem \ref{non linear sarkozy}] 
Let $K$ denote the absolute constant implicit in Lemma \ref{restriction} when $p = 4.995$.  Let $N_0$ and $\theta$ denote the parameters occurring in Theorem \ref{pseudorandom} with respect to a density of $\delta^2/256$, restriction constant $K$ and homogeneity of level $M$.  Employing Lemma \ref{decay}, we may choose $w = w(\del, M)$ sufficiently large to ensure that $\nu$ has Fourier decay of level $\theta$ with respect to $1_{[X]}$.  Setting $f =  \nu 1_{A_1}$ in Theorem \ref{pseudorandom} yields
\[
T_1(\nu 1_{A_1}; B_1) \gg_{\del, M} X^{\frac{5}{2}}.
\]
Hence by \eqref{AA_1} we obtain
$$
T_2(A; B)  \ge \| \nu \|_\infty^{-2} T_1(\nu 1_{A_1};  B_1) \gg_{\del,M}  \| \nu \|_\infty^{-2}X^{\frac{5}{2}}.
$$
This inequality completes the proof of Theorem \ref{non linear sarkozy} on noting that $X \gg_{\delta, M} N^2$ and $\norm{\nu}_\infty \ll N$.
\end{proof}

\section{Multidimensional homogeneous Furstenberg--S\'ark\"ozy}\label{square multidim section}

It remains to establish Theorem \ref{super saturated sarkozy}.  In \S\ref{square varnavides} we derive this supersaturated counting result from a multidimensional `existence' result, Theorem \ref{multidimensional sarkozy}, whose proof is the aim of this section. One can prove Theorem \ref{multidimensional sarkozy} succinctly using the multidimensional polynomial Szemer\'edi theorem of Bergelson--Leibman \cite{BergelsonLeibman}, see Corollary \ref{smooth multidimensional sarkozy generalisation} for such an argument.  One may regard such an approach as overkill, and of little utility if one is interested in quantitative bounds.  In this section we opt for a more circuitous approach which demonstrates how Fourier analysis suffices for Theorem \ref{multidimensional sarkozy}.  More precisely, we adapt the Fourier-analytic density increment strategy originating with Roth \cite{roth} and S\'ark\"ozy \cite{sarkozy}, and show how it may accommodate the presence of homogeneous sets.  The structure of our argument is based on Green \cite{greensarkozy}.
\begin{lemma}[Density increment lemma]\label{density increment lemma}
Let $B_i$ be $M$-homogeneous sets of positive integers, and let $A \subset [N]^d$ have size at least $\delta N^d$.  Suppose that $A \times A$ lacks $(x, x')$ satisfying 
\begin{equation}\label{multidim sarkozy eqn}
x-x' = (y_1^2, \dots, y_d^2) 
\end{equation}
with $(y_1, \dots, y_d) \in B_1 \times \dots \times B_d$.  Then either
\begin{equation}\label{increment inequality}
N \le C_d(\delta^{-1} M^d)^{C},
\end{equation}
or there exist
\begin{enumerate}[(i)]
\item  $M$-homogeneous sets $B_i' \subset \N$;
\item a positive integer $N_1 \geq N^{c}$, where $c > 0$ is an absolute constant;
\item  a multidimensional set $A_1 \subset [N_1]^d$ such that
	\begin{enumerate}[{(i}a)]
	\item $A_1\times A_1$ lacks tuples satisfying \eqref{multidim sarkozy eqn} with $(y_1, \dots, y_d) \in B_1' \times \dots \times B_d' $;
	\item  $|A_1| \geq (\delta + c_d(\delta M^{-d})^6) N_1^d$.
	\end{enumerate}
\end{enumerate}
\end{lemma}

\begin{proof}[Proof of Theorem \ref{multidimensional sarkozy} given Lemma \ref{density increment lemma}]
Let us assume that $A \subset [N]^d$ has size at least $\delta N^d$ and lacks solutions to \eqref{multidim sarkozy eqn} with $y_i \in B_i$, where the $B_i$ are $M$-homogeneous sets. Setting $A_0 := A$, we iteratively apply Lemma \ref{density increment lemma} to obtain a sequence of sets $A_0, A_1, A_2, \dots$, each contained in an ambient grid $[N_n]^d$ with 
$$
N_n \geq N^{c^n}, \qquad |A_n| \geq \brac{\delta + nc_d(\delta M^{-d})^6} N_n^d.
$$
If this iteration continues until $n$ is sufficiently large in terms of $d,\delta, M$, we obtain a density exceeding 1, which would be impossible.  Hence for some $n \ll_{d, \delta, M} 1$ the inequality \eqref{increment inequality} is satisfied with $N_n$ in place of $N$ therein.  Therefore
$$
N \le N_n^{C^n}\leq \brac{C_d(\delta^{-1} M^d)^{C}}^{C^n} \ll_{ d, \delta, M} 1.
$$   
\end{proof}

We henceforth proceed with the proof of Lemma \ref{density increment lemma}. Put
$$
B := B_1 \times \dots \times B_d,
$$
let $A \subset [N]^d$ with $|A| = \delta N^d$, and suppose $A\times A$ lacks tuples $(x, x')$ satisfying \eqref{multidim sarkozy eqn} with $(y_1, \dots, y_d) \in B$.

For $f,g : [N]^d \to \C$, define the counting operator
$$
T_B(f, g) := \sum_{x- x' = (y_1^2, \dots, y_d^2) }f(x) g(x') 1_B(y_1, \dots, y_d) .
$$
Write $T_B(f)$ for $T_B(f, f)$. With this notation, our assumption is that
$$
T_B(1_A) = 0.
$$
Let $f = 1_A - \delta 1_{[N]^d}$ denote the \emph{balanced function} of $A$ in $[N]^d$.  Then by bilinearity
\begin{equation*}
\begin{split}
T_B(1_A) = T_B(1_A, f) +  T_B(f,  \delta1_{[N]^d}) + \delta^2 T_B(1_{[N]^d}).
\end{split}
\end{equation*}
Hence there exists $g :[N] \to [0, 1]$ with $\norm{g}_{2} \le  \sqrt{\delta N}$ and such that
\begin{equation}\label{non uniformity}
|T_B(g, f)| \geq \trecip{2} \delta^2 T_B(1_{[N]^d}) \quad \text{or}\quad |T_B(f, g)| \geq \trecip{2} \delta^2 T_B(1_{[N]^d}).
\end{equation}
Since the balanced function $f$ has average value $0$, one can regard \eqref{non uniformity} as exhibiting the fact that $f$ displays some form of non-uniformity.  In order to demonstrate this formally we require the following lemmas.

%

\begin{lemma}[Homogeneous counting lemma]\label{homogeneous lower bound}
Let $B = B_1\times \dotsm\times B_d $ be a product of $M$-homogeneous sets.  Then for $N \geq 64 M^2$ we have
$$
T_{B}(1_{[N]^d}) \geq \brac{\frac{N^{\frac{3}{2}}}{8M^2}}^d.
$$
\end{lemma}

\begin{proof}
It suffices to prove the result for $d = 1$, since
$$
T_B(1_{[N]^d}) = \prod_{i=1}^d T_{B_i}(1_{[N]}).
$$
If $y \in \sqbrac{\sqrt{N/2}}$ then $y^2 \in [N/2]$, so for $y$ in this interval we have
\begin{equation*}
\begin{split}
\sum_{x-x' = y^2} 1_{[N]}(x) 1_{[N]}(x')  = \hash \set{y^2 + x : 1 \le x \le N- y^2} \geq \frac{N}{2} .
\end{split}
\end{equation*}
Summing over $y$ lying in the intersection of this interval with a homogeneous set $B$, we apply Lemma \ref{homogeneous density lemma} to deduce that \begin{equation*}
\begin{split}
T_B(1_{[N]}) &\geq \sum_{y \in B \cap [\sqrt{N/2}]} \sum_{x-x' = y^2} 1_{[N]}(x) 1_{[N]}(x') \\
& \geq \recip{M}\bigg\lfloor\frac{\lfloor\sqrt{N/2}\rfloor}{M}\bigg\rfloor \frac{N}{2}.
\end{split}
\end{equation*}
The result follows provided that $N$ is sufficiently large.
\end{proof}

\begin{lemma}[Generalised von Neumann theorem]\label{sarkozy von neumann}
Let $f_1, f_2 : [N]^d \to [-1, 1]$.  Then for $i = 1, 2$ we have
$$
|T_B(f_1,f_2)| \ll_d N^{3d/2}  \brac{ \frac{\bignorm{\hat{f_i}}_{L^\infty(\T^d)}}{N^d}}^{1/3}.
$$
\end{lemma}

\begin{proof}
We prove the result for $i = 1$, the other case being similar. For $\alp = (\alp_1, \ldots, \alp_d) \in \bT^d$, define 
$$
S_B(\alpha) := \sum_{y \in B\cap [\sqrt{N}]^d} e(\alpha_1 y_1^2 + \dots + \alpha_d y_d^2).
$$
By orthogonality and H\"older's inequality, we have
\begin{equation*}
\begin{split}
|T_B(f_1, f_2)| &= \abs{ \int_{\T^d} S_B(\alpha)\hat{f}_1(-\alpha) \hat{f}_2(\alpha) \intd\alpha}\\
& \le \norm{S_B}_{L^6(\T^d)} \bignorm{\hat{f}_1}^{1/3}_{L^\infty(\T^d)}\bignorm{\hat{f}_1}_{L^2(\T^d)}^{2/3} \bignorm{\hat{f}_2}_{L^2(\T^d)}.
\end{split}
\end{equation*}
The result now follows on incorporating Parseval's identity 
\[
\bignorm{\hat{f}_i}_{L^2(\T^d)} = \bignorm{f_i}_{L^2(\Z^d)} \leq N^{d/2}
\]
together with the estimate
$$
\int_{\T^d} |S_B(\alpha)|^6 \intd\alpha \ll_d N^{2d}.
$$
The latter mean value estimate follows from orthogonality and Theorem \ref{unrestricted lower bound}.
\end{proof}

When taken in conjunction with \eqref{non uniformity}, Lemmas \ref{homogeneous lower bound} and \ref{sarkozy von neumann} imply that for $N \geq 64 M^2$ there exists $\alpha \in \T^d$ for which
\begin{equation}\label{large fourier}
\begin{split}
\bigabs{\hat{f}(\alpha)} \gg_d  \brac{\delta M^{-d}}^6 N^d.
\end{split}
\end{equation}

\begin{lemma}[Fragmentation into level sets]\label{fragmentation lemma}
If $\alpha \in \T^d$, $Q \ge 1$ and $P \in \bN$ then there exist positive integers $q_i \le Q$ and a partition of $\Z^d$ into sets $R$ of the form
\begin{equation}\label{rectangles}
\prod_{i=1}^d \brac{a_i + q^2_i\cdot(-P, P]}
\end{equation}
such that for any $g : \Z^d \to [-1, 1]$ with finite $L^1$ norm we have the estimate
\begin{equation}\label{frag ineq}
\begin{split}
\abs{\hat{g}(\alpha)} \le   \sum_R \Bigabs{\sum_{x \in R} g(x)} + O_{d}\brac{  \norm{g}_{1}PQ^{-1/3}}.
\end{split}
\end{equation}
\end{lemma}

\begin{proof}
By a weak form of a result of Heilbronn \cite{heilbronn}, there are $q_1, \dots, q_d \le Q$ such that 
\begin{equation}\label{wooley}
\begin{split}
 \norm{\alpha_i q_i^2} \ll Q^{-1/3} \qquad (1\le i \le d).
\end{split}
\end{equation}
We partition $\Z^d$ into congruence classes of the form
$$
\prod_i \brac{a_i + q_i^2 \cdot \Z},
$$
then partition each copy of $\Z$ appearing in this product into a union of intervals of the form $2nP + (-P,P]$ with $n \in \Z$.  This yields a partition of $\Z^d$ into sets $R$ of the form \eqref{rectangles}.

If $x, x'$ lie in the same $R$ then $x - x' = (q_1^2 y_1, \dots, q_d^2 y_d)$ for some $y \in (-P, P]^d$, and so
$$
|e(\alpha\cdot x) - e(\alpha \cdot x')| \ll \sum_{i=1}^d P \norm{q_i^2 \alpha_i} \ll_d PQ^{-1/3}.
$$
It then follows from the triangle inequality that
\begin{equation*}
\begin{split}
\abs{\hat{g}(\alpha)} & \le \sum_R \biggabs{\sum_{x \in R} g(x) e(\alpha \cdot x)}\\
& = \sum_R\biggabs{\sum_{x \in R} g(x)} + O_{d}\Bigbrac{\sum_R\sum_{x \in R} |g(x)| PQ^{-1/3}}.
\end{split}
\end{equation*}
\end{proof}

Let us take $P := \floor{N^{1/9}}$ and $Q := N^{3/8}$. Then, provided that \eqref{increment inequality} fails to hold, we have
\begin{equation}\label{LM bounds}
PQ^{-1/3}\le  c_d(\delta M^{-d})^6, \qquad  PQ^2N^{-1} \le  c_d (\delta M^{-d})^6 .
\end{equation}
With these bounds in hand, we claim that we may apply Lemma \ref{fragmentation lemma} to \eqref{large fourier} and conclude that there exists a set $R$ contained in $[N]^d$ and of the form \eqref{rectangles} for which
\begin{equation}\label{balanced pigeon hole}
\sum_{x \in R} f(x) \gg_{d} (\delta M^{-d})^6 |R|.
\end{equation}
Let us presently set about showing this.

The first bound in \eqref{LM bounds}, together with \eqref{frag ineq}, implies that 
$$
\sum_R\biggabs{\sum_{x \in R} f(x)} \gg_{d} (\delta M^{-d})^6 N^d.
$$
By definition, the balanced function has average value $\sum_x f(x) = 0$, so adding this quantity to either side of the inequality gives
$$
\sum_R\max\Bigset{0, \sum_{x \in R} f(x)} \gg_{d} (\delta M^{-d})^6 N^d.
$$
Inspection of the proof of Lemma \ref{fragmentation lemma} reveals that the number of $R$ which intersect $[N]^d$ is at most 
\begin{equation}\label{no intersection bound}
\brac{ \frac{N}{2P}+ Q^2}^d.
\end{equation}
Similarly, the number of $R$ contained in $[N]^d$ is at least
$$
\brac{ \frac{N}{2P} - Q^2}^d.
$$
Therefore
$$
\sum_{R\subset [N]^d}\max\Bigset{0, \sum_{x \in R} f(x)} \geq \brac{c_d(\delta M^{-d})^6 - C_d Q^2PN^{-1}}N^d.
$$
The second inequality in \eqref{LM bounds} now implies that
$$
\sum_{R\subset [N]^d}\max\Bigset{0, \sum_{x \in R} f(x)} \gg_d (\delta M^{-d})^6N^d.
$$
By \eqref{LM bounds} and \eqref{no intersection bound}, the number of $R$ contained in $[N]^d$ is $O_d((N/P)^d)$. An application of the pigeonhole principle finally confirms \eqref{balanced pigeon hole}.

The estimate \eqref{balanced pigeon hole} completes our proof of Lemma \ref{density increment lemma}, for if $R$ takes the form \eqref{rectangles} with $P = \floor{N^{1/9}}$
then we may take $N_1 := 2P$, $B_i' := \set{y \in \N : q_i y \in B_i}$ and
$$
A_1 := \set{x \in [N_1]^d : (a_1, \dots, a_d) +  \brac{q_1^2(x_1-P), \dots, q_d^2(x_d-P)} \in A}.
$$

\section{Varnavides averaging for supersaturation}\label{square varnavides}
We complete the proof of Theorem \ref{finitary main theorem} by deducing the counting result, Theorem \ref{super saturated sarkozy}, from the multidimensional existence result, Theorem \ref{multidimensional sarkozy}.  The deduction proceeds by collecting a single configuration from many subprogressions, then establishing that these configurations don't coincide too often.  This random sampling argument originates with Varnavides \cite{varnavides}.
\begin{proposition}[Varnavides argument]\label{varnavides}
For any $\delta > 0$ and $d, M \in \N$ there exists $N_0 \in \N$ 
 and $c_0 >0 $ 
  such that for any $N \geq N_0$ the following holds.
If $A \subset [N]^d$ is at least $\delta$-dense in $[N]^d$ and $B_1, \dots, B_d$ are $M$-homogeneous sets of natural numbers, then there are at least $c_0 N^{\frac{3d}{2}}$ tuples $(x, x',y) \in A^2 \times B_1\times \dots \times B_d $ satisfying \eqref{2nd multidim eqn}.
\end{proposition}

\begin{proof}
For $q, n \in \Z^d$ write $q^{\otimes 2} \otimes n$ for the tuple $(q_1^2 n_1, \dots, q_d^2n_d)$ and write $q^{\otimes2} \otimes [N]^d$ for the set
$$
 \set{q^{\otimes 2} \otimes n : n \in [N]^d}.
$$
Let $N_0 = N_0(\delta/2^{1+d}, d, M)$ be the quantity given by Theorem \ref{multidimensional sarkozy}.  Suppose that $N \geq N_0$ and write $Q := \floor{\sqrt{N/N_0}}$. Averaging, we have
\begin{equation*}
\begin{split}
\sum_{z \in \Z^d} \sum_{q\in [Q]^d} |A \cap \brac{z + q^{\otimes2} \otimes [N_0]^d}|  \geq \delta (NQN_0)^d.
\end{split}
\end{equation*}
By the definition of $Q$, there are at most $(2N)^d$ choices for $z$ for which there exists $q \in [Q]^d$ such that
$$
|A \cap \brac{z + q^{\otimes2} \otimes [N_0]^d}| \neq 0.
$$
Hence there are at least $\trecip{2} \delta (NQ)^{d}$ choices for $(z, q) \in \Z^d \times \N^d$ for which 
\begin{equation}\label{popularity lower bound}
\begin{split}
|A \cap \brac{z + q^{\otimes2} \otimes [N_0]^d}| \geq 2^{-1-d}\delta N_0^d.
\end{split}
\end{equation}
Call each such choice of $(z, q)$ a \emph{good} tuple.  Define
$$
A_{z, q} := \set{ y \in [N_0]^d : z + q^{\otimes 2} \otimes y \in A}.
$$
If $(z,q)$ is good then $|A_{z, q}| \geq 2^{-1-d} \delta N_0^d$.  Applying Theorem \ref{multidimensional sarkozy} we see that there exist $x, x' \in A_{z, q}$ satisfying \eqref{2nd multidim eqn} with the $y_i$ restricted to the $M$-homogeneous sets
$$
\set{y_i : q_i y_i \in B_i}.
$$
Translating and dilating, we deduce that each set $A \cap \brac{z + q^{\otimes2} \otimes [N_0]^d})$ satisfying \eqref{popularity lower bound} contains a solution to \eqref{2nd multidim eqn} with $y_i \in B_i$.

For fixed $(x,y) \in \Z^d \times \N^d$ define $R(x,y)$ to be the quantity
$$
\hash\set{(z,q) \in \Z^{d}\times [Q]^d : \set{x,\ x+y^{\otimes 2}}  \subset \brac{z + q^{\otimes 2} \otimes [N_0]^{d}}}.
$$
\begin{claim}
$R(x,y) \le N_0^{2d}.$
\end{claim}
\noindent To see this, observe that if $\set{x,\ x+y^{\otimes 2}} \subset z + q^{\otimes 2} \otimes [N_0]^{d} $ then for each $i$ there exists $m_i \in [N_0]$ such that 
$
y_i^{2} =  q_i^2 m_i.
$
As there are at most $N_0$ choices for $m_i$ for fixed $y_i$, there are at most $N_0^d$ choices for $q$.
Once one has fixed this choice of $q$ we have
$$
z \in x - q^{\otimes 2} \otimes [N_0]^d,
$$
so there are at most $N_0^d$ choices of $z$ for fixed $x$.  This establishes the claim.

Invoking the claim gives
\begin{multline*}
  N_0^{2d} \hash\set{(x, y) \in\Z^{d} \times B_1\times \dotsm \times B_d : \set{x,\ x+y^{\otimes 2}}  \subset A} \\ 
 \geq  \sum_{\substack{x \in \Z^n, y \in B_1\times \dotsm B_d\\ x,\ x+y^{\otimes 2} \in A}} R(x,y).
\end{multline*}
Next we interchange the order of summation to find that
\begin{align*}
\sum_{\substack{x \in \Z^d, y \in B_1\times \dotsm B_d\\ x,\ x+y^{\otimes 2} \in A}} & R(x,y)\\ & \geq \sum_{z\in \Z^d} \sum_{q \in [Q]^d} \hash\set{(x,y) : \set{x,\ x+y^{\otimes 2}} \subset A \cap \brac{ z + q^{\otimes 2}\otimes [N_0]^{d}}}\\
&  \geq \hash\set{(z,q)\in \Z^{d} \times [Q]^d: (z,q) \text{ is good}} \\
&  \geq \trecip{2} \delta (NQ)^{d}.
\end{align*}
It follows that
\begin{multline*}
\hash\set{(x, y) \in\Z^{d} \times B_1\times \dotsm \times B_d : \set{x,\ x+y^{\otimes 2}}  \subset A} \\
\geq \trecip{2} \delta N_0^{-2d} N^d \floor{\sqrt{N/N_0}}^d.
\end{multline*}
The result follows since $N_0 \ll_{\delta, d, M} 1$.
\end{proof}

\begin{proof}[Proof that Proposition \ref{varnavides} implies Theorem \ref{super saturated sarkozy}]
We prove a more general result for sums of $d$ squares.  First note that, by translation, Proposition \ref{varnavides} remains valid for dense subsets of $[-N, N]^d$.
Given $A \subset [N]$ of density at least $\delta$, define
$$
A' := \set{ x \in [-N, N]^d : x_1 + \dots + x_d \in A}
$$
For every element $n$ of $[-N, N]$ there are at least $N$ pairs $(n_1, n_2) \in [-N, N]^2$ such that $n = n_1 + n_2$.  An induction then shows that for each $n \in [-N, N]$ we have
$$
\hash\set{(n_1, \dots, n_d) \in [-N, N]^d : n = n_1 + \dots + n_d} \geq N^{d-1}.
$$
Consequently
$$
|A'| \geq \delta N^d \gg_d \delta |[-N, N]^d|.
$$
Applying Proposition \ref{varnavides} with $B_i := B$ for all $i$, we deduce that there are at least $c_0 N^{3d/2}$ tuples $(x, y) \in A' \times B^d$ such that $x + (y_1^2, \dots, y_d^2) \in A'$.
For each such tuple the sum $x = x_1 + \dots +x_d$ is an element of the one-dimensional set $A$, as is $x + y_1^2 + \dots + y_d^2$.  As each element of $A$ has at most $(2N+1)^{d-1}$ representations of the form $x_1 + \dots + x_d$, it follows that the number of solutions to 
$$
x - x' = y_1^2 + \dots + y_d^2
$$
is at least
$$
c_0 N^{\frac{3d}2 - (d-1)} = c_0 N^{1 + \frac{d}{2}},
$$
as required.
\end{proof}

\part{Rado's criterion over squares and higher powers}\label{higher part}

In this part we prove Theorem \ref{intro main theorem}. Let $\eta = \eta_k > 0$ be a fixed constant, where $\eta_2 = 1$, and $\eta_k$ is sufficiently small when $k \ge 3$. In other words, we will work with smooth numbers when $k \ge 3$, but not when $k = 2$. This choice will improve our mean value estimate in the former situation, and our minor arc estimate in the latter.

\section{The smooth homogeneous Bergelson--Leibman theorem}\label{smooth multidim section}

We begin our investigation of Rado's criterion in $k$th powers by generalising Theorem \ref{super saturated sarkozy}, which asserts that dense multidimensional sets contain configurations of the form 
$$
(x_1,\dots, x_d), \quad (x_1 + y_1^2, \dots, x_d + y_d^2)
$$
 with the $y_i$ lying in a homogeneous set.  We require a version of this result which concerns affine configurations determined by $k$th powers, similar in flavour to the following special case of the multidimensional polynomial Szemer\'edi theorem of Bergelson--Leibman \cite{BergelsonLeibman}.
\begin{bl}
Let $k \in \N$, $\delta > 0$ and let $F \subset \Z^d$ be a finite set.   There exists $N_0 = N_0(k, \delta, F)$ such that for any $N \geq N_0$, if $A \subset [N]^d$ has size $|A| \geq \delta N^d$ then there exists $x \in \Z^d$ and $y \in \N$ such that
$$
x + y^k \cdot F \subset A.
$$
\end{bl}
We require a version of this result in which the $k$th power comes from a homogeneous set.  Fortunately, this strengthening can be deduced from the original.  It is convenient to set up the following notation.
\begin{notation}
Given $q, y, k \in \N^d$ define
$$
q \otimes y := (q_1y_1, \dots, q_d y_d), \qquad y^{\otimes k} := (y_1^{k_1}, \dots, y_d^{k_d}).
$$
For $F \subset \Z^d$, write $q \otimes F$ for the set
$$
\set{q \otimes y : y \in F}.
$$
\end{notation}
Here is our version of the Bergelson--Leibman theorem with common difference arising from a homogeneous set.  

\begin{corollary}[Homogeneous Bergelson--Leibman]\label{smooth multidimensional sarkozy generalisation}
Let $k \in \N^d$, $M \in \N$, $\delta > 0$ and  let $F \subset \Z^d$ be a finite set.   There exists $N_0$ such that for any $N \geq N_0$, if $A \subset [N]^d$ has size $|A| \geq \delta N^d$ and $B_1, \dots, B_d \subset \N$ are $M$-homogeneous, then there exists $x \in \Z^d$ and $y_1 \in B_1, \dots, y_d \in B_d$ such that
\begin{equation}\label{smooth poly config}
x + y^{\otimes k}\otimes F  \subset A.
\end{equation}
\end{corollary}
\begin{proof}
Let $K := \prod_i k_i$ and consider the finite set
$$
F':= [M^K]^d \otimes F.
$$
By the Bergelson--Leibman theorem, provided that $N \gg_{M, K, F, \delta} 1$, there exist $x \in \Z^d$ and $t \in \N$ such that 
$$
x + t^K \cdot F' \subset A.
$$
The result follows if the progression $t^K \cdot [M^K]$ contains an element of the form $y_i^{k_i}$ for some $y_i \in B_i$.

Let $z_i := t^{K/k_i}$.  Then 
$$
\set{z_i^{k_i}, (2z_i)^{k_i}, \dots, (Mz_i)^{k_i}} = t^K \cdot \set{1^{k_i}, 2^{k_i}, \dots, M^{k_i}} \subset t^K\cdot [M^K].
$$
Since each $B_i$ is $M$-homogeneous, it intersects the set $z_i \cdot [M]$.
\end{proof}

Next we require a counting analogue of this result.  In fact, we need to count the number of configurations arising from a smooth common difference.  
Before stating the theorem, we remind the reader of what it means for a set to be $M$-homogeneous in the $N^\eta$-smooths (see Definitions \ref{smooth def} and \ref{homogeneous definition}).

\begin{theorem}[Varnavides averaging]\label{smooth varnavides}
Let $k_1, \dots, k_d, M \in \N$, $\eta, \delta \in (0,1]$, and  let $F \subset \Z^d$ be  a finite set.   There exist $N_0 \in \N$ and $c_0 > 0$ such that for any $N \geq N_0$, if $A \subset [N]^d$ has $|A| \geq \delta N^d$ and $B\subset \N $ is $M$-homogeneous in the $N^\eta$-smooths, then the number of tuples $(x, y) \in \Z^d \times B^d$ for which \eqref{smooth poly config} holds is at least
$$
c_0 N^{d + \recip{k_1} + \dots + \recip{k_d}}.
$$
\end{theorem}

\begin{proof}
Increasing the size of $F$ if necessary, we may assume that $F$ contains two elements which differ in the $i$th coordinate for each $i \in [d]$.  Let $N_0$ be the quantity given by Corollary \ref{smooth multidimensional sarkozy generalisation} with respect to the density $\delta/ 2^{d+1}$.  Suppose that
\begin{equation}\label{N eta large}
N \geq  N_0^{1/\eta},
\end{equation}
and define the following sets of smooths:
$$
S_i := S\bigbrac{\bigfloor{\sqrt[k_i]{N/N_0}}; N^\eta}.
$$  

Interchanging the order of summation, we have
\begin{equation*}
\begin{split}
\sum_{z \in \Z^{d}} \sum_{q_1\in S_1}\dots \sum_{q_d \in S_d} |A \cap \brac{z + q^{\otimes k} \otimes [N_0]^{d}}|  \geq \delta |S_1|\dotsm |S_d|(NN_0)^{d}.
\end{split}
\end{equation*}
Notice that there are at most $(2N)^{d}$ choices for $z$ for which there exists $q \in S_1\times \dots \times S_d$ such that
$$
|A \cap \brac{z + q^{\otimes k} \otimes [N_0]^{d}}| \neq 0.
$$
Hence  there are at least $\trecip{2} \delta N^{d}|S_1|\dotsm |S_d|$ choices for $(z,q) \in \Z^{d}\times \prod_i S_i$ for which 
\begin{equation}\label{smooth popularity lower bound}
\begin{split}
|A \cap \brac{z + q^{\otimes k} \otimes [N_0]^{d}}| \geq 2^{-d-1}\delta N_0^{d}.
\end{split}
\end{equation}
Call such a choice of $(z,q)$ a \emph{good} tuple.\\[5pt] \noindent
\textbf{Claim 1.}  For each good tuple $(z,q)$ the set $A \cap \brac{z + q^{\otimes k} \otimes [N_0]^{d}}$ contains a configuration of the form $x+y^{\otimes k}\otimes F$ for some $x \in \bZ^d$ and some $y \in B^d$.

To see this, define
$$
A_{z, q} := \set{ x \in [N_0]^{d} : z + q^{\otimes k} \otimes x  \in A}.
$$
Then $|A_{z, q}| \geq 2^{-d-1} \delta N_0^{d}$.  Let 
$$
B_i = \set{y_i \in [N_0]: q_i y_i \in B} \cup (N_0, \infty).
$$ 
Using the fact that $B$ is $N^\eta$-smoothly $M$-homogeneous, together with \eqref{N eta large}, one can check that each $B_i$ is $M$-homogeneous (not just smoothly homogeneous).  
Invoking Corollary \ref{smooth multidimensional sarkozy generalisation}, we see that there exist $x \in \Z^{d}$ and $y \in B_1\times \dots \times B_d$ such that 
$$
x + y^{\otimes k} \otimes F \subset A_{z,q}.
$$
Translating and dilating, we deduce that $A \cap \brac{z + q^{\otimes k} \otimes [N_0]^{d}}$ contains a configuration of the form $x'+(q\otimes y)^{\otimes k}\otimes F$.  By definition of the $B_i$ and the fact that $F$ is non-constant in each coordinate, we see that $y\in [N_0]^d$ and thus each coordinate of $q\otimes y$ lies in $B$.  This establishes Claim 1.    

For fixed $(x,y) \in \Z^d \times \N^d$ let $G(x,y)$ denote the number of  tuples $(z,q) \in \Z^d \times \N^d$ satisfying 
\begin{equation}\label{xy config}
x + y^{\otimes k} \otimes F \subset z + q^{\otimes k} \otimes [N_0]^{d}.
\end{equation}
Define
\begin{equation}
\mathcal{A} := \set{(x,y) \in \Z^d \times B^d : x+y^{\otimes k} \otimes F \subset A}.
\end{equation}
Then interchanging the order of summation shows that the sum $\sum_{(x,y) \in \mathcal{A}}G(x,y)$ is at least
\begin{align*}
&\sum_{z \in \Z^{d}} \sum_{q_1\in S_1}\dots \sum_{q_d \in S_d} \Bigl|\set{(x,y) \in \Z^d \times B^d : x+y^{\otimes k}\otimes F \subset A \cap \brac{z + q^{\otimes k} \otimes [N_0]^{d}}} \Bigr|\\
& \geq \Bigabs{ \Bigset{(z, q) \in \Z^d \times \prod_i S_i : (z,q) \text{ is good}}} \geq \trecip{2} \delta N^{d}|S_1|\dotsm |S_d|.
\end{align*}
Applying Lemma \ref{dense smooths} (for $N$ sufficiently large) we deduce that
$$
\sum_{(x,y) \in \mathcal{A}}G(x,y) \gg_{k, \delta, \eta, N_0}\ N^{d + \recip{k_1} + \dots + \recip{k_d}}.
$$  
Since the theorem asserts a lower bound on the size of $\mathcal{A}$, the result is proved provided we have the following upper bound on $G(x,y)$.\\[5pt]\noindent
\textbf{Claim 2.} Suppose that $F$ contains two elements which differ in the $i$th coordinate for each $i \in [d]$.  Then $G(x,y) \leq N_0^{2d}$.\\[5pt]\noindent

To see this, first note that if $x+y^{\otimes k}\otimes F \subset z + q^{\otimes k} \otimes [N_0]^{d} $ then, since $F$ contains two elements differing in their $i$th coordinate, there exist integers $f_i < f_i'$ such that 
$$
x_i + y_i^{k_i} f_i, \ x_i + y_i^{k_i} f_i' \in z_i + q_i^{k_i} \cdot [N_0].
$$
Subtracting these elements, we deduce that there exists $n_i \in [N_0]$ for which
$$
 q_i^{k_i} = \frac{y_i^{k_i} (f_i'- f_i)}{n_i}.
$$
As there are at most $N_0$ choices for $n_i$, and $y_i$ is fixed, there are at most $N_0^d$ choices for $q$.
Once one has fixed this choice of $q$, for any $f \in F$ we have
$$
z \in x + y^{\otimes k} \otimes f - q^{\otimes k} \otimes [N_0]^d,
$$
so there are at most $N_0^d$ choices for $z$.  In summary $G(x,y) \leq N_0^{2d}$, which establishes Claim 2.
\end{proof}

\begin{remark}[Roth's theorem with logarithmically-smooth common difference]\label{smooth roth remark}
The above argument remains valid for much stronger levels of smoothness.  For instance, one can use it to establish that if $A \subset [N]$ lacks a three-term progression with common difference equal to an $R$-smooth number then
\begin{equation}\label{smooth density}
|A| \ll \frac{r_3(R)}{R} N. 
\end{equation}
Here $r_3(N)$ denotes the size of a largest subset of $[N]$ lacking a non-trivial three-term arithmetic progression. 

\end{remark}

\section{A supersaturated generalisation of both Roth and S\'ark\"ozy's theorems}
In this section we deduce a one-dimensional counting result analogous to Theorem \ref{super saturated sarkozy} by projecting down the multidimensional Theorem \ref{smooth varnavides}.  Again we remind the reader of what it means to be $M$-homogeneous in $S(N^{1/k} ; N^\eta)$ (see Definition \ref{homogeneous definition}).

\begin{theorem}[Supersaturated smooth homogeneous Roth--S\'ark\"ozy]\label{smooth super saturated sarkozy generalisation}
Let $\lambda_1, \dots, \lambda_s, \mu_1, \dots, \mu_t \in \Z\setminus\set{0}$ with $\lambda_1 + \dots + \lambda_s = 0$.  For any $\eta, \delta > 0$ and $M \in \N$ there exists $N_0 \in \N$ 
 and $c_0 >0 $ 
  such that for any $N \geq N_0$ the following holds.
If $A $ is at least $\delta$-dense in $[N]$ and $B$ is $M$-homogeneous in $S(N^{1/k}; N^\eta)$, then there are at least $c_0 N^{s+\frac{t}{k} -1}$ tuples $(x, y) \in A^s\times B^t $ solving the equation
\begin{equation}\label{smooth one dim equation}
\lambda_1 x_1 +\dots + \lambda_s x_s = \mu_1y_1^k + \dots + \mu_t y_t^k.
\end{equation}
\end{theorem}



Notice that Theorem \ref{smooth super saturated sarkozy generalisation} is a common generalisation of both the Furstenberg--S\'ark\"ozy theorem (take $s = 2$ and $t = 1$) and Roth's theorem (take $\lambda = (1, -2, 1)$ and $t = 0$).

\begin{proof}
Given $A \subset [N]$ of density at least $\delta$, let us define
$$
\tilde{A} := \Bigset{ x \in [N]^{s + t-2} : \sum_i x_i \in A}
$$
A stars and bars argument shows that for $n \in [N]$ we have
$$
\hash\set{(n_1, \dots, n_d) \in [N]^d : n = n_1 + \dots + n_d} = \binom{n-1}{d-1}.
$$
Since there are at most $\trecip{2}|A|$ elements $x$ of $A$ satisfying the inequality $x \leq \trecip{2}|A|$, it follows that for $N \geq C_{s,t} \delta^{-1}$ we have
\begin{equation}\label{A tilde density}
|\tilde{A}| = \sum_{n \in A} \binom{n-1}{s+t-3} \gg_{s,t} \delta^{s+t-2} N^{s+t-2}.
\end{equation}

In the statement of Theorem \ref{smooth super saturated sarkozy generalisation}, at least one of the coefficients $\lambda_i$ must be positive.  Relabelling indices, we may assume that $\lambda_{s} > 0$. For a technical reason, it will be useful in a later part of the argument if we can ensure that 
  \begin{equation}\label{width assumption}
  \tilde{A} - \tilde{A} \subset \sqbrac{-\tfrac{N}{\lambda_s},\tfrac{N}{\lambda_s}}^{s+t-2}.
  \end{equation}
This follows on partitioning the hypercube $[N]^{s+t-2}$ into subhypercubes of sufficiently small side length and applying the pigeonhole principle to ensure that $\tilde{A}$ has large density on one such part (worsening the density \eqref{A tilde density} by a factor of $O_{s,t, \lambda_s}(1)$ in the process).
  
Define $F \subset \Z^{s+t-2}$ to be the set consisting of the zero vector together with the rows 
 of the following matrix
 \begin{equation}\label{F matrix}
 \begin{pmatrix}
-\lambda_s & \ & \ &\ & \ & \ & \ \\
\ & -\lambda_s & \ &\  & \ & \ & \ \\
 \ & \ & \ddots &\  & \  & \ & \ \\
  \ & \ & \ &- \lambda_{s} & \  & \ & \ \\
 \lambda_1 & \lambda_2 & \dots &\lambda_{s-2} & \lambda_s^{k-1} \mu_1 & \dots & \lambda_s^{k-1}\mu_t
 \end{pmatrix}.
 \end{equation}
Consider the set
  $$
 \tilde{B} := \set{ y \in \N : \lambda_{s} y \in B} \cup (N^{1/k}\lambda_s^{-1}, \infty) .
$$
Provided that $N^\eta \geq \max\set{\lambda_s, M}$ (as we may assume), we see that $\tilde{B}$ is $M$-homogeneous in the $N^\eta$-smooths.
Applying Theorem \ref{smooth varnavides}, we find that there are at least $c_0 N^{s+t-2 + s-2 +\frac{t}{k}}$
 tuples $(x, y, z) \in \Z^{s+t - 2} \times \tilde{B}^{s-2} \times \tilde{B}^t$ such that $\tilde{A}$ contains the configuration 
 $$
x + (y_1, \dots, y_{s-2}, z_1^k, \dots, z_t^k)\otimes F.
 $$  
By \eqref{width assumption} and \eqref{F matrix} we have $\lambda_s^k \mu_i z_i^k \in [-N, N]$, hence by definition of $\tilde{B}$ we deduce that $\lambda_s z_i \in B$.
 Projecting down to one dimension and taking into account the multiplicities of representations, we obtain $\gg N^{s +\frac{t}{k}- 1}$ tuples $(x, y, z) \in \Z \times \N^{s-2} \times \N^t$ with $\lambda_sz_i \in B$ and such that $A$ contains the  configuration
 \begin{equation*}
\begin{split}
x,\quad x-\lambda_sy_1 ,\quad \dots,\quad x-\lambda_s y_{s-2},\quad 
 x + \sum_{i=1}^{s-2}\lambda_iy_i +\lambda_s^{k-1}\sum_{j=1}^t \mu_jz_j^k.
\end{split}
\end{equation*}
Let us set  $x_i := x -\lambda_s y_i$ for $i=1, \dots, s-2$, along with $x_{s-1} = x$ and
$$
x_{s} :=x + \sum_{i=1}^{s-2}\lambda_iy_i +\lambda_s^{k-1}\sum_{j=1}^t \mu_jz_j^k.
$$ 
One can then check that the tuple $(x_1, \dots, x_s, \lambda_{s}z_1, \dots, \lambda_s z_t)$ is an element of $A^s \times B^t$ satisfying 
\eqref{smooth one dim equation}.  By construction there are $\gg N^{s +\frac{t}{k} - 1}$ such tuples.
\end{proof}
\section{Pseudorandom Roth--S\'ark\"ozy}\label{smooth transference section}

In this section we develop a pseudorandom variant of Theorem \ref{smooth super saturated sarkozy generalisation}.  As in Part \ref{pythag part}, we begin by relaxing Theorem \ref{smooth super saturated sarkozy generalisation} to encompass general bounded functions.  In order to count solutions to our equation weighted by general functions, we use the following notation.
\begin{definition}[$T_\ell$ counting operator]  Fix $\lambda_1, \dots, \lambda_s, \mu_1, \dots, \mu_t \in \Z\setminus \set{0}$ with $\lambda_1 + \dots + \lambda_s = 0$.  Given functions $f_1, \dots, f_s : \Z \to \C$ and $B \subset \Z$, write (when defined) 
\begin{align*}
T_\ell(f_1, \dots, f_s; B) := \sum_{\substack{\lambda_1x_1^\ell + \dots + \lambda_sx_s^\ell =\\ \mu_1y_1^k + \dots + \mu_t y_t^k}}f_1(x_1)  \dotsm f_s(x_s)1_B(y_1) \dotsm 1_B(y_t).
\end{align*}
We write $T_\ell(f; B)$ for $T_\ell(f, f, \dots, f; B)$ and $T_\ell(A;B)$ for $T_\ell(1_A; B)$.
\end{definition}

\begin{remark}[Dependence on constants]
In the sequel we regard the coefficients $\lambda_i$ and $\mu_j$ as fixed, and suppress their dependence in any implied constants. Similarly for the degree $k$ and the number of variables $s+ t$. We also fix $\eta = \eta_k$ globally: recall that this is 1 if $k=2$, and a small positive constant if $k \ge 3$. We opt to keep any dependence on the following explicit: the level of homogeneity $M$, and the density $\delta$.
\end{remark}

\begin{lemma}[Functional Roth--S\'ark\"ozy]\label{smooth functional sarkozy}
For any $\delta > 0$ and $M \in \N$ there exist $N_0 \in \N$ and $c_0 > 0$ such that for any $N \geq N_0$ the following holds.  Let $f: [N] \to [0, 1]$ with $\norm{f}_1 \geq \delta N$, and let $B$ be $M$-homogeneous in $S(N^{1/k};N^\eta)$.  Then 
\begin{equation*}
\begin{split}
T_1(f; B) \geq c_0 N^{s+\frac{t}{k}-1}.
\end{split}
\end{equation*}
\end{lemma}

\begin{proof} Let $A = \{ x \in [N]: f(x) \ge \del/2 \}$. As $\|f\|_1 \ge \del N$, we must necessarily have $|A| \ge \del N / 2$. Since $f \ge \del 1_A / 2$, we deduce that
\[
T_1(f;B) \ge (\del/2)^s T_1(A;B),
\]
and an application of Theorem \ref{smooth super saturated sarkozy generalisation} completes the proof.
\end{proof}

Our next step is to weaken the assumptions of Theorem \ref{smooth super saturated sarkozy generalisation} even further, replacing bounded functions with unbounded functions which are sufficiently pseudorandom, in that they possess a majorant with good Fourier decay (Definition \ref{decay def}) and $p$-restriction (Definition \ref{restriction def}).  


\begin{theorem}[Pseudorandom Roth--S\'ark\"ozy] \label{smooth pseudorandom}
There exists $s_0(k)$ such that for $s+t \geq s_0(k)$, $\delta> 0$ and $K,M \in \N$ there exist $N_0\in \N$ and $c_0, \theta > 0$ such that for $N \geq N_0$ the following holds. 
\begin{itemize}
\item  Let $\nu: [N] \to [0, \infty)$ satisfy a $(s+t - 10^{-8})$-restriction estimate with constant $K$, and have Fourier decay of level $\theta$;
\item  let $B$ be  $M$-homogeneous in $S(N^{1/k}; N^\eta)$;
\item let  $f: [N] \to [0, \infty)$ with $f \le \nu$ and $\norm{f}_1 \geq \delta \norm{\nu}_1$.
\end{itemize}
Then
\begin{equation}\label{smooth pseud conclusion}
\begin{split}
T_1(f;B) \geq c_0 \norm{\nu}_1^sN^{\frac{t}{k}-1}.
\end{split}
\end{equation}
Moreover, we may take $s_0(2) = 5$, $s_0(3) = 8$, and $s_0(k)$ satisfying \eqref{s0 size}.
\end{theorem}

\begin{proof}
By replacing $B$ with $B \cap S(N^{1/k}; N^\eta)$, we may freely suppose that $B \subset S(N^{1/k}; N^\eta)$. Deploying the dense model lemma \cite[Theorem 5.1]{FourVariants}, there exists $g: \bZ \to \bC$ satisfying $0 \le g \le 1_{[N]}$ and
\begin{equation}\label{dense model approx}
\biggnorm{\frac{\hat{f}}{\norm{\nu}_1}-\frac{\hat{g}}{N}}_\infty \ll  \log(\theta^{-1})^{-3/2}.
\end{equation}
Provided that $\theta \le \exp(-C\delta^{-1})$ with $C$ a large positive constant, we can compare Fourier coefficients at 0 to deduce that $\norm{g}_1 \gg \delta N$. Lemma \ref{smooth functional sarkozy} then gives
\begin{equation}\label{g lower bound}
T_1( g;B) \gg_{\delta, M} N^{s + \frac{t}{k} -1}.
\end{equation}

Let $h$ denote the indicator function of the set $\{x^k : x \in B \}$.  Then for functions $h_1, \dots, h_s : [N] \to [0,\infty)$ we have
\begin{equation}\label{T_1 as a linear sum}
T_1(h_1, \dots, h_s ; B) = \sum_{\lambda \cdot x  = \mu \cdot y} h_1(x_1)\dotsm h_s(x_s) h(y_1)\dotsm h(y_t).
\end{equation}
The function $h$ is majorised by the indicator function of the set 
\[
\{x^k : x \in S(N^{1/k}; N^\eta)\}
\]
which, by Lemma \ref{BabyRestriction}, satisfies an $(s+t-10^{-8})$-restriction estimate with constant $O_\eta(1)$.

Observe that $g$ is majorised by $1_{[N]}$, which also satisfies an $(s+t-10^{-8})$-restriction estimate with constant $O(1)$. The generalised von Neumann theorem (Lemma \ref{gen von neu}), together with \eqref{dense model approx} and \eqref{T_1 as a linear sum}, yields
\begin{align*}
\Biggl|  \frac{ T_1(f;B) }{\| \nu\|_1^s}
-
\frac{ T_1(g;B) }{N^s}
\Biggr|
&\ll  \frac {K |S(N^{1/k}; N^\eta)|^t} N   \log(\tet^{-1})^{1.5 \times 10^{-8}} 
\\
&\le KN^{\frac tk - 1} \log(\tet^{-1})^{1.5 \times 10^{-8}} .
\end{align*}
Pairing this with \eqref{g lower bound}, and choosing $\theta \le \tet_0(\del, M, K)$, completes the proof.

\end{proof}

\section{The $W$-trick for smooth powers and a non-linear Roth--S\'ark\"ozy theorem}

Our objective in this section is to use Theorem \ref{smooth pseudorandom} to deduce the following non-linear density result. Recall that $\eta = \eta_k$ is 1 if $k=2$, and a small positive constant if $k \ge 3$.

\begin{theorem}[Non-linear Roth--S\'ark\"ozy]\label{smooth non linear sarkozy generalisation}
There exists $s_0(k)$ such that  
the following holds.  Let $\lambda_1, \dots, \lambda_s, \mu_1, \dots, \mu_t \in \Z\setminus \set{0}$ with $s+t \geq s_0(k)$ and $\lambda_1 + \dots + \lambda_s = 0$.  For any $\delta > 0$ and $M \in \N$ there exist $N_0 \in \N$ and $c_0>0$ such that for any $N \geq N_0$ the following holds.  Let $A$ have density at least $\delta$ in $S(N;N^\eta)$ and let $B$ be  $M$-homogeneous in $S(N;N^\eta)$.  Then 
\begin{equation*}
\begin{split}
\hash\Bigset{(x, y) \in A^{s} \times B^{t} : \sum_{i =1}^s \lambda_ix_i^k = \sum_{j=1}^t \mu_j y_j^k } \geq c_0 N^{s+t-k}.
\end{split}
\end{equation*}
Moreover, we may take $s_0(2) = 5$, $s_0(3) = 8$, and $s_0(k)$ satisfying \eqref{s0 size}. 
\end{theorem}

This deduction proceeds by developing a $W$-trick for smooth $k$th powers, analogous to that developed for prime powers in \cite{chow}. Let
\begin{equation}\label{smooth Wdef}
W = k^{k-1} \prod_{p \le w} p^k,
\end{equation}
where $w = w(\eta, \delta, M)$ is a constant to be determined, and the product is over primes. 
We apply Lemma \ref{greedy} with $S = S(N; N^\eta)$, using Lemma \ref{dense smooths} in the process. This allows us to conclude that there exists a $w$-smooth positive integer $\zeta \ll_{\eta, \delta, w} 1$ and $\xi \in [W]$ with $(\xi, W) = 1$ such that
\begin{equation}\label{untruncated A1 density}
\hash\{ x \in \Z: \zeta(\xi + Wx) \in A \} \geq \trecip{2} \delta \hash\{ x \in \Z: \zeta(\xi + Wx) \in S(N;N^\eta) \}.
\end{equation}

Define
\begin{equation}\label{X and P}
P:= \frac{N}{\zeta}, \qquad X := \frac{P^k }{kW}
\end{equation}
and set
\begin{equation}\label{A1 defn}
A_1 := \set{  \tfrac{(Wx + \xi)^k - \xi^k}{kW} : \zeta(Wx + \xi) \in A  \text{ and } Wx + \xi \in S(P; P^\eta)}\setminus \set{0}.
\end{equation}
Then $A_1 \subset [X]$.  Combining \eqref{untruncated A1 density} and Lemma \ref{changing smoothness}, we have the lower bound
\begin{equation}\label{smooth A1 size}
|A_1| \geq \tfrac{\delta }{2} \hash\set{x \in S(P; P^\eta) : x \equiv \xi \bmod W} - O_{\eta, \delta, w}\brac{ N(\log N)^{-1}}.
\end{equation}
Noting that $(kW)^{1/k}$ is a positive integer, let
\begin{equation}\label{B1 defn}
B_1 := \bigset{y \in \bN: \zeta (kW)^{1/k}y \in B}.
\end{equation}
Provided that $N \geq \max\set{k,w,\zeta}^{1/\eta}$, one may check that $B_1$ is $M$-homogeneous in $S(X^{1/k} ; X^\eta)$.  Recalling that $\sum_{i = 1}^s \lambda_i = 0$, we have 
\begin{equation}\label{smooth AA_1}
T_k(A; B) \geq T_1(A_1; B_1).
\end{equation}

Define $\nu : [X]  \to [0, \infty)$ by 
\begin{equation}\label{nu defn}
\nu(n) = \begin{cases} x^{k-1}, & \text{if } n = \frac{x^k - \xi^k}{kW} \text{ for some } x \in S(P; P^\eta) \text{ with } x \equiv \xi \bmod W \\
0, & \text{otherwise.}\end{cases}
\end{equation}
First we check our $L^1$ normalisation. Let $\rho(\cdot)$ denote the Dickman--de Bruijn $\rho$-function (see \cite{Granville}). 

\begin{lemma}
 We have
\begin{equation} \label{smooth normalisation}
\sum_n \nu(n) = \rho(1/\eta)X + O_{ \eta, w}(P^k/\log P).
\end{equation}
\end{lemma}

\begin{proof}
Throughout the following argument, all implied constants in our asymptotic notation are permitted to depend on $k, \eta,  w$. Bear in mind that $\eta \le \eta_k$ is small.

From the definition
\begin{equation*}
\sum_n \nu(n) = \sum_{\substack{x \in S(P; P^\eta) \\ x \equiv \xi \mmod W}} x^{k-1} + O(1).
\end{equation*}
We obtain from the start of the proof of \cite[Lemma 5.4]{Vau1989} the fact that if $m \le P$ then
\begin{equation}\label{number of smooths}
\sum_{\substack{x \in S(m; P^\eta) \\ x \equiv \xi \mmod W}} 1 = \frac1W \sum_{x \in S(m; P^\eta)}1 + 
O \Bigl(  \frac P{\log P} \Bigr).
\end{equation}
Now partial summation and Lemma \ref{dense smooths} yield
\begin{align*}
\sum_n \nu(n) &= W^{-1} P^k \rho \brac{1/\eta} - \int_{P^{1/2}}^P \frac{k-1}W t^{k-1} \rho \Bigl( \frac{ \log t}{\eta \log P} \Bigr) \d t 
\\& \qquad + O \Bigl( \frac{P^k}{ \log P} \Bigr)
\end{align*}
so, by the mean value theorem and the boundedness of $\rho'$, it remains to show that
\[
\int_{P^{1/2}}^P kt^{k-1} \rho \Bigl( \frac{\log t}{\eta\log P} \Bigr) \d t = P^k \rho \brac{1/\eta} + O \Bigl( \frac{P^k}{\log P} \Bigr).
\]
Integration by parts gives
\begin{align*}
\int_{P^{1/2}}^P kt^{k-1} \rho \Bigl( \frac{\log t}{\eta \log P} \Bigr) \d t &= P^k \rho \brac{1/\eta} - P^{k/2} \rho(1/(2\eta))
\\ &\qquad - \int_{P^{1/2}}^P \frac{t^{k-1}}{\eta \log P} \rho' \Bigl( \frac{\log t}{\eta\log P} \Bigr) \d t,
\end{align*}
and the estimate now follows from the boundedness of $\rho, \rho'$.
\end{proof}

\begin{lemma} [Density transfer]\label{smooth density transfer} For $N$ large in terms of $k$, $\eta$, $w$ and $\del$ we have
\begin{equation}\label{A1 density}
\sum_{n \in A_1} \nu(n) \gg_{\eta,k} \delta^k \sum_n \nu(n).
\end{equation}
\end{lemma}

\begin{proof}
We employ \eqref{smooth A1 size} in conjunction with \eqref{number of smooths} to conclude that
\begin{align*}
&\hash \Bigset{ x \in S(P;P^\eta) :  x \equiv \xi \bmod W, \quad \frac{x^k - \xi^k}{kW} \in A_1, \quad x > Z} \\
& \ge |A_1| - ZW^{-1} - 1 \\
&  \geq  \tfrac{\delta }{2} \hash\set{x \in S(P; P^\eta) : x \equiv \xi \bmod W} - ZW^{-1} - O_{\eta, \delta, w}\brac{ N(\log N)^{-1}}  \\
& \ge \tfrac{\delta}{2W} |S(P; P^\eta)| -ZW^{-1} - O_{\eta, \delta, w}\brac{N(\log N)^{-1}}.
\end{align*}
Choosing
$$
Z = \frac \delta 4 |S(P;P^\eta)|
$$
furnishes
\begin{align*}
\sum_{n \in A_1} \nu(n) & \ge \tfrac{(\delta/4)^k}{W} |S(P; P^\eta)|^k - O_{\eta,\delta, w}\brac{N
^k(\log N)^{-1}}.
\end{align*}
Using Lemma \ref{dense smooths} and recalling \eqref{X and P} we obtain
\begin{align*}
\sum_{n \in A_1} \nu(n) & \ge W^{-1} (\del P \rho(1/\eta)/4)^k  - O_{\eta,\delta, w}\brac{N^k(\log N)^{-1}}
\\ & \ge (k (\rho(1/\eta)/4)^k) \cdot \delta^k X - O_{\eta,\delta, w}\brac{N^k(\log N)^{-1}}.
\end{align*}
Taking $N$ sufficiently large, an application of \eqref{smooth normalisation} completes the proof.\end{proof}
The following two ingredients are established in Appendices \ref{AppendixB} and \ref{AppendixC}.

\begin{lemma}[Fourier decay] \label{smooth decay} We have
\begin{equation} \label{FourierDecayIneq}
\norm{\frac{\hat \nu}{\norm{\nu}_1} - \frac{\hat{1}_{[X]}}{X} }_\infty \ll_{\eta}  w^{-1/k}.
\end{equation}
\end{lemma}

\begin{lemma}[Restriction estimate] \label{smooth restriction} There exists $s_0(k)$ such if $s \ge s_0(k)$ then
\begin{equation*}
\begin{split}
\sup_{|\phi| \le \nu} \int_{\T} \abs{\hat{\phi}(\alpha)}^{s - 10^{-8}} \intd\alpha \ll_{\eta, k} \norm{\nu}_1^{s- 10^{-8}}X^{-1}.
\end{split}
\end{equation*}
Moreover, we may take $s_0(2) = 5$, $s_0(3) = 8$ and $s_0(k)$ satisfying \eqref{s0 size}. 
\end{lemma}

\begin{proof}[Proof of Theorem \ref{smooth non linear sarkozy generalisation}] We employ Theorem \ref{smooth pseudorandom} with majorant $\nu$ given by \eqref{nu defn}, homogeneous set $B_1 \subset S(X^{1/k}; X^\eta)$ given by \eqref{B1 defn}, and function $f =  \nu 1_{A_1}$ (recall \eqref{A1 defn}). It is first necessary to check that these choices satisfy the hypotheses of Theorem \ref{smooth pseudorandom}.  

By Lemma \ref{smooth restriction}, the function $\nu$ satisfies a $(s+t-10^{-8})$-restriction estimate with constant $K = O_{\eta, k}(1)$.  Let $c_{\eta,k}$ denote the implied constant in \eqref{A1 density} and set $\tilde{\delta} := c_{\eta,k} \delta^k$.  Theorem \ref{smooth pseudorandom} guarantees the existence of a positive constant 
\begin{equation}\label{theta dependence}
\theta = \theta(\eta,\tilde\delta, M, K) 
\end{equation}
such that provided $\nu$ has Fourier decay of level $\theta$ and $\norm{f}_1 \geq \tilde\delta \norm{\nu}_1$ we may conclude that \eqref{smooth pseud conclusion} holds.  Taking 
$$
w = C_{\eta} \theta^k
$$
guarantees sufficient Fourier decay, by Lemma \ref{smooth decay}.  We note that this choice of $w$ satisfies $w \ll_{\eta,  \delta, M} 1$, as can be checked by unravelling the dependencies in \eqref{theta dependence}. We obtain $\norm{f}_1 \geq \tilde\delta \norm{\nu}_1$ via Lemma \ref{smooth density transfer}.  This requires us to take $N$  sufficiently large in terms of $k, \eta, w$ and $\delta$.  By our choice of $w$, this is ensured if $N$ is sufficiently large in terms of $\eta,  \delta$ and $M$ (as we may assume).

Applying Theorem \ref{smooth pseudorandom} and \eqref{smooth normalisation} yields
\[
T_1(\nu 1_{A_1}; B_1) \gg_{\eta, \delta, M} \|\nu\|_1^s X^{\frac t k -1}\gg_{\eta,\delta,M} X^{s+\frac t k -1}.
\]
By \eqref{smooth AA_1} and the bound $\|\nu\|_\infty \ll_{\eta,\del,M} N^{k-1}$, we finally have
\[
T_k(A;B) \ge T_1(A_1; B_1) \ge \| \nu \|_\infty^{-s} T_1(\nu 1_{A_1};  B_1) \gg_{\eta, \del,M} N^{s+t-k}.
\]
\end{proof}

\section{Deducing partition regularity}\label{smooth colouring section}

In this final section of this part of the paper we prove a finitary version of Theorem \ref{intro main theorem}.

\begin{theorem}[Smooth finitary colouring result]\label{smooth partial colouring result}
Define $s_0(k)$ as in Theorem \ref{intro main theorem}, and let $s \ge s_0(k)$. Let $c_1, \dots, c_s \in \Z\setminus \set{0}$ and suppose that $\sum_{i \in I} c_i = 0$ for some non-empty $I$.  Then, for any $r \in \N$, there exists $N_0 \in \N$ such that the following holds: for any $N \geq N_0$, if we have a finite colouring of the $N^\eta$-smooth numbers in $[N]$
$$
S(N;N^\eta) = C_1\cup \dots \cup C_r,
$$
then there exists a colour $i \in [r]$ and distinct $x_1, \dots, x_s \in C_i$ solving \eqref{kth power rado eqn}. 
\end{theorem}

\subsection{The inductive base: one colour}

As in \S\ref{colouring section}, given functions $f_1, \dots, f_s : \Z \to \C$ with finite support, define the counting operator
\begin{align*}
T(f_1, \dots, f_s) := \sum_{c_1x_1^k + \dots + c_sx_s^k = 0}f_1(x_1) f_2(x_2) \dotsm f_s(x_s)
\end{align*}
and write $T(f)$ for $T(f, f, \dots, f)$.

It follows from Theorem \ref{unrestricted lower bound} that there exist $\eta = \eta(k) > 0$, $N_1 = N_1(\eta, k, \vc) \in \N$ and $c_1= c_1(\eta, k, \vc) > 0$ such that for $N \geq N_1$ and we have
$$
T(1_{S(N;N^\eta)}) \geq c_1 N^{s-k}.
$$
 By Lemma \ref{trivial count}, the number of trivial solutions in $S(N;N^\eta)$ is $o(N^{s-k})$, so there must be at least one non-trivial solution  $(x_1,\ldots,x_s) \in S(N; N^\eta)^s$ to \eqref{kth power rado eqn} for $N$ sufficiently large in terms of $\eta$, $k$, $s$ and $\vc$.
The base case follows.

\subsection{The inductive step}

Let $S(N;N^\eta)= C_1\cup \dots \cup C_r$.  Re-labelling indices, we may assume that $C_r$ is the largest colour class, so that 
\begin{equation}\label{smooth lower bound on C_r}
|C_r| \geq |S(N;N^\eta)|/r  .
\end{equation}
We split our proof into two cases depending on the properties of $C_r$.
\subsubsection{The inhomogeneous case}  Let $M := N_0(r-1)$ be the quantity whose existence is guaranteed by our inductive hypothesis.  We may assume that $N \geq M^{1/\eta}$, so every element of $[M]$ is $N^\eta$-smooth.  First let us suppose that $C_r$ is not  $M$-homogeneous in $S(N;N^\eta)$.  
Consequently there exists $q \in S(N/M;N^\eta)$ such that
\begin{equation}\label{smooth inhomogeneous assumption}
\begin{split}
C_r \cap q \cdot [M]= \emptyset .
\end{split}
\end{equation}

For $i = 1, \dots, r-1$ let us define
$$
C_i' := \set{x \in S(M; M^\eta) : qx \in C_i}.
$$
Then it follows from \eqref{smooth inhomogeneous assumption} that $C_1' \cup \dots \cup C_{r-1}' = S(M;M^\eta)$.  By the induction hypothesis, there exist distinct elements of some $C'_i$ which solve \eqref{kth power rado eqn}. 
Since this equation is homogeneous, 
we obtain a non-trivial solution in $C_i$ by multiplying the equation through by $q^k$.
\subsubsection{The homogeneous case}
We now assume that $C_r$ is  $M$-homogeneous in $S(N;N^\eta)$.  We apply Theorem \ref{smooth non linear sarkozy generalisation}, taking $A = B =  C_r$. 
 By \eqref{smooth lower bound on C_r} the density of $A$ in $S(N;N^\eta)$ is at least $\trecip{r}$.  Theorem \ref{smooth non linear sarkozy generalisation} then implies that, provided $N \geq N_0(\eta,1/r, M)$ we have
$$
T(1_{C_r})\geq c_0(\eta,1/r,  M) N^{s-k}.
$$
By Lemma \ref{trivial count}, the number of solutions in $S(N; N^\eta)$ with two or more  coordinates equal is $o(N^{s-k})$, hence taking $N$ sufficiently large yields at least one non-trivial solution in $C_r$. We note that a quantity dependent on the tuple $(\eta, 1/r,  M)$ is ultimately dependent only on $\eta$ and $r$, by the definition of $M$.  The induction step thereby follows, completing the proof of Theorem \ref{smooth partial colouring result}.

\part{Supersmooths and shifted squares}\label{super sat part}

In this part we establish Rado's criterion for a linear equation in logarithmically-smooth numbers (Theorem \ref{super smooths}).  Furthermore, we show how a direct application of the transference principle yields a supersaturated version of this result, and analogously for a linear equation in the set of squares minus one (Theorem \ref{shifted squares}).  Both of these results are established without recourse to properties of homogeneous sets.  This reflects the fact that supersmooths and shifted squares possess subsets which can be projectively transformed to obtain equidistribution in congruence classes to small moduli, ruling out possible local obstructions to partition regularity---obstructions which must be surmounted when working with perfect squares and higher powers.  This phenomenon manifests itself when massaging the perfect powers to obtain equidistribution; this can be done, but requires an \emph{affine} transformation, as opposed to a projective one. Unfortunately, a typical equation satisfying Rado's criterion is only projectively invariant, so the methods of this part do not succeed in establishing partition regularity for equations in perfect powers.

\section{Modelling a pseudorandom partition with a colouring}

As described above, the proofs of Theorems \ref{shifted squares} and \ref{super smooths} proceed by first passing to a subset of the sparse arithmetic set of interest (supersmooths or shifted squares).  We then projectively transform this subset to obtain a set which is well distributed in arithmetic progressions to small moduli.  We can then define a weight $\nu : [N] \to \infty$ supported on our equidistributed set which has nice pseudorandomness properties.  

Given a finite colouring of our original arithmetic set, the above procedure induces a finite partition of our pseudorandom weight function into non-negative functions $f_i$, so that
$$
\nu = \sum_{i} f_i.
$$
Deducing supersaturation then amounts to showing that the count of solutions to our equation weighted by some $f_i$ is within a constant factor of the maximum possible.

The main tool in deriving this lower bound is to model the $f_i$ with functions $g_i$ whose sum dominates the indicator function of the interval $1_{[N]}$.   It is a short step to show that, in essence, we may assume that the $g_i$ correspond to indicator functions of a colouring of $[N]$.  For such colourings there is already a supersaturation result in the literature due to Frankl, Graham and R\"{o}dl \cite[Theorem 1]{FGR}.  Employing this theorem and then (quantitatively) retracing our steps yields Theorems \ref{shifted squares} and \ref{super smooths}.

In this section we establish the modelling part of the above procedure: non-negative functions $f_i$ with pseudorandom sum $\sum_i f_i$ have approximants $g_i$ whose sum dominates the constant function $1_{[N]}$.  This `transference principle'\footnote{This is also referred to as a `dense model' or `bounded approximation'  lemma in the literature.} for colourings is based on Green's transference principle for dense sets \cite{greenprimes}, as exposited in \cite{FourVariants}.  We recall the concepts of Fourier decay and $p$-restriction given in Definitions \ref{decay def} and \ref{restriction def}.

\begin{proposition}[Modelling lemma]\label{modelling lemma}
Suppose that $\nu: [N] \to [0, \infty)$ satisfies a $p$-restriction estimate with constant $K$, and has Fourier decay of level $1/M$ with $M \geq M_0(p, K)$ .  Then for any $f_i : [N] \to [0, \infty)$ with $f_1 + \dots + f_r =  \nu$ there exists $g_i :  [N] \to [0, \infty)$ such that $ g_1 + \dots + g_r = (1 + \trecip{\sqrt{M}})1_{[N]}$ and 
$$
\biggnorm{\frac{\hat{f}_i}{\norm{\nu}_1} - \frac{\hat{g}_i}{N}}_\infty \ll_{r,p,K} (\log M)^{-\recip{p+2}} \qquad (1\le i \le  r).
$$
\end{proposition}

Let $\kap, \eps > 0$ be parameters, to be determined later. In proving this result we utilise the large spectrum of $f_i$, which we take as 
\begin{equation}\label{large spectrum}
S_i := \set{\alpha \in \T : |\hat{f}_i(\alpha)| \geq \kap \norm{\nu}_1}.
\end{equation}
Define the Bohr set with frequencies $S := S_1 \cup \dots \cup S_{r-1}$ and width $\eps \le 1/2$ by
$$
B(S, \eps) := \set{n \in [-\eps N, \eps N] : \norm{n\alpha} \le \eps\quad( \forall \alpha \in S)}.
$$
Next define
\begin{equation*}
g_i := \frac{N f_i * 1_B * 1_B}{\norm{\nu}_1|B|^{2}} \qquad (1 \le i \le r-1),
\end{equation*}
where, for finitely supported $f_i$, we set
$$
f_1 * f_2(n) := \sum_{m_1+m_2 = n} f_1(m_1) f_2(m_2).
$$

We first estimate $\big|\frac{\hat{f}_i}{\norm{\nu}_1}- \frac{\hat{g}_i}{N}\big|$ with $i = 1, \dots, r-1$.  The key identity is 
$$
\widehat{f_1*f_2} = \hat{f}_1\hat{f}_2.
$$
If $\alpha \in \T\setminus S$ then by the definition \eqref{large spectrum} of the large spectrum  we have
$$
\biggabs{\frac{\hat{f}_i(\alpha)}{\norm{\nu}_1}- \frac{\hat{g}_i(\alpha)}{N}} = \biggabs{\frac{\hat{f}_i(\alpha)}{\norm{\nu}_1}}\abs{1 - \tfrac{\hat{1}_B(\alpha)^2}{|B|^2}} \le  2 \kap.
$$
If $\alpha \in S$, then for each $n \in B$ we have
$
e(\alpha n) = 1 + O(\eps).
$  
Hence
\[
\hat{1}_B(\alpha) = |B| + O(\eps|B|),
\]
and consequently
$$
\biggabs{\frac{\hat{f}_i(\alpha)}{\norm{\nu}_1}- \frac{\hat{g}_i(\alpha)}{N}} = \biggabs{\frac{\hat{f}_i(\alpha)}{\norm{\nu}_1}}\abs{1+\tfrac{\hat{1}_B(\alpha)}{|B|}}\abs{1 - \tfrac{\hat{1}_B(\alpha)}{|B|}} \ll \eps.
$$
Combining both cases gives 
\begin{equation*}
\biggnorm{\frac{\hat{f}_i}{\norm{\nu}_1} - \frac{\hat{g}_i}{N}}_\infty \ll \eps + \kap.
\end{equation*}
From this it is apparent we should choose $\kap = \eps$, which we do.

We will show that, for any $n$, the sum $\sum_{i \le r-1} g_i(n)$ is almost bounded above by 1.  By positivity and orthogonality, we have 
\begin{align*}
\sum_{i=1}^{r-1}g_i(n) & = N\norm{\nu}_1^{-1}|B|^{-2}\sum_{x+y+z = n}\sum_{i=1}^{r-1} f_i(x) 1_B(y) 1_B(z)\\
 &\le N\norm{\nu}_1^{-1} |B|^{-2}\sum_{x+y+z = n} \nu(x) 1_B(y) 1_B(z)\\ & = N\norm{\nu}_1^{-1}|B|^{-2}\int_\T \hat{\nu}(\alpha) \hat{1}_B(\alpha)^2 e(-\alpha n) \intd \alpha. 
\end{align*}
Inserting our Fourier decay assumption, and using Parseval, yields
\begin{align*}
\int_\T \frac{\hat{\nu}(\alpha)}{\norm{\nu}_1} \hat{1}_B(\alpha)^2 e(-\alpha n) \intd \alpha & \le \int_\T \frac{\hat{1}_{[N]}(\alpha)}{N} \hat{1}_B(\alpha)^2 e(-\alpha n) \intd \alpha + M^{-1}  \int_\T| \hat{1}_B(\alpha)|^2 \intd\alpha\\
& = N^{-1}\sum_{x+y+z = n} 1_{[N]}(x) 1_B(y) 1_B(z) +  M^{-1}|B| \\
& \le N^{-1}|B|^2 +   M^{-1}|B|.
\end{align*}
Following the proof of \cite[Lemmas A.1 and A.2]{FourVariants}, the restriction estimate yields a constant $C  = C(p,K) > 1$ such that 
$
|B| \geq \exp(-C\eps^{-p-2}) N.
$
Taking 
$
\eps = (2C/\log M)^{\recip{p+2}}
$
with $M$ large, we deduce that 
\begin{equation}\label{sum bound}
\sum_{i=1}^{r-1}g_i(n) \le 1 + 1/\sqrt{M},
\end{equation}
and that
$$
\biggnorm{\frac{\hat{f}_i}{\norm{\nu}_1} - \frac{\hat{g}_i}{N}}_\infty \ll \eps \ll_{p,K}   \brac{\log M}^{-\recip{p+2}} \qquad (1 \le i \le r-1).
$$

Having found suitable bounded approximants $g_i$ for $i = 1, \dots, r-1$, we define
$$
g_r := \brac{1+\trecip{\sqrt{M}}} 1_{[N]} - (g_1 + \dots + g_{r-1}).
$$
This is non-negative, by \eqref{sum bound}. Finally, we calculate how well $g_r$ approximates $f_r = \nu - (f_1 + \dots + f_{r-1})$. The triangle inequality gives
\begin{equation*}
\begin{split}
\biggnorm{\frac{\hat{f}_r}{\norm{\nu}_1} - \frac{\hat{g}_r}{N}}_\infty & \le\frac{1}{\sqrt{M}} +  \biggnorm{\frac{\hat{\nu}}{\norm{\nu}_1} - \frac{\hat{1}_{[N]}}{N}}_\infty + \sum_{i=1}^{r-1} \biggnorm{\frac{\hat{f}_i}{\norm{\nu}_1} - \frac{\hat{g}_i}{N}}_\infty\\
& \ll_{r,p, K} 1/(\log M)^{\recip{p+2}}.
\end{split}
\end{equation*}

\section{A pseudorandom Rado theorem}
Frankl, Graham and R\"{o}dl \cite[Theorem 1]{FGR} proved that if $c_1, \dots, c_s \in \Z\setminus\set{0}$ are such that $\sum_{i\in I } c_i = 0$ for some non-empty $I \subset [s]$, then for any $r$ there exists $c_0 = c_0(r,\vc) >0 $ such that in any $r$-colouring of $[N]$ there are at least $c_0 N^{s-1}$ monochromatic solutions $\vx$ to the equation 
\begin{equation*}
\begin{split}
c_1 x_1 + \dots + c_s x_s = 0.
\end{split}
\end{equation*}
The purpose of this section is to generalise this result from colourings to partitions of pseudorandom functions.
\begin{proposition}[Pseudorandom FGR]\label{pseudorandom fgr}
Let $c_1, \dots, c_s \in \Z\setminus\set{0}$ with $\sum_{i\in I } c_i = 0$ for some non-empty $I \subset [s]$.  For any $r, K \in \N$ there exist $N_0, M \in \N$ and $c_0 > 0$ such that for $N \geq N_0$ the following holds.   Let $\nu: [N] \to [0, \infty)$ satisfy a $(s-0.005)$-restriction estimate with constant $K$, and have Fourier decay of level $1/M$.  
Then for any partition $\nu = \sum_{i \le r} f_i$ with $f_i$ non-negative we have
\begin{equation}\label{pseud super sat}
\begin{split}
\sum_{i=1}^r \sum_{\vc \cdot \vx = 0} f_i(x_1) \dotsm f_i(x_s) \geq c_0 \norm{\nu}_1^s N^{-1}.
\end{split}
\end{equation}
\end{proposition}

We begin the proof of this theorem by generalising \cite{FGR} from colourings to bounded weights.
\begin{lemma}[Functional FGR]\label{functional fgr}
Let $c_1, \dots, c_s \in \Z\setminus\set{0}$ with $\sum_{i\in I } c_i = 0$ for some non-empty $I \subset [s]$.  For any $r$ there exists $N_0 \in \N$ and $c_0 >0 $ such that for $N \geq N_0$ and $g_1, \dots, g_r : [N] \to [0, \infty)$ with $\sum_i g_i \geq 1_{[N]}$ we have
\begin{equation*}
\begin{split}
\sum_{i = 1}^r \sum_{\vc\cdot\vx = 0} g_i(x_1) \dotsm g_i(x_s) \geq c_0 N^{s-1}.
\end{split}
\end{equation*}
\end{lemma}

\begin{proof}
By the pigeonhole principle, for each $x \in [N]$ there exists $i \in [r]$ such that $g_i(x) \geq 1/r$. Let $i$ be minimal with this property, and assign $x$ the colour $i$.  By the result of Frankl, Graham and R\"{o}dl, for some such choice of $i$ there are at least $c'_0 N^{s-1}$ tuples $\vx$ where each coordinate receives the colour $i$ and such that $\vc \cdot \vx = 0$.  It follows that
\begin{equation*}
\begin{split}
\sum_{\vc \cdot\vx = 0} g_i(x_1) \dotsm g_i(x_s) \geq r^{-s} c'_0 N^{s-1}.
\end{split}
\end{equation*}
\end{proof}

With this in hand, we proceed to prove Proposition \ref{pseudorandom fgr}.  Since $\nu$ satisfies a $(s-0.005)$-restriction estimate with constant $K$, and has Fourier decay of level $1/M$, we may apply the modelling lemma (Proposition \ref{modelling lemma}, provided $M \geq M_0(s, K)$ as we may assume) to conclude the existence of $g_i : [N] \to [0, \infty)$ with $\sum_i g_i = (1 + \trecip{\sqrt{M}}) 1_{[N]}$ and 
$$
\biggnorm{\frac{\hat{f}_i}{\norm{\nu}_1} - \frac{\hat{g}_i}{N}}_\infty \ll_{r,s,K} (\log M)^{-\recip{p+2}} \qquad (1\le i \le  r),
$$
where $p= s-0.005$. This also implies that
$$
\biggnorm{\frac{\hat{f}_i}{\norm{\nu}_1} - \frac{\hat{g}_i}{(1+M^{-1/2})N}}_\infty \ll_{r,s,K} (\log M)^{-\recip{p+2}} \qquad (1\le i \le  r).
$$
Applying Lemma \ref{functional fgr} (provided that $N \geq N_0(r, \vc)$, as we may assume) furnishes a colour class $i$ for which
$$
\sum_{\vc\cdot\vx = 0} g_i(x_1) \dotsm g_i(x_s) \gg_{r, \bc} N^{s-1}.
$$

Our assumption that $\sum_{i \in I} c_i = 0$ ensures that $s \geq |I| \geq 2$.  We may in fact assume that $s \geq 3$, for if $s = |I| = 2$ then Proposition \ref{pseudorandom fgr} is trivial.  Hence $(1+ M^{-1/2})1_{[N]}$ satisfies a $(s-0.005)$-restriction estimate with constant 1, and  majorises each $g_i$. Employing the generalised von Neumann lemma (Lemma \ref{gen von neu}), with $i$ as in the previous paragraph, we deduce that
$$
\frac{N}{\norm{\nu}_1^s}\sum_{\vc\cdot\vx = 0} f_i(x_1) \dotsm f_i(x_s) \geq c_0(r, \bc)   - O_{r, \bc,K}\brac{  \brac{\log M}^{-\frac{1}{400s}}}.
$$
Assuming that $M \geq M_0(r,\bc,K)$ completes the proof of Proposition \ref{pseudorandom fgr}.

\section{Supersaturation for shifted squares}

In this section we relate a colouring of the shifted squares to a partition of a pseudorandom majorant $\nu$ satisfying the hypotheses of Proposition \ref{pseudorandom fgr}, and thereby prove Theorem \ref{shifted squares}.  As in \S\ref{W trick for squares}, we accomplish this through the $W$-trick for squares.


Define $W$ by \eqref{Wsquaredef}, where $w = w(\vc, r)$ is a constant to be determined.
Let
$$
S' := \set{\trecip{2}Wx^2 + x :   x \in \bN, (Wx+1)^2   \le N },
$$
so that $S' \subset [N']$, where $N' := N/(2W)$. If $c$ is an $r$-colouring of the squares minus one, we induce an $r$-colouring of $S'$ via
$$
c'(\trecip{2}Wx^2 + x) := c\brac{(Wx + 1)^2 - 1}.
$$
Let $S_1', \dots, S_r'$ denote the induced colour classes. From the definition of $S'$ and the homogeneity of the equation, we see that the left-hand side of \eqref{shifted square count} is at least as large as
\begin{equation}\label{S' count}
\sum_{i=1}^r\sum_{\vc \cdot \vx = 0} 1_{S_i'}(x_1) \dotsm 1_{S_i'}(x_s).
\end{equation}

As in \eqref{square majorant}, define a weight function $\nu : [N']  \to [0, \infty)$ supported on $S'$ by 
\begin{equation} \label{ShiftedNu}
\nu(n) = \begin{cases} Wx + 1, & \text{if } n = \trecip{2}W x^2 + x \in S' \text{ for some } x \in \N \\
0, & \text{otherwise.}\end{cases}
\end{equation}
We reassure the reader that neither the constant term 1 nor the factor $W$ appearing above are necessary, but their presence is consistent with \eqref{square majorant} and \eqref{nu defn}. A calculation similar to \eqref{normalisation} gives
$$
\norm{\nu}_1 \gg  \norm{\nu}_\infty|S'| \gg_w \norm{\nu}_\infty|S\cap [N]|,
$$
where $S$ is the set of shifted squares as defined in the theorem.

We recall that $W$ ultimately depends only on $w = w(\vc, r)$.  Therefore, to show that \eqref{S' count} is of order $|S\cap[N]|^s N^{-1}$, and hence to prove Theorem \ref{shifted squares}, it suffices to establish that for $f_i =  \nu 1_{S_i'}$ we have
\begin{equation}\label{weighted N' count}
\sum_{i=1}^r \sum_{\vc\cdot\vx=0} f_i(x_1) \dotsm f_i(x_s) \gg_{r,\bc} \norm{\nu}_1^s (N')^{-1}.
\end{equation}

Appendices \ref{AppendixB} and \ref{AppendixC} yield the following.

\begin{lemma}[Fourier decay] \label{shifted decay}
We have 
$$
\biggnorm{\frac{\hat \nu}{\norm{\nu}_1} - \frac{\hat{1}_{[N']}}{N'} }_\infty \ll  w^{-1/2}.
$$
\end{lemma}

\begin{lemma}[Restriction estimate] \label{shifted restriction} We have
\begin{equation*}
\begin{split}
\sup_{|\phi| \le \nu} \int_{\T} \abs{\hat{\phi}(\alpha)}^{4.995} \intd\alpha \ll \norm{\nu}_1^{4.995}(N')^{-1}.
\end{split}
\end{equation*}
\end{lemma}

Let $K$ denote the absolute constant occurring in Lemma \ref{shifted restriction}, and let $N_0, M$ denote the parameters occurring in Proposition \ref{pseudorandom fgr}.   By Lemma \ref{shifted decay}, provided that we take $w = w(r, \vc)$ sufficiently large, we may apply Proposition \ref{pseudorandom fgr} with $f_i =  \nu 1_{S_i'}$ to conclude that \eqref{weighted N' count} holds.  This completes the proof of Theorem \ref{shifted squares}.

\section{Supersaturation for logarithmically-smooth numbers}

The proof of Theorem \ref{super smooths} follows in analogy with the argument of the prior section. The situation is somewhat simpler in this context, as there is no need to massage the set of smooths to exhibit sufficient pseudorandomness.  

Define the indicator function $\nu : [N] \to [0,\infty)$ of the $R$-smooth numbers in $[N]$ by
$$
\nu(x) := \begin{cases} 1 & \text{if } p\mid x \implies p \leq R\\
					0 & \text{otherwise.}\end{cases}
$$The relevant pseudorandomness  properties follow from work of Harper \cite{harper}. 

\begin{lemma}[{\cite[Theorem 2]{harper}}]
There exists an absolute constant $C > 0$ such that for $R \geq \log^C N$ we have
\begin{equation*}
\begin{split}
\sup_{|\phi| \le \nu} \int_{\T} \abs{\hat{\phi}(\alpha)}^{2.995} \intd\alpha \ll \norm{\nu}_1^{2.995} N^{-1}.
\end{split}
\end{equation*}
\end{lemma}

\begin{lemma}[{\cite[\S 5]{harper}}]\label{smooth fourier decay}
We have the Fourier decay estimate
$$
\biggnorm{\frac{\hat \nu}{\norm{\nu}_1} - \frac{\hat{1}_{[N]}}{N} }_\infty \ll  \frac{\log\log N}{\log R}.
$$
\end{lemma}

\begin{proof}[Proof of Theorem \ref{super smooths}]
We are assuming that $\sum_{i \in I} c_i = 0$ for some $I \neq \emptyset$, and this forces $s \geq 2$.  If $s = 2$ then we are counting monochromatic solutions to $x_1 - x_2 = 0$, for which we have the lower bound $|S(N; R)| \geq |S(N; R)|^2 N^{-1}$.

Let us therefore assume that $s \geq 3$.  Provided that $R \geq \log^CN$ we have that $\nu$ satisfies a $p = 2.995$ restriction estimate with constant $K = O(1)$.  Applying Proposition \ref{pseudorandom fgr} with these parameters, there exist $N_0, M, c_0 >0$ such that \eqref{smooth count} holds, provided that $\nu$ has Fourier decay of level $M^{-1}$.  This can be guaranteed on employing Lemma \ref{smooth fourier decay} and ensuring that
$$
R \geq \log^C N,
$$
where $C = C(r, \bc)$ is sufficiently large.
 \end{proof}

\part{Appendices}
\appendix

\section{Results on smooth numbers}

\begin{definition}[$R$-smooth numbers]
We say that a positive integer is \emph{$R$-smooth} if all of its prime divisors are at most $R$.  We denote the set of such numbers in the interval $[N]$ by
$$
S(N; R) := \set{ n \in [N] : p\mid n \implies p \le R}. 
$$
\end{definition}

The following estimate was proved by de Bruijn; see \cite[Eq. (1.8)]{Granville}. Here $\rho: [0,\infty) \to (0,1]$ denotes the Dickman--de Bruijn $\rho$-function. Note that $\rho$ is decreasing and has bounded derivative.

\begin{lemma}\label{dense smooths}
We have
$$
|S(N ; N^\eta)| = \rho(1/\eta) N + O_\eta\brac{N/\log N}.
$$
In particular, there exists $N_0 = N_0(\eta)$ such that for $N \geq N_0$ we have
$$
|S(N; N^\eta)| \gg_\eta N.
$$
\end{lemma}

\begin{lemma}\label{rankin lemma}
There are at most $10^wNM^{-1/2}$ elements of $[N]$ divisible by a $w$-smooth number greater than $M$. 
\end{lemma}

\begin{proof}
It follows from Rankin's trick that the 
number of integers in $[N]$ divisible by a $w$-smooth number exceeding $M$ is at most
\begin{align*}
\sum_{\substack{m >M\\ \text{$m$ is $w$-smooth}}} \frac{N}{m} &\le
\sum_{\substack{\text{$m$ is $w$-smooth}}} \frac{N}{m} \left(\frac{m}{M}\right)^{1/2} 
= NM^{-1/2} \prod_{p\leq w} \brac{1 + \recip{p^{1/2}-1}}.
\end{align*}
The result follows on noting that $1 + \recip{p^{1/2} -1 } \leq 10$.
\end{proof}

Notice that if $W$ is a $w$-smooth positive integer divisible by the primorial $\prod_{p\leq w} p$, then every positive integer can be written in the form $\zeta(\xi + Wy)$ for a unique choice of a $w$-smooth positive integer $\zeta$ and a unique $\xi \in [W]$ with $(\xi, W) = 1$.

\begin{lemma} \label{greedy}
Let $W$ be a $w$-smooth positive integer divisible by the primorial $\prod_{p\leq w} p$.  For any sets $A \subset S \subset [N]$ with $|A| \geq \delta |S|$ and $|S| \geq  \eta N$, there exist a $w$-smooth number 
$
\zeta \ll_{\del, \eta, w} 1,
$ 
and $\xi \in [W]$ with $(\xi,W) = 1$, such that  
$$
 \hash\set{x \in \Z : \zeta( \xi + W x ) \in A} \geq \trecip{2}\delta\hash \set{x \in \Z :  \zeta( \xi + W x ) \in S}.
$$
\end{lemma}
\begin{proof}
For $\zeta, \xi \in \bN$ and $T \subseteq \Z$, write
$$
T_{\zeta, \xi, W} := \set{x \in T : x= \zeta( \xi + W y ) \text{ for some } y \in \Z}.
$$
Let $M = 4(\del \eta)^{-2} 10^{2w}$, so that $N10^w  M^{-1/2} = \frac \del 2 \eta N \le \frac \del 2 |S|$. By the remarks preceding this lemma, together with Lemma \ref{rankin lemma}, we have
\begin{align*}
\del |S| \le |A| &\le \sum_{\substack{\zeta \le M\\ \zeta \text{ is $w$-smooth}}}\sum_{\substack{\xi \in [W] \\ \ (\xi, W) = 1}} |A_{\zeta, \xi, W}| + N10^w  M^{-1/2}
\\&
\le  \sum_{\substack{\zeta \le M\\ \zeta \text{ is $w$-smooth}}}\sum_{\substack{\xi \in [W] \\ \ (\xi, W) = 1}} |A_{\zeta, \xi, W}| + \frac \del 2 |S|.
\end{align*}
Therefore
\[
\delta \sum_{\substack{\zeta \le M\\ \zeta \text{ is $w$-smooth}}}\sum_{\substack{\xi \in [W] \\ \ (\xi, W) = 1}} |S_{\zeta, \xi, W}| \le \del |S| \le 2\sum_{\substack{\zeta \le M \\ \zeta \text{ is $w$-smooth}}}\sum_{\substack{\xi \in [W] \\ \ (\xi, W) = 1}} |A_{\zeta, \xi, W}|,
\]
and the pigeonhole principle completes the proof.
\end{proof}

\begin{lemma}\label{changing smoothness}
For any $K \geq 1$ we have
$$
S(N; KN^\eta) - S(N;N^\eta) \ll_{K, \eta} \frac{N}{\log N}.
$$
\end{lemma}

\begin{proof}
By Lemma \ref{dense smooths}, we have
\[
\frac{S(N; KN^\eta) -  S(N; N^\eta)}N = \rho\Bigl( \frac{\log N}{\eta \log N + \log K} \Bigr) - \rho(1/\eta) + O(1/\log N).
\]
The estimate now follows from the mean value theorem, since $\rho'$ is bounded and
\[
\frac{\log N}{\eta \log N + \log K} - \frac1\eta \ll \frac1{ \log N}.
\]
\end{proof}

\section{The unrestricted count and mean values estimates}
\label{unrestricted}

Recall that $\eta$ is 1 if $k=2$ and a small positive constant if $k \ge 3$. The following is a consequence of the current state of knowledge in Waring's problem.
\begin{theorem}\label{unrestricted lower bound}
Let $c_1, \dots, c_s \in \Z\setminus\set{0}$ with $\sum_{i \in I}c_i = 0$ for some non-empty subset $I$ of $[s]$. Then, for $k \ge 2$, there exists $s_0(k) \in \bN$ such that if $s \geq s_0(k)$ and $N \ge N_0$ then
\[
\# \Bigset{\bx \in S(N;N^\eta)^s : \sum_{i= 1}^s c_i x_i^k = 0} \asymp_{\bc,\eta, k } N^{s-k}.
\]
Moreover, one can take $s_0(2) = 5$, $s_0(3) = 8$, and $s_0(k)$ satisfying \eqref{s0 size}.
\end{theorem}

The $k=2$ case was known to Hardy and Littlewood. In an influential paper, Kloosterman \cite{kloosterman} opens with a discussion of this, then adapts the Hardy--Littlewood method to address the quaternary problem. Details of a proof may be found in \cite[Ch. 8]{davenport}. 

As we cannot find the precise statement that we require for $k \ge 3$ in the literature, we outline a proof below. The conclusion largely follows from the earlier techniques of Vaughan and of Wooley \cite{Vau1989, VW1991, Woo1992}, but we find it convenient to also draw material from other sources. Indeed, the aforementioned articles on Waring's problem involve a combination of smooth and full-range variables, so for our lower bound the results cannot be imported directly. Theorem \ref{unrestricted lower bound} is an indefinite version of a special case of \cite[Theorem 2.4]{DS2016}; the latter is more profound, as it tackles a more challenging smoothness regime. One approach would be simply to imitate the proof of that theorem, until needing to treat the local factors---this is approximately what we do below.

\begin{proof} Let $k \ge 3$, and let $\eta = \eta_k$ be a small positive constant. By orthogonality, our count is
\[
\int_0^1 g_1(\alp) \cdots g_s(\alp) \d \alp,
\]
where 
\[
g(\alp) = \sum_{x \in S(N; N^\eta)} e(\alp x^k), \qquad g_i(\alp) = g(c_i \alp) \quad (1 \le i \le s).
\]
Let $A \ge A_0(k)$, and put $Q = (\log N)^A$. Now perform a Hardy--Littlewood dissection \cite{vaughan97}: define major arcs
\[
\fM = \bigcup_{ \substack{ 0 \le a < q \le Q \\ (a,q) = 1} } \fM(q,a), 
\qquad \fM(q,a) = \{ \alp \in [0,1]: | q \alp - a| \le QN^{-k} \}
\]
and minor arcs $\fm = [0,1] \setminus \fM$. It follows from \cite[Lemma 8.6]{DS2016}, by slightly adjusting the parameters therein to allow for constant multiples, that
\[
\int_\fm |g_i(\alp)|^s \d \alp = c_i^{-1} \int_{c_i \fm} |g(\bet)|^s \d \bet \ll N^{s-k} Q^{-c},
\]
for some $c = c(k) > 0$. Therefore
\[
\int_0^1 g_1(\alp) \cdots g_s(\alp) \d \alp = \int_\fM g_1(\alp) \cdots g_s(\alp) \d \alp + o(N^{s-k}).
\]

First we prune our major arcs down to a lower height. Set $Q_1 = \sqrt{\log N}$. Let
\[
\fN = \bigcup_{\substack{0 \le a < q \le Q_1 \\ (a,q) = 1}} \fN(q,a), \qquad \fN(q,a) = \{ \alp \in [0,1]: |q \alp - a| \le Q_1 N^{-k}  \},
\]
and put $\fn = [0,1] \setminus \fN$. Let $\alp \in \fM(q,a)$ with $0 \le a < q \le Q$ and $(a,q) = 1$ and, by Dirichlet's approximation theorem \cite[Lemma 2.1]{vaughan97}, choose relatively prime $r \in \bN$ and $b \in \bZ$ such that $r \le 2Q$ and $|r c_s \alp - b| \le (2Q)^{-1}$. The triangle inequality gives
\[
\Bigl| \frac a q - \frac b{rc_s} \Bigr| \le \frac Q{qN^k} + \frac1 {2 rc_s Q} < \frac1{qrc_s},
\]
so $\frac a q = \frac b {rc_s}$. As $(a,q) = 1$ and $(rc_s, b) \ll 1$, we have $q \asymp r$, $|rc_s \alp - b| \asymp |q \alp - a|$, and it now follows from \cite[Lemma 8.5]{VW1991} that
\[
g_s(\alp) \ll_\eps q^\eps N (q + N^k | q \alp - a| )^{-1/k} + N \exp(-c \sqrt{\log N}  ) (1 + N^k |\alp - a/q|),
\]
where $c = c(A, \eta)$ is a small positive constant. In particular, if $\alp \in \fM \setminus \fN$ then 
\[
g_s(\alp) \ll  N Q_1^{\eps - 1/k}.
\]
Furthermore, the sharp mean value estimate \cite[Theorem 2.3]{DS2016} implies
\begin{equation} \label{dsmv}
\int_0^1 |g_i(\alp)|^{s-0.1} \d \alp \ll N^{s-0.1-k} \qquad (1 \le i \le s).
\end{equation}
Using H\"older's inequality, we now obtain
\[
\int_{\fM \setminus \fN} |g_1(\alp) \cdots g_s(\alp)| \d \alp = o(N^{s-k}),
\]
and so
\begin{equation} \label{pruned}
\int_0^1 g_1(\alp) \cdots g_s(\alp) \d \alp = \int_\fN g_1(\alp) \cdots g_s(\alp) \d \alp + o(N^{s-k}).
\end{equation}

For $q \in \bN$, $a \in \bZ$ and $\bet \in \bR$, define
\[
S(q,a) = \sum_{x \le q} e_q(ax^k), \qquad w(\bet) = \sum_{N^{\eta k} < m  \le N^k} \frac 1k m^{\frac1k - 1} \rho\Bigl( \frac{\log m}{\eta k \log N} \Bigr) e(\bet m)
\]
and
\[
W( \alp, q, a) = q^{-1} S(q,a) w(\alp - a/q),
\]
where as before $\rho$ denotes the Dickman--de Bruijn $\rho$-function. Next, we apply \cite[Lemma 5.4]{Vau1989} to $c_i \alp$, for $1 \le i \le s$ and $\alp \in \fN(q,a) \subset \fN$, where $0 \le a < q \le Q_1$ and $|q \alp - a| \le Q_1 N^{-k}$. With $c'_i = c_i/(c_i,q)$ and $q_i = q/(c_i,q)$, this gives
\[
g_i(\alp) = W(c_i \alp, q_i, c'_i a) + O( (\log N)^{-1/2}) ,
\]
and furthermore
\begin{equation} \label{MajorArcBound}
W(c_i \alp, q, c_i a) = W(c_i \alp, q_i, c_i' a) \ll q^{-1/k} \min(N, | \alp - a/q |^{-1/k}).
\end{equation}
By \eqref{dsmv} and \eqref{pruned}, together with H\"older, we now have
\begin{align*}
&\int_0^1 g_1(\alp) \cdots g_s(\alp) \d \alp 
\\ & \quad =
\sum_{q \le \sqrt{\log N}} q^{-s}  \sum_{\substack{a \le q \\ (a,q)=1}} 
\int_{|\bet| \le \frac{\sqrt{\log N}} {q N^k}}
\Bigl(\prod_{i \le s} S(q, c_i a) 
 w(c_i \bet) \Bigr) \d \bet
+ o(N^{s-k}).
\end{align*}

The bound \eqref{MajorArcBound} enables us to extend the integral to $[-1/2, 1/2]^s$ and then the outer sum to infinity with $o(N^{s-k})$ error, as is usual for a major arc analysis \cite{davenport, vaughan97}. We thus obtain
\[
\int_0^1 g_1(\alp) \cdots g_s(\alp) \d \alp = \fS J  + o(N^{s-k}),
\]
where 
\[
\fS = \sum_{q = 1}^\infty \sum_{\substack{a \le q \\ (a,q) = 1}} q^{-s} S(q, c_1 a) \cdots S(q, c_s a)
\]
and
\[
J = \int_{[-1/2,1/2]^s} w(c_1 \bet) \cdots w(c_s \bet) \d \bet.
\]

As discussed in \cite[Ch. 8]{davenport}, the singular series is a product of $p$-adic densities,
\[
\fS = \prod_p \chi_p,
\]
and is strictly positive if and only if $\chi_p > 0$ for all $p$. The positivity of the $p$-adic densities $\chi_p$ follows from the assumption that $\sum_{i \in I} c_i = 0$ for some non-empty $I \subseteq [s]$: one takes a non-trivial solution in $\{ 0, 1 \}^s$, and this is a non-singular $p$-adic zero.

Our final task is to show that $J \asymp N^{s-k}$. By orthogonality
\[
J = k^{-s} \sum_{\substack{\bm \in (N^{\eta k}, N^k]^s \\
\bc \cdot \bm = 0}} \prod_{i \le s} m_i^{\frac1k -1} \rho \Bigl( \frac{ \log m_i } { \eta k \log N} \Bigr).
\]
With $c > 0$ small, we have the crude lower bound
\[
J \gg N^{s(1-k)} \sum_{ \substack{\bm \in (cN^k, N^k]^s \\ \bc \cdot \bm = 0}} 1 \gg N^{s(1-k)} (N^k)^{s-1} = N^{s-k},
\]
since the $c_i$ are not all of the same sign. We also have the complementary upper bound
\[
J \ll N^{s(1-k)} \sum_{ \substack{\bm \in [1, N^k]^s \\ \bc \cdot \bm = 0}} 1 \ll N^{s-k}.
\]
\end{proof}

\begin{remark}
By working harder, we could have obtained a main term $\lam N^{s-k}$, for some positive constant $\lam = \lam(\bc)$, similarly to Drappeau---Shao \cite{DS2016}.
\end{remark}

We also need the following bounded restriction inequalities.

\begin{lemma} \label{BabyRestriction} Let
\[
f: [N] \to \{z \in \bC: |z| \le 1 \}.
\]
If $p > 4$ then
\[
\int_\bT \Bigl|\sum_{x \in S(N;N^\eta)} f(x) e(\alp x^2) \Bigr|^p \d \alp \ll _p N^{p-2}.
\]
For $k \ge 3$, there exists $s_0(k)$ such that if $s \geq s_0(k)$ then
\[
\int_\bT \Bigl|\sum_{x \in S(N;N^\eta)} f(x) e(\alp x^k) \Bigr|^{s - 10^{-8}} \d \alp \ll  N^{s - 10^{-8}-k}.
\]
Moreover, one may take $s_0(3) = 8$, and $s_0(k)$ satisfying \eqref{s0 size}.
\end{lemma}

\begin{proof} The quadratic statement is a direct consequence of \cite[Eq.\ (4.1)]{Bou1989}. 
Assuming for the time being that $k \ge 4$, write $2t$ for the smallest even integer greater than or equal to the integer $s_0(k)$ appearing in Theorem \ref{unrestricted lower bound}. Note that modifying $s_0(k)$ by adding a constant does not affect the veracity of \eqref{s0 size}, and so we will prove the statement for $s\ge s_0(k)+2$ in this case. 

By orthogonality, the triangle inequality and Theorem \ref{unrestricted lower bound}, we have
\begin{align*}
&\int_\bT \Bigl|\sum_{x \in S(N;N^\eta)} f(x) e(\alp x^k) \Bigr|^{2t} \d \alp 
\\ &\qquad \le
\# \Bigset{(\bx, \by) \in S(N;N^\eta)^t \times S(N;N^\eta)^t : \sum_{i \le t} x_i^k = \sum_{i \le t} y_i^k} 
 \ll_{t, \eta} N^{ 2t - k}.
\end{align*}
The trivial estimate $\Bigl|\sum_{x \in S(N; N^\eta)} f(x) e(\alp x^k) \Bigr| \le N$ completes the proof when $k \geq 4$.

For $k = 3$ we require a more elaborate argument to prove that the precise value of $s_0(3) = 8$ is admissible. In particular, our approach relies on a `subconvex' mean value estimate of Wooley \cite{Woo1995}. Define $\phi: \bZ \to \bC$ by $\phi(n) = f(x)$ if $n = x^3$ for some $x \in S(N; N^\eta)$, and zero otherwise. Our objective is to show that
\[
\int_\bT | \hat \phi(\alp)|^{8-10^{-8}} \d \alp \ll N^{5-10^{-8}}.
\]

In the present appendix, we let $\del$ denote a parameter in the range
\[
0 < \del < 1,
\]
and consider the large spectra
\[
\cR_\del = \{ \alp \in \bT: |\hat \phi(\alp)| > \del N \}.
\]
By the dyadic pigeonholing argument in \cite[\S 6]{densesquares}, it suffices to prove that
\[
\meas(\cR_\del) \ll \frac1{\del^{8-10^{-7}}N^3}.
\]
By orthogonality, Wooley's estimate \cite[Theorem 1.2]{Woo1995} implies that
\[
\int_\bT | \hat \phi(\alp) |^6 \d \alp \ll N^{3.25 - 10^{-4}}.
\]
Thus, we may assume without loss of generality that
\begin{equation} \label{wmaBounded}
N^{10^{-5} - \frac18} < \del < 1.
\end{equation}
Indeed, if $\del \le N^{10^{-5} -\frac18}$ then 
\[
\meas(\cR_\del)  \le (\del N)^{-6} \int_{\cR_\del} |\hat \phi(\alp)|^6 \d \alp \ll (\del N)^{-6} N^{3.25-10^{-4}} \le \frac1{\del^{8-10^{-7}}N^3}.
\]

Let $\tet_1, \tet_2, \ldots, \tet_R$ be $N^{-3}$-separated points in $\cR_\del$. As $8-10^{-7}\ge 6.3$ it suffices to show that
\begin{equation} \label{STPcubicBounded}
R \ll \del^{-6.3}.
\end{equation}
Let $\mu(n) = 1$ if $n = x^3$ for some $x \in [N]$, and zero otherwise. For some $a_n \in \bC$ with $|a_n| \le 1$, we then have $\phi(n) = a_n \mu(n)$; this `throws away' smoothness. With $\gam = 3.1$, the calculation in \cite[\S 6]{densesquares} yields
\begin{equation} \label{CubicBourgain0}
\del^{2\gam} N^\gam R^2 \ll \sum_{r,r' \le R} | \hat \mu(\tet_r - \tet'_r)| ^\gam.
\end{equation}

Consider the value of $\tet = \tet_r - \tet'_r$ in the right-hand side of \eqref{CubicBourgain0}. Define a set of `minor arcs'
\[
\fn = \{ \alp \in \bT: |\hat \mu(\alp)| \le N^{10^{-8}+3/4} \}.
\]
In light of \eqref{wmaBounded}, the contribution from $\tet \in \fn$ to the right-hand side of \eqref{CubicBourgain0} is $o(\del^{2\gam} N^\gam R^2)$, and so
\begin{equation} \label{CubicBourgain01}
\del^{2\gam} N^\gam R^2 \ll \sum_{\substack{r,r' \le R:\\ \tet \notin \fn}} | \hat \mu(\tet_r - \tet'_r)| ^\gam.
\end{equation}

Next, suppose $\tet \in \bT \setminus \fn$, and fix a small $\eps > 0$. By \cite[Lemma 2.3]{wps}, there exist relatively prime $q \in \bN$ and $a \in \bZ$ such that 
\[
q \le N^{3/4}, \qquad |q \tet - a| \le N^{-9/4}
\]
and
\[
\hat \mu (\tet) \ll q^{\eps-\frac13} N (1 + N^3 |\tet- a/q|)^{-1/3}.
\]
With $C$ a large positive constant, put $Q= C+\del^{-9}$. The contribution to the right-hand side of \eqref{CubicBourgain01} from denominators $q > Q$ is $O(R^2 N^\gam Q^{\gam(\eps - \frac13)})$, which is negligible compared to the left-hand side. 

Hence
\begin{equation} \label{CubicBourgain02}
\del^{2\gam} R^2 \ll \sum_{1 \le r, r' \le R} G(\tet_r - \tet'_r),
\end{equation}
where
\[
G(\tet) = \sum_{q \le Q} \sum_{a=0}^{q-1} \frac{q ^{\gam(\eps - \frac 13)}}{ (1+ N^3 | \sin (\tet - a/q) |)^{\gam/3}}.
\]
The inequality \eqref{CubicBourgain02} is a cubic version of \cite[Eq. (4.16)]{Bou1989}. As $\gam(\eps - \frac13) > 1$, Bourgain's argument carries through, and yields \eqref{STPcubicBounded}.
\end{proof}

Finally, we need an upper bound on the number of trivial solutions.
\begin{lemma}\label{trivial count}
Let $k \geq 2$, and let $c_1, \ldots, c_s$ be non-zero integers summing to zero. Then there exists $s_0(k)$ such that if $s \geq s_0(k)$ then
\[
\# \Bigset{\bx \in S(N;N^\eta)^s : \sum_{i= 1}^s c_i x_i^k = 0 \text{ and } x_i = x_j \text{ for some } i \neq j} = o\brac{ N^{s-k}}.
\]
Moreover, one can take $s_0(2)=5$, $s_0(3) = 8$, and $s_0(k)$ satisfying \eqref{s0 size}.
 \end{lemma}

\begin{proof}
Let $s_0(k)$ be as in Lemma \ref{BabyRestriction}.  By the union bound, it suffices to prove an estimate of the required shape for the number of solutions with $x_{s-1} = x_s$.  In this case we are estimating
$$
\# \Bigset{x \in S(N;N^\eta)^{s-1} : \sum_{i = 1}^{s-2} c_i x_i^k + (c_{s-1} + c_s) x_{s-1}^k = 0}.
$$
It may be that $c_{s-1} + c_s = 0$, so we estimate the contribution from the $x_{s-1}$ variable trivially.  Using orthogonality and H\"older's inequality, it therefore suffices to prove that
\begin{equation}\label{little oh plus one}
\int_\T\Bigabs{ \sum_{x \in S(N;N^\eta)} e(\alpha x^k)}^{s-2} \intd \alpha = o(N^{s-1 - k}).
\end{equation}

Let $p = s_0(k) - 10^{-8}$.  When $s-2 \ge p$, the estimate \eqref{little oh plus one} follows from Lemma \ref{BabyRestriction} and the trivial estimate $\Bigabs{ \sum_{x \in S(N;N^\eta)} e(\alpha x^k)} \le N$. When $s-2 < p$, we apply H\"older's inequality and Lemma \ref{BabyRestriction} to obtain
\begin{align*}
\int_\T\Bigabs{ \sum_{x \in S(N;N^\eta)} e(\alpha x^k) }^{s-2} \intd \alpha &\leq \brac{\int_\T\Bigabs{ \sum_{x \in S(N;N^\eta) } e(\alpha x^k)}^{p} \intd \alpha}^{\frac{s-2}{p}} 
\\&  \ll  N^{s-2 - \frac{k(s-2)}{p}}.
\end{align*}
It remains to check that 
$
s-2 - \frac{k(s-2)}{p} < s- 1 - k,
$
or equivalently that
$
2 + p(1-\trecip{k}) < s.
$
Since $s > p$, this follows if $p/k \geq 2$, which we can certainly ensure without affecting the bound \eqref{s0 size}.
\end{proof}

\section{A generalised von Neumann lemma}

Recall the notion of $p$-restriction introduced in Definition \ref{restriction def}.

\begin{lemma}\label{restriction additivity}
Let $\nu_1, \nu_2 : [N] \to [0, \infty)$.  If both $\nu_1$ and $\nu_2$ satisfy a $p$-restriction estimate with constant $K$, then so does $\nu_1 + \nu_2$.
\end{lemma}

\begin{proof}
Let $|\phi| \leq \nu_1 + \nu_2$.  Then $\phi = \psi\times \theta$, where $\psi : [N] \to [0, \infty)$ satisfies $\psi \leq \nu_1 + \nu_2$ and $\theta : [N] \to \C$ satisfies $|\theta| \leq 1$. Put $\psi_1 := \min\set{\psi, \nu_1}$ and $\psi_2 := \psi - \psi_1$. On setting $\phi_i := \psi_i \theta$, we have $\phi = \phi_1 + \phi_2$ with $|\phi_i| \leq \nu_i$.  Applying the triangle inequality and restriction estimates for each $\nu_i$ gives
\begin{align*}
\bignorm{\hat{\phi}}_p & \leq \bignorm{\hat\phi_1}_p + \bignorm{\hat\phi_2}_p\\
& \leq (K/N)^{1/p} \brac{\norm{\nu_1}_1 + \norm{\nu_2}_1}.
\end{align*}
Positivity gives that $\norm{\nu_1}_1 + \norm{\nu_2}_1 = \norm{\nu_1 + \nu_2}_1$, and the result then follows on taking $p$th powers.
\end{proof}

\begin{lemma}\label{gen von neu on a summand}
Let $c_1, \dots, c_s \in \Z\setminus\set{0}$, $\delta \in (0,1)$ 
and suppose that $\nu_1, \dots, \nu_s : [N] \to [0, \infty)$ each satisfy a $(s- \delta)$-restriction estimate with constant $K$.  Then for any $|f_i| \leq \nu_i$ we have
\begin{equation}\label{holder to establish}
\Biggabs{\sum_{\vc\cdot\vx = 0} \frac{f_1(x_1)}{\norm{\nu_1}_1} \dotsm \frac{f_s(x_s)}{\norm{\nu_s}_1} }\leq \frac{K}{N}  \min_i\biggnorm{\frac{\hat{f}_i}{\norm{\nu_i}_1}}_\infty^{\delta}.
\end{equation}
\end{lemma}

\begin{proof}
We prove the upper bound with $i =1$, the remaining cases following by re-labelling indices. Let $p = s - \delta$. By orthogonality and H\"older's inequality, we have
\begin{align*}
\Bigl| \sum_{\vc\cdot \vx= 0} f_1(x_1) \dotsm f_s(x_s) \Bigr| & = \Bigl| \int_\T\hat{f}_1(c_1\alpha) \dotsm \hat{f}_s(c_s\alpha)\intd\alpha \Bigr|
\le \int_\T \Bigl|\hat{f}_1(c_1\alpha) \dotsm \hat{f}_s(c_s\alpha) \Bigr| \intd\alpha 
\\
& \leq \bignorm{\hat{f}_1}_\infty^{\delta}\bignorm{\hat{f}_1}_p^{1-\delta}\bignorm{\hat{f}_2}_p\dotsm\bignorm{\hat{f}_s}_p.
\end{align*}
Inequality \eqref{holder to establish} then follows from  our $p$-restriction assumption.
\end{proof}

\begin{lemma}[Generalised von Neumann]\label{gen von neu}
Let $c_1, \dots, c_s \in \Z\setminus\set{0}$, $\delta\in(0,1)$ 
and suppose that $\nu_i, \mu_i : [N] \to [0, \infty)$ each satisfy a $(s-\delta)$-restriction estimate with constant $K$.  Then for any $|f_i| \leq \nu_i$ and $|g_i| \leq \mu_i$ we have
\begin{multline*}
\Biggabs{\sum_{\vc\cdot\vx = 0} \brac{\frac{f_1(x_1)}{\norm{\nu_1}_1} \dotsm \frac{f_s(x_s)}{\norm{\nu_s}_1}\ -\ \frac{g_1(x_1)}{\norm{\mu_1}_1} \dotsm \frac{g_s(x_s)}{\norm{\mu_s}_1}}}\\ \leq \frac{sK}{N}  \max_i\biggnorm{\frac{\hat{f}_i}{\norm{\nu_i}_1} - \frac{\hat{g}_i}{
\norm{\mu_i}_1}}_\infty^{\delta}.
\end{multline*}
\end{lemma}

\begin{proof}
Let $p = s - \delta$.  By Lemma \ref{restriction additivity}, the weight
$$
\frac{\nu_i}{\norm{\nu_i}_1} + \frac{\mu_i}{\norm{\mu_i}_1}
$$
satisfies a $p$-restriction estimate with constant $K$ and majorises the difference 
$$
\frac{f_i}{\norm{\nu_i}_1} - \frac{g_i}{\norm{\mu_i}_1}.
$$
Observing that this weight has $L^1$ norm equal to two, the lemma follows on applying the telescoping identity 
$$
a_1\dotsm a_s - b_1\dotsm b_s = \sum_{i=1}^s (a_i - b_i)\prod_{j < i} a_j \prod_{j > i} b_j,
$$ 
together with Lemma \ref{gen von neu on a summand}.
\end{proof}

\section{Pointwise exponential sum estimates}
\label{AppendixB}

The primary objective of this section is to establish the Fourier decay estimates in Lemmas \ref{decay}, \ref{smooth decay} and \ref{shifted decay}.  Of these, Lemma \ref{smooth decay} concerns an exponential sum over smooth numbers.  
As before, put $R = P^\eta$, with $\eta = 1$ when $k=2$ and $\eta=\eta_k$ a small positive number when $k\ge 3$, and define $P$ and $X$ by \eqref{X and P}. Our weight function $\nu$ is defined by \eqref{nu defn}, with $k=2$ when dealing with Lemmas \ref{decay} and \ref{shifted decay} as well as $\xi=1$ in the latter scenario. This is consistent with \eqref{square majorant} and \eqref{ShiftedNu}. We assume throughout that $X$ is sufficiently large in terms of $w$.

Our goal is to prove the inequality \eqref{FourierDecayIneq}, using the Hardy--Littlewood circle method. More explicitly, we wish to show that if $\alp \in \bT$ then
\begin{equation} \label{DecayEst}
\frac{\hat \nu (\alp)}{\norm{\nu}_1} =\frac{\hat{1}_{[X]} (\alp)}{X}  +O_{\eta}  (w^{-1/k}).
\end{equation}
We treat the $k \ge 3$ and $k=2$ cases separately, as smooth numbers are used for the former. 

\subsection{Smooth Weyl sums}

We first consider the case $k \ge 3$, recalling that here we choose $\eta =\eta_k$ sufficiently small. The idea is to consider a rational approximation $a/q$ to $\alpha$; there will ultimately be four regimes to consider, according to the size of $q$. We begin with a variant of \cite[Lemma 5.4]{Vau1989}, which is useful for low height major arcs. Let
\[
S_{q,a} = \sum_{r \mmod q} e \Bigl( \frac a q  \cdot \frac{(Wr + \xi)^k - \xi^k}{kW} \Bigr), \qquad
I(\bet) = \int_0^X   e(\bet z) \d z.
\]

\begin{lemma} [First level] \label{major} Suppose $q \in \bN$ and $a \in \bZ$, with $q \le R/W$ and $\| q \alp \| =  |q \alp - a|$. Then
\[
 \hat \nu (\alp) = \rho(1/\eta) q^{-1} S_{q,a} I \Bigl(\alp - \frac a q \Bigr) + O_\eta \Bigl(\frac{P^k}{\log P}(q+ P^k \| q \alp \|) \Bigr).
\]
\end{lemma}

\begin{proof} 
The start of the proof of \cite[Lemma 5.4]{Vau1989} yields
\[
\sum_{\substack{x \in S(m; R) \\ x \equiv Wr + \xi \mmod Wq}} 1 = \frac1{Wq} \sum_{x \in S(m;R)}1 + O\Bigl( \frac{P}{\log P} \Bigr),
\]
valid for $r \in [q]$ and $m \le P$. Therefore 
\[
\sum_{\substack{x \in S(m; R) \\ x \equiv \xi \mmod W}} e \Bigl( \frac a q \cdot \frac{ x^k - \xi^k }{kW} \Bigr) 
= \frac {S_{q,a}}{Wq} \sum_{x \in S(m;R)}1 + O\Bigl( \frac{qP}{\log P} \Bigr).
\]
In particular, if $\alp(x)$ equals $e \Bigl( \frac a q \cdot \frac{ x^k - \xi^k }{kW} \Bigr)$ when $x \equiv \xi \mmod W$ is $R$-smooth and 0 otherwise, then
\[
\sum_{x \le m} \Bigl(\alp(x)- \frac{S_{q,a}}{Wq} \rho \Bigl(\frac {\log m}{\log R} \Bigr) \Bigr) \ll \frac{qP}{\log P}.
\]
By partial summation and the boundedness of $\rho'$, we also have
\begin{equation*}
\sum_{x \le m} \rho \Bigl( \frac{ \log x}{\log R} \Bigr) = m \rho \Bigl( \frac{ \log m}{\log R} \Bigr) + O\Bigl( \frac P{\log P} \Bigr) \qquad (1 \le m \le P),
\end{equation*}
and so
\[
\sum_{x \le m} \Bigl( \alp(x) - \frac{S_{q,a}}{Wq}  \rho \Bigl(\frac {\log x}{\log R} \Bigr) \Bigr)  \ll \frac{qP}{\log P}.
\]

Next, observe that with $\bet = \alp - a/q$ we have $|\bet| = q^{-1} \|q \alp \|$ and
\begin{align}
 \notag \hat \nu (\alp) &= 
\sum_{\substack{x \in S(P;R) \\ x \equiv \xi \mmod W}} e \Bigl( \frac a q \cdot \frac{ x^k - \xi^k }{kW} \Bigr) \phi(x) \\
\label{VaughanWay} &= \frac{ S_{q,a}}{Wq}  \sum_{x \le P} \rho \Bigl( \frac{ \log x}{ \log R} \Bigr) \phi(x) + E,
\end{align}
where 
\[
\phi(x) = x^{k-1} e \Bigl( \bet \frac{ x^k - \xi^k }{kW} \Bigr)
\]
and
\[
E = \sum_{x \le P} \Bigl(\alp(x)- \frac{S_{q,a}}{Wq} \rho \Bigl(\frac {\log x}{\log R} \Bigr) \Bigr) \phi(x).
\]
Partial summation gives
\[
E \ll  \frac{qP}{\log P} ( \| \phi \|_{L^\infty ([1,P])} + P \| \phi'\|_{L^\infty ([1,P])} ) \ll \frac{qP^k}{\log P} (1+ P^k | \bet |),
\]
and with the boundedness of $\rho'$ it also implies that
\[
\sum_{x \le P} \rho \Bigl( \frac{ \log x}{ \log R} \Bigr) \phi(x) = 
\rho(1/\eta) \sum_{x \le P} \phi(x) + O \Bigl( \frac{ P^k}{\log P} \Bigr).
\]
Meanwhile, Euler--Maclaurin summation \cite[Eq. (4.8)]{vaughan97} yields
\begin{align*}
\sum_{x \le P} \phi(x) &= \int_1^P \phi(x) \d x + O(P^{k-1} (1+ P^k | \bet |)) 
\\&= WI(\bet) + O(P^{k-1} (1+ P^k | \bet |)).
\end{align*}
Substituting these estimates into \eqref{VaughanWay} concludes the proof.

\end{proof}

We supplement this by bounding $S_{q,a}$ and $I(\bet)$.

\begin{lemma} \label{complete}
If $(q,a) = 1$ then $S_{1,0} = 1$, 
\begin{equation} \label{complete1}
S_{q,a} = 0 \qquad(2 \le q \le w)
\end{equation}
and
\begin{equation} \label{complete2}
S_{q,a} \ll q^{1-1/k}.
\end{equation}
\end{lemma}

\begin{proof}
Plainly $S_{1,0} = 1$, so let $q \ge 2$, and let $a \in \bZ$ with $(q,a) = 1$. The binomial expansion gives
\[
S_{q,a} =\sum_{r \mmod q} e_q\Biggl(a \sum_{\ell = 1}^k \frac{{k \choose \ell} W^{\ell - 1}}k \xi^{k-\ell} r^\ell \Biggr),
\]
and we note that $\frac{{k \choose \ell} W^{\ell - 1}}k \in \bZ$ ($1 \le \ell \le k$). Write $q = uv$, where $u$ is $w$-smooth and $(v,W) = 1$. Since $(u,v) = 1$, a standard calculation reveals that
\begin{equation} \label{decompose}
S_{q,a} = S_{u,a_1} S_{v,a_2},
\end{equation}
where $a_1 = av^{-1} \in (\bZ / u \bZ)^\times$ and $a_2 = au^{-1} \in (\bZ / v \bZ)^\times$ (see \cite[Lemma 2.10]{vaughan97}). 

Put $u = hu'$, where $h = (u, W/k)$. Representing $r\mmod q$ as $r = r_1 + u' r_2$, where $0 \le r_1 < u'$ and $0 \le r_2 < h$, gives
\begin{align*}
S(u,a_1) &= \sum_{\substack{0 \le r_1 < u' \\ 0 \le r_2 < h}}  e_{hu'} \Biggl(a_1 \sum_{\ell = 1}^k {k \choose \ell} \frac{{k \choose \ell} W^{\ell - 1}}k \xi^{k-\ell} (r_1 + u'r_2)^\ell \Biggr) \\
&= 
\sum_{r_1 = 0}^{u'-1}   e_{hu'} \Biggl(a_1 \sum_{\ell = 1}^k \frac{{k \choose \ell} W^{\ell - 1}}k \xi^{k-\ell} r_1^\ell \Biggr)\\
& \qquad  \cdot \sum_{r_2 = 0 }^{h-1} e_h 
\Biggl(a_1 \sum_{\ell = 1}^k \frac{{k \choose \ell} W^{\ell - 1}}k \xi^{k-\ell} (u')^{\ell - 1} r_2^\ell \Biggr).
\end{align*}
As $h$ divides $W/k$, the inner sum is
\[
\sum_{r_2 \mmod h} e_h (a_1 \xi^{k-1} r_2),
\]
which vanishes unless $h \mid a_1 \xi^{k-1}$. As $(h,a_1) = (h,\xi) = 1$, and as
\[
(u,W/k) = 1 \Leftrightarrow (u,W) = 1 \Leftrightarrow u =1,
\]
we conclude that
\begin{equation} \label{vanishing}
S_{u,a_1} = \begin{cases}
0 & \text{if } u \ne 1 \\
1 &\text{if } u =1.
\end{cases}
\end{equation}
Moreover, note that if $2 \le q \le w$ then $u =q$ and $v = 1$. Now \eqref{decompose} and \eqref{vanishing} complete the proof of \eqref{complete1}. 

Next we prove \eqref{complete2}. By \eqref{decompose} and \eqref{vanishing}, we may assume $u =1$. Consider
\[
e_{kWv}(a_2 \xi^k) S_{v,a_2} = \sum_{r \mmod v} e_v \Bigl(a_2 \frac{(Wr + \xi)^k}{kW} \Bigr).
\]
As $(v,W) = 1$, we can change variables by $t = \xi W^{-1} + r \in \bZ / v \bZ$, which gives
\[
e_{kWv}(a_2 \xi^k) S_{v,a_2} = \sum_{t \mmod v} e_v \Bigl(a_2 \frac{W^{k-1}}k t^k\Bigr).
\]
Since $\Bigl( v, a_2 \frac{W^{k-1}}k \Bigr) = 1$, we may apply \cite[Theorem 4.2]{vaughan97}, which gives
\[
S_{v,a_2} \ll v^{1-1/k} = q^{1-1/k}.
\]
By \eqref{decompose} and \eqref{vanishing}, we now have $S_{q,a} \ll q^{1-1/k}$.
\end{proof}

A standard calculation provides the following bound.

\begin{lemma} \label{integral}
We have 
\[
I(\bet) \ll \min \{ X, \| \bet \|^{-1} \}.
\]
\end{lemma}

Before continuing in earnest, we briefly describe the plan. We can modify \cite[Theorem 1.8]{Vau1989} to handle a set of minor arcs. At that stage, our major and minor arcs fail to cover the entire torus $\bT$, but we can bridge the gap using a classical circle method contraption known as \emph{pruning} (also used in Appendix \ref{unrestricted}). Adapting \cite[Lemma 7.2]{VW1991}, we can prune down to $q \le (\log P)^A$. Finally, by adapting \cite[Lemma 8.5]{VW1991}, we prune down to $q \le (\log P)^{1/4}$. 

In order to tailor the classical theory to suit our needs, we begin with the observation that
\begin{align} \notag &
\sum_{\substack{x \in S(m;R) \\ x \equiv \xi \mmod W}} e \Bigl( \alp \frac{x^k - \xi^k}{kW} \Bigr)
= \\  \label{StartingPoint} & \qquad
 \frac1W \sum_{t \mmod W} e \Bigl(-\frac \alp {kW} \xi^k - \frac t W \xi \Bigr) \sum_{x \in S(m;R)} e \Bigl(\frac \alp {kW} x^k + \frac tW x \Bigr).
\end{align}
The inner summation is a classical quantity with a linear twist.

\begin{lemma} [Minor arcs] \label{minor0} Suppose $0 < \del < (2k)^{-1}$, and let $\fm_1$ denote the set of real numbers $\gam$ with the property that if $a \in \bZ$, $q \in \bN$, $(a,q)=1$ and $|q \gam - a| \le P^{\frac12-k + \del k}$ then $q > P^{\frac12 + \del k}$. Put
\[
\iota(k) = \max_{\lam \in \bZ_{\ge 2}} \frac1{4\lam} (1- (k-2) (1-1/k)^{\lam-2}).
\]
Then, assuming $\eta \le \eta_0(\eps,k)$, we have
\[
\hat \nu(\alp) \ll P^{k + \eps} (P^{- \del}+P^{ -\iota(k)}) \qquad (\alp \in kW \fm_1).
\]
\end{lemma}

\begin{remark} We will later apply this with $\eps = \eps_k$, so that the condition $\eta \le \eta_0(\eps,k)$ will be met.
\end{remark}

\begin{proof}
Following the proof of \cite[Theorem 1.8]{Vau1989}, we find that if $\alp \in \fm_1$ and $1 \le m \le P$ then
\[
\sum_{x \in S(m;R)} e \Bigl(\frac \alp {kW} x^k + \frac tW x \Bigr) \ll P^{1+\eps} (P^{- \del}+P^{ -\iota(k)}).
\]
Indeed, already built into that proof are bounds uniform over linear twists; see \cite[Eq. (10.9)]{Vau1989}. The sum above is over $x \in S(m;R)$, where $1 \le m \le P$, rather than over $x \in S(P;R)$, however we can assume that $\sqrt P \le m \le P$ and then run Vaughan's argument.

Now, by \eqref{StartingPoint}, we have
\[
\sum_{\substack{x \in S(m;R) \\ x \equiv \xi \mmod W}} e \Bigl( \alp \frac{x^k - \xi^k}{kW} \Bigr)
\ll P^{1+\eps} (P^{- \del}+P^{ -\iota(k)}).
\]
From here, partial summation gives
\[
\hat \nu(\alp) = \sum_{\substack{x \in S(P;R) \\ x \equiv \xi \mmod W}} x^{k-1} e \Bigl( \alp \frac{x^k - \xi^k}{kW} \Bigr)
\ll P^{k+\eps} (P^{- \del}+P^{ -\iota(k)}).
\]
\end{proof}

\begin{lemma} [First pruning step] \label{prune1} Suppose $R \le M \le P$, where $R = P^\eta$ as before. Suppose $a \in \bZ$, $q \in \bN$ with $(a,q)=1$ and $|q \alp - a| \le M / (P^k R)$. Then for any $\eps > 0$ we have
\begin{align*}
\hat \nu(\alp) \ll_{\eps,k,W,\eta} & P^k  (\log P)^3  q^\eps 
\\ &\quad \cdot ( (q + P^k|q \alp - a|)^{-1/(2k)}
+ (MR/P)^{1/2} + q^{\frac12 - \frac1{2k}} (R/M)^{1/2}).
\end{align*}
\end{lemma}

\begin{proof} By partial summation and \eqref{StartingPoint}, it suffices to show that if $\sqrt{P} \le m \le P$ then
\begin{align*}
&\sum_{x \in S(m;R)} e \Bigl(\frac \alp {kW} x^k + \frac tW x \Bigr) \\
&\quad \ll P(\log P)^3 q^\eps ( (q + P^k|q \alp - a|)^{-1/(2k)}
+ (MR/P)^{1/2} + q^{\frac12 - \frac1{2k}}  (R/M)^{1/2}).
\end{align*}
To show this, we work through the proof of \cite[Lemma 7.2]{VW1991}; the inner sum of Eq. (7.4) therein becomes
\[
S = \sum_{y \in I \cap \bZ} e \Bigl( \frac \alp {kW} p^k (u_2^k - u_1^k)y^k + \frac tW p (u_2 - u_1)y \Bigr),
\]
where 
\[
I = (V/p, \min \{2V/p, m/(u_1p), m/(u_2p) \} ]
\]
is an interval of length at most $V/p$ and $u_1,u_2,p$ are the outer summation variables in \cite[Eq. (7.4)]{VW1991}. With reference to that proof, we have
\[
D = \gcd(kWq, ap^k(u_2^k-u_1^k), t qk p (u_2 - u_1)) \ll_{k,W} (q, p^k(u_2^k-u_1^k)),
\]
and
\[
S = \sum_{y \in I \cap \bZ} e_{q'}(a' y^k + b' y) e(\bet y^k),
\]
where 
\[
Dq' = kWq, \quad  Da' = ap^k(u_1^k - u_2^k),\quad  Db' = tkqp (u_2 - u_1),
\]
$\gcd(a',b',q') = 1$ and
\[ 
\bet = \frac{ p^k(u_2^k - u_1^k) } {kW} (\alp - a/q).
\]

Continuing to follow the proof of \cite[Lemma 7.2]{VW1991}, we now apply \cite[Lemma 4.4]{Bak1986} and \cite[Theorems 7.1 and 7.3]{vaughan97} in lieu of the more specific \cite[Lemma 2.8 and Theorems 4.1 and 4.2]{vaughan97}. One can check that
\[
|q \bet| \le (2k^2)^{-1} (V/p)^{1-k}.
\]
As $|I| \le V/p$, this condition enables us to apply \cite[Lemma 4.4]{Bak1986}, giving
\[
S = (q')^{-1}\sum_{x \le q'} e_{q'}(a' x^k + b' x) \int_I e(\bet z^k) \d z + O(q^{1- \frac1k+\eps} ).
\]
The error term $RU^2 q^{\frac12+\eps}$ in \cite[Eq. (7.5)]{VW1991} is enlarged to $RU^2 q^{1- \frac1k+  \eps }$, and the effect of applying \cite[Theorems 7.1 and 7.3]{vaughan97} is to increase the quantity $S_3$ appearing therein by a multiplicative factor of $O_{\eps,k,W}(q^{\eps/8})$. 

The remainder of the proof of \cite[Lemma 7.2]{VW1991} carries through in the present context, \emph{mutatis mutandis}. The eventual outcome of the changes above is to increase the term $q^{1/4} P (R/M)^{1/2}$ to $q^{\frac12 - \frac1{2k}} P (R/M)^{1/2}$, and we obtain the asserted bound.
\end{proof}

\begin{lemma} [Second pruning step] \label{prune2} Suppose $R = P^\eta$ with $0 < \eta < 1/2$, and that $a,q \in \bZ$ with $(a,q) = 1$ and $1 \le q \le (\log P)^A$. Then for some $c = c(\eta,A)$ we have
\[
\hat \nu(\alp) \ll_{\eps,k,w,\eta,A,c} P^k (q + P^k|q \alp - a|)^{\eps- \frac1k}
+ P^k \cdot \exp( - c \sqrt{\log P}) (1 + P^k |\alp - a/q|).
\]
\end{lemma}

\begin{proof} Again we apply partial summation and \eqref{StartingPoint}, leaving us to show that if $P^{0.99} \le m \le P$ then
\begin{align*}
g(\alp) &:= \sum_{x \in S(m;R)} e \Bigl(\frac \alp {kW} x^k + \frac tW x \Bigr)  \\
&\ll  P (q + P^k|q \alp - a|)^{\eps- \frac1k} + P \cdot \exp( - c \sqrt {\log P}) (1 + P^k |\alp - a/q|).
\end{align*}
This time we follow the proof of \cite[Lemma 8.5]{VW1991}. Writing $\bet = \alp - a/q$, this  initially formats our smooth Weyl sum as
\begin{align*}
g(\alp) = \sum_{\substack{d \mid kWq \\ q/d \in S(m;R)}} \sum_{\substack{y =1 \\ (y,d)=1}}^d 
&e((kWq/d)^{k-1} y^k a/d + tkqy/d)  \\ &\quad \cdot \Psi \Bigl( \frac{md}{kWq} , R; d, y, \frac{\bet(kWq/d)^k }{kW} \Bigr),
\end{align*}
where
\[
\Psi(Q,R; d,y, \gam) = \sum_{\substack{z \in S(Q; R) \\ z \equiv y \mmod d}} e( \gam z^k).
\]
The calculation by Vaughan and Wooley in the proof of \cite[Lemma 8.5]{VW1991} ensures that $\Psi(Q,R;d,y,\gam)$ is, up to a small additive error, independent of $y$. As $m \le P$ and $kW \ll_{k,W}1$, the outcome of this calculation is unaffected, and we obtain
\begin{align*}
g(\alp) &\ll P(1+P^k |\bet|)^{-1/k} \\&\qquad \cdot
\Bigl(\exp(- c \sqrt{\log P}) + \sum_{d \mid kWq} \frac{d}{q \varphi(d)} |\cW(d, a(kWq/d)^{k-1}, tkq)| \Bigr),
\end{align*}
where
\[
\cW(Q,A,B) = \sum_{\substack{y \mmod Q \\ (y,Q)=1}} e_Q(Ay^k + By).
\]
Our final task is to show that if $d \mid kWq$ then
\begin{equation*}
\cW(d, a(kWq/d)^{k-1}, tkq) \ll_{k,w, \eps} q^{1-\frac1k + \eps}.
\end{equation*}

One may readily verify the usual multiplicativity property: if $(Q_1, Q_2) = 1$ then
\begin{equation} \label{FirstHua}
\cW(Q_1 Q_2, A, B) = \cW(Q_1, AQ_2^{k-1}, B) \cdot \cW(Q_2, AQ_1^{k-1}, B);
\end{equation}
see \cite[Lemma 8.1]{Hua}. Next we analyse 
\[
\cW (p^i, A, tkq),
\]
when $p$ is prime and $p^i \| d$. If $p > w$ then $p^i \mid tkq$, so 
\[
\cW(p^i, A, tkq) = \cW(p^i, A, 0) \ll p^{i/2} (p^i, A)^{1/2},
\]
using \cite[Lemma 8.4]{VW1991}. Meanwhile, if $p \le w$ then we use the identity
\[
\cW(p^i, A, tkq) = S(p^i, A, tkq) - S( p^{i-1}, Ap^{k-1}, tkq),
\]
where
\[
S(Q, A, B) = \sum_{y \le Q} e_Q( Ay^k + By).
\]
Since $p \le w$, we have 
\[
\gcd(p^i,A, tkq) , \: \gcd(p^{i-1}, Ap^{k-1}, tkq) \ll_{k,w} (p^i,A),
\]
so we may use \cite[Eq. (7.9)]{vaughan97} to infer that
\[
\cW(p^i, A, tkq) \ll_{k,w} (p^i,A)^{1/k} (p^i)^{1- \frac1k}.
\]
In both cases we have
\[
|\cW(p^i, A, tkq)| \le c_{k,w} (p^i,A)^{1/k} (p^i)^{1- \frac1k},
\]
and inputting this into \eqref{FirstHua} reveals that
\[
\cW(d, A, tkq) \ll q^\eps (d,A)^{1/k} d^{1-\frac1k}.
\]

Apply this with $A = a (kWq/d)^{k-1}$. With this choice of $A$, we have
\begin{align*}
(d,A) \le (d,a) (d,(kWq/d)^{k-1}) &\le (kWq,a)  (d,(kWq/d)^{k-1})
\\ &\le kW (d,(kWq/d)^{k-1}).
\end{align*}
Letting $p^i \| d$ and $p^j \| kWq$ gives
\[
(d,A) \ll_{k,w} \prod_p p^{\min \{i, (k-1)(j-i)\} },
\]
and so
\begin{align*}
\cW(d, a(kWq/d)^{k-1}, tkq) &\ll q^\eps \prod_{\substack{p^i \| d \\ p^j \| kWq}} p^{i(1-\frac1k) + k^{-1} \min \{i, (k-1)(j-i)\}}
\\ &\ll q^\eps \prod_{p^j \| kWq} (p^j)^{1-\frac1k} \ll q^{1-\frac1k+\eps}.
\end{align*}
\end{proof}

To tie together what we have gleaned, we make a Hardy--Littlewood dissection. For $q \in \bN$ and $a \in \bZ$, let $\fM(q,a)$ be the set of $\alp \in \bT$ such that $|\alp - a/q| \le (\log P)^{1/4}/P^k$.
Let $\fM(q)$ be the union of the sets $\fM(q,a)$ over integers $a$ such that $(a,q) = 1$, and let $\fM$ be the union of the sets $\fM(q)$ over $q \le (\log P)^{1/4}$. By identifying $\bT$ with a unit interval, we may write $\fM(q)$ as a disjoint union
\[
\fM(q) = \bigcup_{\substack{a=0\\(a,q)=1}}^{q-1} \fM(q,a).
\]

First we consider the minor arcs $\fm := \bT \setminus \fM$. 

\begin{lemma} \label{minor} If $\eps > 0$ and $\alp \in \fm$ then $\hat \nu(\alp) \ll_{\eps,W,\eta}  X (\log X)^{\eps- \frac1{4k}}$.
\end{lemma}

\begin{proof} Let $\alp \in \fm$. If $\frac \alp {kW} \in \fm_1$, where $\fm_1$ is as in Lemma \ref{minor0} with $\del = (4k)^{-1}$, then Lemma \ref{minor0} applies and is more than sufficient (recall \eqref{X and P}). We may therefore assume that $\frac \alp {kW} \notin \fm_1$, and then deduce the existence of relatively prime integers $q > 0$ and $a$ for which $q + P^k |q \alp - a| \ll P^{3/4}$. If the `natural height' $q + P^k |q \alp - a| $ exceeds $(\log P)^{9k}$, then an application of Lemma \ref{prune1} with $M \asymp RP^{3/4}$ suffices. So we may suppose instead that $q + P^k |q \alp - a|  \le (\log P)^{9k}$. As $\alp \notin \fM$, we must also have 
\[
q + P^k |q \alp - a| \ge \max \Bigl \{q , P^k \Bigl| \alp - \frac a q \Bigr| \Bigr\}  > (\log P)^{1/4},
\]
and now Lemma \ref{prune2} delivers the sought inequality.
\end{proof}

We are ready to prove Lemma \ref{smooth decay}, in the case $k \ge 3$. As discussed at the beginning of this appendix, our task is to establish the estimate \eqref{DecayEst}. It will be useful to have \eqref{X and P} and \eqref{smooth normalisation} in mind. By a geometric series calculation, we have
\begin{equation} \label{geometric}
\widehat{1_{[X]}}(\alp) = \sum_{x \le X} e(\alp x) \ll \| \alp \|^{-1}.
\end{equation}

First suppose $\alp \in \fm$. By Dirichlet's approximation theorem, we obtain relatively prime integers $q$ and $a$ such that 
\[
1 \le q \le (\log P)^{1/4}, \qquad |q \alp - a| \le (\log P)^{-1/4}.
\]
As $\alp \notin \fM$, we must have 
\[
\|q \alp \| = |q \alp - a| > \frac{q (\log P)^{1/4}}{kWX},
\]
so
\[
\widehat{1_{[X]}}(\alp) \ll \| \alp \|^{-1} \ll \frac q{\|q \alp \|} \ll \frac{WX}{(\log P)^{1/4}}.
\]
By Lemma \ref{minor}, we now have \eqref{DecayEst}.

Next we consider the case in which $q=1$ and $\alp \in \fM(q)$, in other words $| \alp | \le (\log P)^{1/4}/P^k$. By Lemma \ref{major}, we have
\begin{equation} \label{e1}
\hat \nu (\alp) -  \rho(1/\eta) I(\alp) \ll \frac{P^k}{\log P}  (1 + P^k \| \alp \|) \ll \frac{P^k}{\sqrt{\log P}}.
\end{equation}
By Euler--Maclaurin summation \cite[Eq. (4.8)]{vaughan97}, we have
\begin{equation} \label{e2}
\widehat{1_{[X]}}(\alp) - I(\alp) \ll 1 + X \| \alp \| \ll \sqrt{\log P}.
\end{equation}
Coupling \eqref{e1} with \eqref{e2} yields
\[
\hat \nu(\alp) - \rho(1/\eta) \widehat{1_{[X]}}(\alp) \ll \frac{P^k}{\sqrt{\log P}} \ll Xw^{-1/k},
\]
and now \eqref{smooth normalisation} confirms \eqref{DecayEst}.

Finally, let $\alp \in \fM(q,a)$ with $2 \le q \le (\log P)^{1/4}$ and $(a,q) = 1$, and put 
\[
\bet = \alp - \frac a q \in \Bigl[-\frac{(\log P)^{1/4} }{P^k}, \frac{(\log P)^{1/4}}{P^k} \Bigr].
\]
Substituting
\[
\| \alp \| \ge q^{-1} - |\bet| \ge q^{-1} - \frac{(\log P)^{1/4}}{P^k} \ge \frac1{2q}
\]
into \eqref{geometric} gives 
\[
\widehat{1_{[X]}}(\alp) \ll q \ll (\log P)^{1/4}.
\]
By Lemma \ref{major}, we also have
\[
\hat \nu(\alp) \ll  \frac{P^k}{\sqrt{\log P}} + X |q^{-1} S_{q,a}|,
\]
and now Lemma \ref{complete} yields \eqref{DecayEst}. 

We have established \eqref{DecayEst} for all $\alp \in \bT$, assuming $k \ge 3$.

\subsection{Quadratic Weyl sums}

The purpose of this subsection will be a proof of Lemmas \ref{decay} and \ref{shifted decay}, together with the $k=2$ case of Lemma \ref{smooth decay}. In all of these cases $k=2$, so $\eta = 1$, and the weight function is simpler, namely
\[
\nu(n) = \begin{cases} x, & \text{if } n = \frac{x^2 - \xi^2}{2W} \text{ for some } x \in [P] \text{ with } x \equiv \xi \bmod W \\
0, & \text{otherwise.}\end{cases}
\]
For the Fourier transform of this weight function, we can obtain a power saving on the minor arcs, as in \cite{densesquares}. This will be used in the next appendix, in the proof of the restriction estimate. We keep this brief, as the analysis is essentially the same as that of \cite{densesquares}.

As discussed at the beginning of this appendix, we seek to establish \eqref{DecayEst}. The Fourier transform is given by
\[
\hat \nu (\alp) = \sum_{\substack{x \le P \\ x \equiv \xi \mmod W}} x e \Bigl( \alp \frac{x^2 - \xi^2}{2W} \Bigr).
\]
The following is a straightforward adaptation of \cite[Lemma 5.1]{densesquares}.

\begin{lemma} [Major arc asymptotic] \label{MajorArcAsymptotic}
Suppose that $\| q \alp \| = |q \alp - a|$ for some $q,a \in \bZ$ with $q > 0$. Then
\[
\hat \nu(\alp) = q^{-1} S_{q,a} I \Bigl (\alp - \frac a q \Bigr) + O_w( \sqrt{X} (q + X \| q \alp \|)).
\]
\end{lemma}

\noindent Lemmas \ref{complete} and \ref{integral} still hold when $k=2$, with the same proof. 

Following \cite{densesquares}, put $\tau = \frac1{100}$, and to each reduced fraction $a/q$ with $0 \le a < q \le X^\tau$ associate a major arc
\[
\fM_2(q,a) = \{ \alp \in \bT: \Bigl|\alp - \frac a q \Bigr| \le X^{\tau - 1} \}.
\]
Let $\fM_2$ denote the union of all major arcs, and define the minor arcs by $\fm_2 = \bT \setminus \fM_2$. The following is a straightforward adaptation of \cite[Eq. (5.3)]{densesquares}. 

\begin{lemma} \label{QuadraticMinor}  If $\eps>0$ and $\alp \in \fm_2$ then
\[
\hat \nu(\alp) \ll_\eps  X^{1-\frac{\tau}2 +\eps}.
\]
\end{lemma}

We proceed towards \eqref{DecayEst}. Let $\alp \in \bT$. By \eqref{normalisation}, it suffices to prove that
\begin{equation} \label{DecayEst2}
\hat \nu(\alp) - \widehat{1_{[X]}}(\alp) \ll Xw^{-1/2}.
\end{equation}
First suppose $\alp \in \fm_2$. As in the proof of \cite[Lemma 5.5]{densesquares}, we have
\[
\widehat{1_{[X]}} (\alp) \ll X^{1-\tau}.
\]
Pairing this with Lemma \ref{QuadraticMinor} yields \eqref{DecayEst2}.

Next, suppose $\alp \in \fM_2(q,a)$ for some coprime $q,a \in \bZ$ with $0 \le a < q \le N^\tau$, where $q \ge 2$. Lemmas \ref{complete}, \ref{integral} and \ref{MajorArcAsymptotic} give 
\[
\hat \nu (\alp) \ll Xw^{-1/2}.
\]
Meanwhile $\| \alp \| \ge (2q)^{-1}$, so 
\[
\widehat{1_{[X]}}(\alp) \ll \| \alp \|^{-1} \ll q \ll X^\tau,
\]
and now the triangle inequality yields \eqref{DecayEst2}.

Finally, when $q = 1$ and $\alp \in \fM_2(1,0)$, Lemma \ref{MajorArcAsymptotic} gives
\[
\hat \nu(\alp) - I(\alp) \ll_w X^{\frac12 + 2 \tau},
\]
and Euler--Maclaurin summation gives
\[
\widehat{1_{[X]}}(\alp) - I(\alp) \ll X^{2 \tau}.
\]
The triangle inequality now furnishes \eqref{DecayEst2}.

We have examined all cases, thereby completing the proofs of Lemmas \ref{decay}, \ref{smooth decay} and \ref{shifted decay}.

\section{Restriction estimates}
\label{AppendixC}

In this section we prove the restriction estimates claimed in Lemmas \ref{restriction}, \ref{smooth restriction} and \ref{shifted restriction}. 
The core elements of our setup are the same as in Appendix \ref{AppendixB}, but we repeat all of this for clarity. Put $R = P^\eta$, and define $P$ and $X$ by \eqref{X and P}. In the cases of Lemmas \ref{decay} and \ref{shifted decay} let $\eta = 1$ and $k=2$, and $\xi = 1$ in the latter scenario. Our weight function $\nu$ is defined by \eqref{nu defn}. When $k \ge 3$, we choose $\eta = \eta_k$ sufficiently small. We assume throughout that $X$ is sufficiently large in terms of $w$.

Let $\phi: \bZ \to \bC$ with $|\phi| \le \nu$ pointwise. For an appropriate restriction exponent $p$, our task is to establish the restriction inequality
\begin{equation} \label{RestrictionInequality}
\int_{\T} \abs{\hat{\phi}(\alpha)}^p \intd\alpha \ll X^{p-1}.
\end{equation}
The implied constant, in particular, will not depend on $w$. As 
\begin{equation} \label{sup}
\|\hat \phi\|_\infty \le \| \phi \|_1 \le \| \nu \|_1 \ll X,
\end{equation}
it suffices to show this when
\begin{equation} \label{pdef}
p = \begin{cases}
5 - \frac1{200}, &\text{if } k=2\\
s_0(k)+2-\frac1{200}, &\text{if } k\ge 4\\
8-10^{-8}, &\text{if } k=3, \\
\end{cases}
\end{equation}
where $s_0(k) \in \bN$ is as in Theorem \ref{intro main theorem}. Fix this choice of $p$.

To summarise what is written above, we seek to establish the restriction inequality \eqref{RestrictionInequality} when the exponent $p$ is given by \eqref{pdef}. This will prove Lemmas \ref{restriction}, \ref{smooth restriction} and \ref{shifted restriction} at one fell swoop.

Even moments play a key role, owing to the presence of an underlying Diophantine equation. In particular, they allow bounded weights to be freely removed. Let $2m$ be the greatest even integer strictly less than $p$. 

\begin{lemma} \label{slack} It holds that
\[
\int_\bT |\hat \phi (\alp)|^{2m} \d \alp \ll_{k,\eps}
\begin{cases} 
(WX)^{2m-1}, &\text{if } k \ge 4 \\
X^{2m-1+\eps},&\text{if } k =2 \\
P^{15.25-10^{-4}}, &\text{if } k = 3.
\end{cases}
\]
\end{lemma}

\begin{remark}
The sixth moment estimate, for the case $k=3$, has a slightly different flavour; it is a consequence of Wooley's `subconvex' mean value estimate \cite{Woo1995}. It is this that ultimately enables us to procure a $p$-restriction estimate with $p < 8$.
\end{remark}

\begin{proof} By orthogonality and the triangle inequality
\[
\int_\bT |\hat \phi (\alp)|^{2m} \d \alp  \ll P^{2m(k-1)} \cN,
\]
where $\cN$ is the number of solutions $(\bx, \by) \in S(P;P^\eta)^m \times S(P;P^\eta)^m$ to the Diophantine equation
\[
x_1^k + \cdots + x_m^k = y_1^k + \cdots + y_m^k.
\]
Note that adding a constant to $s_0(k)$ in the case $k\ge 4$ does not cause it to violate \eqref{s0 size}, and so we may assume that $2m\ge s_0(k)$ for the quantity $s_0(k)$ appearing in Theorem \ref{unrestricted lower bound}. For $k\ge 4$ we therefore have, by Theorem \ref{unrestricted lower bound}, that
\[
\int_\bT |\hat \phi (\alp)|^{2m} \d \alp  \ll P^{2m(k-1)} P^{2m-k} = P^{k(2m-1)} \ll (WX)^{2m-1}.
\]
The case $k=2$ is similar, as the crude bound $\cN \ll_\eps P^{2+ \frac{\eps}2}$ is standard. When $k=3$ the proof may be concluded using \cite[Theorem 1.2]{Woo1995}, which implies that $\cN \ll P^{3.25-10^{-4}}$.
\end{proof}

These estimates fall short of being sharp. By increasing the exponent, we are able to make them sharp, using Bourgain's epsilon-removal procedure \cite{Bou1989}. In the case $k=3$, an additional intermediate exponent is required.

\subsection{Epsilon-removal} 
\label{NotCubic}

In this subsection we assume that $k \ne 3$. The case $k = 3$ is treated in the next subsection by incorporating a small finesse. Denote by $\del$ a parameter in the range
\[
0 < \del \ll 1.
\]
Define the large spectra
\[
\cR_\del = \{ \alp \in \bT: |\hat \phi(\alp)| > \del X \},
\]
and note from \eqref{sup} that $\cR_\del$ is empty unless $\del \ll 1$. By the dyadic pigeonholing argument in \cite[\S 6]{densesquares}, it suffices to prove that 
\begin{equation} \label{goal1}
\meas(\cR_\del) \ll  \frac1 {\del^{p-10^{-8}}X}.
\end{equation}
Moreover, Lemma \ref{slack} ensures that
\[
(\del X)^{2m} \meas (\cR_\del) \le \int_\bT |\hat \phi(\alp)|^{2m} \d \alp  \ll_{k,\eps}
\begin{cases} 
(WX)^{2m-1}, &\text{if } k \ge 4 \\
X^{3+ \frac{\eps}2},&\text{if } k =2, 
\end{cases}
\]
so we may assume without loss that
\begin{equation} \label{wma}
\del >
\begin{cases} 
W^{2(1-2m) }, &\text{if } k \ge 4 \\
X^{-\eps}, &\text{if } k =2,
\end{cases}
\end{equation}
for any $\eps> 0$. Let $\tet_1, \ldots, \tet_R$ be $X^{-1}$-spaced points in $\cR_\del$. As 
\[
p-10^{-8} \ge 2k+ 0.3,
\]
it suffices to show that
\begin{equation} \label{goal2}
R \ll \del^{-2k-0.3}.
\end{equation}
Put $\gam = k + 0.1$. By the calculation in \cite[\S 6]{densesquares}, we have
\begin{equation} \label{h1}
\del^{2 \gam} X^\gam R^2 \ll \sum_{1 \le r, r' \le R} | \hat \nu(\tet_r - \tet_r')|^\gam.
\end{equation}

First suppose $k \ge 4$. Consider $\tet = \tet_r - \tet'_r$ in the summand on the right-hand side of \eqref{h1}. By Lemma \ref{minor}, the contribution from $\tet \in \fm$ is 
\[
O(R^2 (X (\log X)^{-1/(8k)})^\gam),
\]
and by \eqref{wma} this is $o(\del^{2\gam} X^\gam R^2)$. Hence
\begin{equation} \label{h2}
\del^{2 \gam} X^\gam R^2 \ll \sum_{\substack{1 \le r, r' \le R: \\ \tet = \tet_r - \tet'_r \in \fM}} | \hat \nu(\tet)|^\gam.
\end{equation}
If $\tet \in \fM(q,a)$ with $(a,q) = 1$ and $q \le (\log P)^{1/4}$ then, by Lemmas \ref{major}, \ref{complete} and \ref{integral} we have
\begin{align*}
\hat \nu (\tet) &\ll q^{-1/k} \min \Bigl \{X, \Bigl \| \tet - \frac a q \Bigr \|^{-1} \Bigr \} + \frac{P^k}{\log P} ( q + P^k \| q \alp \|)
\\ & \ll_k q^{-1/k} \frac X{1+ X | \tet - \frac a q |} + \frac {WX}{\sqrt{\log X}}.
\end{align*}
With $C$ a large positive constant, the contribution to the right-hand side of \eqref{h2} from denominators $q > Q_1 := C + \del^{-3k}$ is therefore bounded, up to a constant, by
\[
R^2 X^\gam (Q_1^{-\gam/k} + W^\gam (\log X)^{- \gam/2 })
\]
which, by \eqref{wma}, is negligible compared to the left-hand side of \eqref{h2}. Therefore
\begin{equation} \label{h3}
\del^{2\gam} R^2 \ll \sum_{1 \le r, r' \le R} G(\tet_r - \tet'_r),
\end{equation}
where 
\[
G(\alp) = \sum_{q \le Q_1} \sum_{a=0}^{q-1} \frac{q^{-\gam/k}}{(1+ X| \sin(\alp- \frac a q)|)^\gam}.
\]
The inequality \eqref{h3} is very similar to \cite[Eq. (4.16)]{Bou1989}, but with $N^2$ replaced by $X$, and with $Q_1 \sim \del^{-3k}$ rather than $Q_1 \sim \del^{-5}$. The exponents differ but, since $\gam > k$, Bourgain's argument carries through, and we obtain \eqref{goal2} in the case $k \ge 4$. 

Now suppose $k=2$. Consider $\tet = \tet_r - \tet'_r$ in the summand on the right-hand side of \eqref{h1}.  By Lemma \ref{QuadraticMinor}, the contribution from $\tet \in \fm_2$ is 
\[
O(R^2 (X^{1-\frac \tau 2 + \tau^2})^\gam),
\]
and by \eqref{wma} this is $o(\del^{2\gam} X^\gam R^2)$. Hence
\begin{equation} \label{h2quad}
\del^{2 \gam} X^\gam R^2 \ll \sum_{\substack{1 \le r, r' \le R: \\ \tet = \tet_r - \tet'_r \in \fM}} | \hat \nu(\tet)|^\gam.
\end{equation}
If $\tet \in \fM_2(q,a)$ with $(a,q) = 1$ and $q \le X^\tau$ then, by Lemmas  \ref{complete}, \ref{integral} and \ref{MajorArcAsymptotic}, we have
\[
\hat \nu (\tet)  \ll q^{-1/2} \min \Bigl \{X, \Bigl \| \tet - \frac a q \Bigr \|^{-1} \Bigr \} + O_w(X^{\frac12 + 2 \tau}) \ll \frac {Xq^{-1/2} }{1+ X | \tet - \frac a q |} .
\]
We obtain \eqref{h3}, but with $k=2$ in the definition of $G(\cdot)$, and Bourgain's argument again completes the proof.

\subsection{An intermediate exponent}

In this subsection let $k=3$, and let $\eta$ be a small positive constant as before. We proceed in two steps, effectively `pruning' the large spectrum. In the first step, we use a power-saving minor arc estimate for an auxiliary majorant to come close to a sharp restriction estimate. In the second step, we no longer require a power saving on the minor arcs, so we are able to obtain a sharp restriction estimate by reverting to the majorant $\nu$.

\subsubsection{A close estimate} Here we concede a small loss. By slightly increasing the exponent, we will recover it in the next subsection. Our goal for the time being is to establish the following.

\begin{lemma} \label{slackCubic} We have
\[
\sup_{|\phi| \le \nu} \int_\bT | \hat \phi(\alp) |^{8-10^{-6}} \d \alp \ll X^{-1} (WX)^{8-10^{-6}}.
\]
\end{lemma}

Similarly to the $k \ne 3$ case, it suffices to prove that
\[
\mathrm{meas} (\mathcal R_\del) \ll \frac1 {\del^{8 - 10^{-5}} X},
\]
where it is now convenient to redefine
\[
\cR_\del = \{ \alp \in \bT: |\hat \phi (\alp)| > \del WX \}.
\]
Note that in this setting $\del \ll W^{-1}$. Since Lemma \ref{slack} implies that
\[
(\del W X)^6 \meas(\cR_\del) \le \int_\bT | \hat \phi(\alp) |^6 \d \alp \ll P^{15.25 - 10^{-4}},
\]
we may assume without loss that
\begin{equation} \label{WMAcubic}
\del > P^{10^{-5} - \frac18}.
\end{equation}
Let $\tet_1, \tet_2, \ldots, \tet_R$ be $X^{-1}$-spaced points in $\cR_\del$. It suffices to show that
\begin{equation} \label{STPcubic}
R \ll \del^{-6.3}.
\end{equation}
For some $a_n \in \bC$ with $|a_n| \le 1$, we have $\phi(n) = a_n \mu(n)$, wherein we employ the majorant
\[
\mu(n) = 
\begin{cases} x^2, &\text{if } n = \frac{x^3 - \xi^3}{3W} \text{ for some }x \le P \text{ with } x \equiv \xi \mmod W \\
0, &\text{otherwise}.
\end{cases}
\]
With $\gam = 3.1$, the calculation in \cite[\S 6]{densesquares} implies
\begin{equation} \label{CubicBourgain1}
\del^{2\gam} (WX)^\gam R^2 \ll \sum_{r,r' \le R} | \hat \mu(\tet_r - \tet'_r)| ^\gam.
\end{equation}

Consider $\tet = \tet_r - \tet'_r$ in the summand. We require a circle method analysis. The majorant $\mu$ is very similar to the `auxiliary majorant' from \cite[\S 5]{chow}. Therein, the calculations are based on partial summation and Roger Baker's estimates, as packaged in \cite[\S 2]{wps}. The same approach yields the following major arc estimate, where the corresponding set of minor arcs is
\[
\fn := \{ \alp \in \bT: |\hat \mu (\alp)| \le P^{2.75 + 10^{-6}} \}.
\]

\begin{lemma} \label{package} If $\alp \in \bT \setminus \fn$ then there exist $q,a \in \bZ$ such that $0 \le a \le q-1$ and
\[
\hat \mu (\alp) \ll WX q^{\eps - \frac13} \Bigl(1 + X \Bigl| \alp - \frac a q \Bigr | \Bigr)^{-1/3}.
\]
\end{lemma}

We continue the proof of Lemma \ref{slackCubic}. In light of \eqref{WMAcubic}, the contribution from $\tet \in \fn$ to the right-hand side of \eqref{CubicBourgain1} is $o(\del^{2\gam} (WX)^\gam R^2)$, and so
\[
\del^{2\gam} (WX)^\gam R^2 \ll \sum_{\substack{r,r' \le R:\\ \tet \notin \fn}} | \hat \mu(\tet_r - \tet'_r)| ^\gam.
\]
With $C$ a large positive constant and $Q_2 = C + \del^{-9}$, we can use Lemma \ref{package} to obtain
\[
\del^{2\gam} R^2 \ll \sum_{1 \le r, r' \le R} G_1(\tet_r - \tet'_r),
\]
where
\[
G_1(\tet) = \sum_{q \le Q_2} \sum_{a=0}^{q-1} \frac{q ^{\gam(\eps - \frac 13)}}{ (1+ X | \sin (\tet - \frac a q) |)^{\gam/3}}.
\]
As $\gam (\eps- \frac13)> 1$, the proof may be completed by Bourgain's argument, as in the previous subsection.

\subsubsection{A sharp estimate} We are ready to prove \eqref{RestrictionInequality}; recall that $p = 8-10^{-8}$. This time let
\[
\mathcal \cR_\del = \{ \alp \in \bT: |\hat \phi (\alp) | > \del X \},
\]
where $\del \ll 1$. Following the same strategy, it suffices to prove that
\[
\meas (\cR_\del) \ll \frac1 {\del^{8 - 10^{-7}} X}.
\]
From Lemma \ref{slackCubic} we have
\[
(\del X)^{8 - 10^{-6}} \meas(\cR_\del) \ll W^8 X^{7 - 10^{-6}},
\]
so we may assume without loss that
\[
W^{-10^7} < \del \ll 1.
\]
With $1, 2, \ldots, R$ being $X^{-1}$-spaced points in $\cR_\del$, it again suffices to prove \eqref{STPcubic}. This time put $\phi(n) = a_n \nu (n)$, where $a_n \in \bC$ with $|a_n| \le 1$. The calculation in \cite[\S 6]{densesquares} then gives \eqref{h1}, and by the method of subsection \ref{NotCubic} (in the $k \ge 4$ case) we obtain \eqref{h3}; once again Bourgain's argument carries through.

We have considered all cases, thereby completing the proof of  Lemmas \ref{restriction}, \ref{smooth restriction} and \ref{shifted restriction}.

\section{Lefmann's criterion}\label{lefmann}

In this section we prove Theorem \ref{lefmann theorem}, which is a consequence of Lefmann's lemma \cite[Fact 2.8]{lefmann}.  The theorem is a special case of Theorem \ref{intro main theorem}, but can be established more simply, and we presently provide a proof. By rearranging the variables, we may suppose that for some $t \in \{ 6,7,\ldots,s \}$ we have
\begin{equation} \label{rearrange}
c_1 + \cdots + c_t = 0.
\end{equation}
Let
\[
a := c_{t+1} + \cdots + c_s.
\]
The case $a = 0$ was treated by Browning and Prendiville \cite{densesquares} so, for simplicity, we assume henceforth that $a \ne 0$. 

The following obscure fact was shown by Lefmann \cite[Fact 2.8]{lefmann}.

\begin{lemma} [Lefmann]
Let $c_1, \ldots, c_s$ be non-zero integers. Assume that there exists $t \in [s]$ for which we have \eqref{rearrange}.
Assume further that there exist $y \in \bZ \setminus \{ 0 \}$ and $y_1, \ldots, y_t \in \bZ$ such that
\begin{equation} \label{lef1}
c_1 y_1 + \cdots + c_t y_t = 0
\end{equation}
and
\begin{equation} \label{lef2}
ay^2 + c_1 y_1^2 + \cdots + c_t y_t^2 = 0.
\end{equation}
Then \eqref{gensq} is partition regular over $\bN$.
\end{lemma}

To complete the proof of Theorem \ref{lefmann theorem}, it remains to prove that the system has a solution $(y, \by) \in (\bZ \setminus \{ 0 \}) \times \bZ^t$. The number of such solutions in $[-P,P]^{t+1}$ is $\cN_1 - \cN_2$, where $\cN_1$ is the total number of integer solutions $(y,\by) \in [-P,P]^{t+1}$ and $\cN_2$ is the number of integer solutions $\by \in [-P,P]^t$ to
\[
c_1 y_1 + \cdots + c_t y_t = c_1 y_1^2 + \cdots + c_t y_t^2 = 0.
\]
Here $P$ is a large positive real number. 

\begin{lemma} \label{N2}
We have
\[
\cN_2 \ll P^{t-3} \log P.
\]
\end{lemma}

\begin{proof}  Rogovskaya \cite{Rogovskaya} showed that the system
\begin{align*}
x_1 + x_2 + x_3 &= y_1 + y_2 + y_3 \\
x_1^2 + x_2^2 + x_3^2 &= y_1^2 + y_2^2 + y_3^2
\end{align*}
has $\frac{18}{\pi^2} P^3 \log P + O(P^3)$ solutions $(\bx, \by) \in [P]^6$. By orthogonality, one can deduce from this that
\[
\int_{\bT^2} \Bigl| \sum_{|x| \le P} e(\alp_1 x + \alp_2 x^2) \Bigr|^6 \d \alp_1 \d \alp_2 \ll P^3 \log P.
\]
As $t \ge 6$, the lemma now follows from orthogonality, H\"older's inequality, and the trivial bound $\sum_{|x| \le P} e(\alp_1 x + \alp_2 x^2) \ll P$.
\end{proof}

\begin{lemma} \label{N1}
We have
\[
\cN_1 \gg P^{t-2}.
\]
\end{lemma}

\begin{proof} Let
\begin{align*}
Q(y_1,\ldots,y_{t-1}) &= c_t^{-1} (c_1y_1 + \cdots + c_{t-1} y_{t-1})^2 + \sum_{i \le t-1} c_iy_i^2 \\
&= \sum_{i \le t-1} (c_i + c_i^2/c_t)y_i^2 + 2 \sum_{1 \le i < j \le t-1} (c_i c_j/c_t) y_i y_j,
\end{align*}
and put
\[
C = |c_1| + \cdots + |c_t|.
\]
Now $\cN_1$ is greater than or equal to the number of integer solutions
\[
(y,y_1,\ldots,y_{t-1}) \in [-P/C,P/C]^t
\]
to
\[
ay^2 + Q(y_1, \ldots, y_{t-1}) = 0
\]
with $c_1 y_1 + \cdots + c_{t-1} y_{t-1} \equiv 0 \mod c_t$. By considering only multiples of $c_t$, we find that $\cN_1$ is greater than or equal to the number of integer solutions $\bx \in [-P/C^2,P/C^2]^t$ to
\[
Q_1(\bx) := Q(x_1, \ldots, x_{t-1}) +ax_t^2 = 0.
\]

For the sake of brevity, we appeal to Birch's very general theorem \cite[Theorem 1]{Birch}. The \emph{Birch singular locus} is the set $\cS$ of $\bx \in \bC^t$ at which the gradient of $Q_1$ vanishes identically. (In this instance, the Birch singular locus coincides with the usual singular locus.) We compute that
\[
\frac12 \partial_i Q(y_1,\ldots, y_{t-1}) = (c_i + c_i^2/c_t) y_i + \sum_{\substack{j \le t- 1 \\ j \ne i}} c_i c_j y_j / c_t,
\]
and so
\[
\frac{c_t}{2c_i} \partial_i Q = (c_t + c_i) y_i + \sum_{\substack{j \le t- 1 \\ j \ne i}} c_j y_j = c_t ( y_i - y_t) + \sum_{j \le t} c_j y_j = c_t ( y_i - y_t),
\]
where $y_t := -c_t^{-1} (c_1y_1 + \cdots + c_{t-1}y_{t-1})$. Therefore
\[
\cS = \{ (x,x,\ldots,x,0) \in \bC^t \},
\]
and in particular $\dim \cS = 1$. 

As $t - \dim \cS > 4$, Birch's theorem \cite[Theorem 1]{Birch} gives
\begin{equation} \label{BirchAsymptotic}
\cN_1 = \fS \fJ P^{t-2} + O(P^{t-2-\del}),
\end{equation}
for some constant $\del > 0$, where $\fS$ and $\fJ$ are respectively the \emph{singular series} and \emph{singular integral} arising from the circle method analysis. Birch notes in \cite[\S 7]{Birch} that $\fS$ is positive as long as $Q_1$ has a non-singular $p$-adic zero for each prime $p$, and that $\fJ$ is positive as long as $Q_1$ has a real zero outside of $\cS_1$. Note that $Q$ has a zero $\bx^* \in \bZ^{t-1}$ with pairwise distinct coordinates; this follows from \cite[Theorem 1.1]{Keil}, or from a circle method analysis. Now $(\bx^*,0)$ is a real zero of $Q_1$ outside of $\cS_1$, and is also a non-singular $p$-adic zero for each $p$. Hence $\fS \fJ > 0$, and by \eqref{BirchAsymptotic} the proof is now complete.
\end{proof}

The previous two lemmas yield $\cN_1 > \cN_2$, and this completes the proof of Theorem \ref{lefmann theorem}.

\begin{remark}
Lefmann's lemma generalises straightforwardly to higher degrees. We do not explore this avenue further, as any results thus obtained are likely subsumed by Theorem \ref{intro main theorem}.
\end{remark}


\begin{thebibliography}{amsalpha}

\bibitem[Bak86]{Bak1986}
R. C. Baker, \emph{Diophantine Inequalities}, London Math. Soc. Monographs (N.S.) \textbf{1,} Clarendon Press, Oxford, 1986.



\bibitem[Ber96]{bergelson}  V.~Bergelson, Ergodic Ramsey theory -- an update. \emph{Ergodic theory of $\Z^d$ actions (Warwick, 1993--1994)}, 1--61, London Math. Soc. Lecture Note Ser., 228, Cambridge Univ. Press, Cambridge, 1996.

\bibitem[Ber16]{bergyoutube} V.~Bergelson, \emph{Mutually enriching connections between ergodic theory and combinatorics - lecture 7}, CIRM lecture available at \url{https://bit.ly/2GNaL7d}.

\bibitem[BG16]{glasscock} V.~Bergelson and D.~Glasscock, \emph{Interplay between notions of additive and multiplicative largeness}, preprint \href{https://arxiv.org/abs/1610.09771}{\texttt{arXiv:1610.09771}} (2016).

\bibitem[BL96]{BergelsonLeibman}
V.~Bergelson and A.~Leibman, \emph{Polynomial extensions of van der {W}aerden's
  and {S}zemer\'edi's theorems}, J. Amer. Math. Soc. \textbf{9} (1996), no.~3,
  725--753.
  
  

\bibitem[Bir61]{Birch}
B. J. Birch, \emph{Forms in many variables}, Proc. Roy. Soc. Ser. A \textbf{265} (1961/62),
245--263.


\bibitem[Bou89]{Bou1989}
J.~Bourgain, \emph{On $\Lambda(p)$-subsets of squares}, Israel J. Math. \textbf{67} (1989), 291--311. 

\bibitem[BDG16]{bdg} J. Bourgain, C. Demeter and L. Guth, \emph{Proof of the main conjecture in Vinogradov's mean value theorem for degrees higher than three}, 
Ann. of Math. \textbf{184} (2016), 633--682.


\bibitem[BP17]{densesquares}
T.~D. Browning and S. Prendiville, \emph{A transference approach to a Roth-type theorem in the squares}, IMRN \textbf{7} (2017), 2219--2248.

\bibitem[Cha18]{chapman}
J. Chapman, \emph{Multiplicatively syndetic sets}, manuscript in preparation.

\bibitem[Cho16]{wps}
S. Chow, \emph{Waring's problem with shifts}, Mathematika \textbf{62} (2016), 13--46.

\bibitem[Cho17]{chow}
\bysame, \emph{Roth--Waring--Goldbach}, 
IMRN (2017), 34 pp.

\bibitem[Coo71]{cook} R.~Cook, \emph{Simultaneous quadratic equations}, 
J. London Math. Soc. \textbf{4} (1971), 319--326.

\bibitem[CRS07]{CRS07} E. Croot, I. Z. Ruzsa and T. Schoen, \emph{Arithmetic progressions in sparse sumsets}. Combinatorial Number Theory, 157--164, de Gruyter, Berlin, 2007.

\bibitem[CGS12]{cgs} P. Csikv\'ari, K. Gyarmati and A. S\'ark\"ozy, \emph{Density and ramsey type results on algebraic equations with restricted solution sets}, Combinatorica \textbf{32} (2012), 425--449.

\bibitem[CS17]{CwalinaSchoen} K. Cwalina and T. Schoen, \emph{Tight bounds on additive Ramsey-type numbers}, J. London Math. Soc., \textbf{96} (2017) 601--620.

\bibitem[Dav2005]{davenport} H.~Davenport, \emph{Analytic methods for Diophantine equations and Diophantine inequalities}, Second edition. Cambridge University Press, Cambridge, 2005.

\bibitem[Deu73]{deuber} W. Deuber, \emph{Partitionen und lineare Gleichungssysteme}, Math. Z. \textbf{133} (1973) 109--123.

\bibitem[DNB18]{DiNasso} M.~Di Nasso and L.~Baglini,
\emph{Ramsey properties of nonlinear Diophantine equations},
Adv. Math. \textbf{324} (2018), 84--117.

\bibitem[DS16]{DS2016}
S. Drappeau and X. Shao, \emph{Weyl sums, mean value estimates, and Waring's problem with friable numbers}, Acta Arith. \textbf{176} (2016), 249--299.


\bibitem[FGR88]{FGR}
 P. Frankl, R.~L. Graham and V. R\"odl, \emph{Quantitative theorems for regular systems of equations}, J. Combin. Theory Ser. A \textbf{47} (1988), 246--261.



\bibitem[FH14]{FranHost}
N.~Frantzikinakis and B.~Host, \emph{Higher order Fourier analysis of multiplicative functions and applications}, J. Amer. Math. Soc. \textbf{30} (2017), 67--157.

\bibitem[Fur77]{furstenberg}  H.~Furstenberg, \emph{Ergodic behavior of diagonal measures and a theorem of Szemer\'edi on arithmetic progressions}, 
J. Analyse Math. \textbf{31} (1977), 204--256. 

\bibitem[Grah07]{graham07} R.~Graham,
\emph{Some of my favorite problems in Ramsey theory}. Combinatorial number theory, 229--236, de Gruyter, Berlin, 2007. 

\bibitem[Grah08]{graham} \bysame, \emph{Old and new problems in Ramsey theory}.
Horizons of combinatorics, 105--118, Bolyai Soc. Math. Stud. {\bf 17}, Springer, Berlin, 2008.

\bibitem[GRS90]{ramseytheory} R.~Graham, B.~Rothschild and J.~H. Spencer, \emph{Ramsey Theory}. Wiley, 1990.

\bibitem[Gran08]{Granville} A. Granville, \emph{Smooth numbers: computational number theory and beyond}. Algorithmic number theory: lattices, number fields, curves and cryptography, 267--323, 
Math. Sci. Res. Inst. Publ. \textbf{44,} Cambridge Univ. Press, Cambridge, 2008. 

\bibitem[Gre02]{greensarkozy} B. Green, \emph{On arithmetic structures in dense sets of integers}, Duke Math. J. \textbf{114} (2002), 215--238.

\bibitem[Gre05]{greenprimes}
\bysame, \emph{Roth's theorem in the primes}, Ann.\ of Math.
  \textbf{161} (2005), 1609--1636.
  

  
   \bibitem[GL16]{greenlindqvist}
  B. Green and S.~Lindqvist,
  \emph{Monochromatic solutions to $x+y= z^2$},
 to appear in Canad. J. Math. (2018) \href{https://arxiv.org/abs/1608.08374}{\texttt{arXiv:1608.08374}}.
  
  \bibitem[GS16]{greensanders}
B. Green and T. Sanders, \emph{Monochromatic sums and products}, Discrete Anal. (2016), Paper No. 5, 43 pp.
  

      \bibitem[GT10a]{LinearEquationsPrimes} B.~Green and T.~Tao, \emph{Linear equations in primes}, Ann. of Math. \textbf{171} (2010) 1753--1850.
  
    \bibitem[GT10b]{greentaoregularity}
    \bysame,
    \emph{An arithmetic regularity lemma, an associated counting lemma, and applications}, An irregular mind, Bolyai Soc. Math. Stud. \textbf{21}, 261--334, J\'anos Bolyai Math. Soc. 2010.
    
    \bibitem[Har16]{harper}
    A.~Harper, \emph{Minor arcs, mean values, and restriction theory for exponential sums over smooth numbers},
Compos. Math. \textbf{152} (2016), 1121--1158. 
    

  \bibitem[Hei48]{heilbronn}
  H.~Heilbronn, \emph{On the distribution of the sequence $n^2\theta \pmod 1$}. 
Quart. J. Math., (1948). 249--256.

 \bibitem[HKM16]{computer} M. Heule, O. Kullmann
and V. Marek, \emph{Solving and verifying the Boolean Pythagorean triples problem via cube-and-conquer}, Theory and applications of satisfiability testing -- SAT 2016: 19th International Conference, Bordeaux, France, July 5-8, 2016, Proceedings (2016), Springer, 228--245.
   
  
\bibitem[Hin79]{hindman} N. Hindman, \emph{Partitions and sums and products of integers}, Trans. Amer. Math. Soc. \textbf{247} (1979), 227--245.

\bibitem[Hua65]{Hua} L.-K. Hua, \emph{Additive Theory of Prime Numbers}, Transl. Math. Monogr. \textbf{13,} Amer. Math. Soc., Providence, 1965.  

\bibitem[Kei14]{Keil}
E. Keil, \emph{On a diagonal quadric in dense variables}, Glasg. Math. J. \textbf{56} (2014), 601--628. 
  
\bibitem[KS06]{ks06}
 A.~Khalfalah and E.~Szemer\'edi,
\emph{On the number of monochromatic solutions of $x+y=z^2$}, 
Combin. Probab. Comput. \textbf{15} (2006), 213--227. 

\bibitem[Klo27]{kloosterman}
H. D. Kloosterman, \emph{On the representation of numbers in the form $ax^2+by^2+cz^2+dt^2$}, Acta Math. \textbf{49} (1927), 407--464. 

\bibitem[Lam16]{nature} E. Lamb, \emph{Maths proof smashes size record}, Nature \textbf{534} (2016), 17--18.

\bibitem[L\^e12]{le} T.~H. L\^e, \emph{Partition regularity and the primes}, C. R. Math. Acad. Sci. Paris \textbf{350} (2012), 439--441.

\bibitem[Lef91]{lefmann} H.~Lefmann,
\emph{On partition regular systems of equations},
J. Combin. Theory Ser. A \textbf{58} (1991), 35--53. 


   \bibitem[LP12]{LiPan} H. Li and H. Pan, \emph{A Schur-type addition theorem for primes}, J. Number Theory \textbf{132} (2012), 
117--126.

\bibitem[Mor17]{moreira} J.~Moreira, \emph{Monochromatic sums and products in $\N$}, Ann.\ of Math. \textbf{185} (2017), 1069--1090.

\bibitem [NSS18]{NSS18} J. Noel, A. Scott and B. Sudakov, \emph{Supersaturation in posets and applications involving the container method}, J. Combin. Theory Ser. A \textbf{154} (2018), 247--284. 

\bibitem[Pac18]{pach} P.~Pach, \emph{Monochromatic solutions to $x+y=z^2$ in the interval $[N,cN^4]$}, preprint \href{https://arxiv.org/abs/1805.06279}{\texttt{arXiv:1805.06279}} (2018).



\bibitem[Pre17a]{FourVariants}
S. Prendiville, \emph{Four variants of the Fourier analytic transference principle}, Online J. Anal. Comb. \textbf{12} (2017), 25 pp.

\bibitem[Pre17b]{quantitativeBL}
\bysame, \emph{Quantitative bounds in the polynomial Szemer\'edi theorem: the homogeneous case}, Discrete Anal. (2017), Paper No. 5, 34 pp.

\bibitem[Rad33]{rado} R. Rado, \emph{Studien zur Kombinatorik}, Math. Z. \textbf{36} (1933), 242--280.

\bibitem[Rog86]{Rogovskaya} N. N. Rogovskaya, \emph{An asymptotic formula for the number of solutions of a system of equations}. Diophantine approximations, Part II (Russian), 78--84, Moskov. Gos. Univ., Moscow, 1986. 

\bibitem[Rot53]{roth}
K.~F. Roth, \emph{On certain sets of integers}, J. London Math. Soc.
  \textbf{28} (1953), 104--109.

\bibitem[S\'ar78]{sarkozy}
  A. S\'ark\"ozy,
\emph{On difference sets of sequences of integers I},
Acta Math. Acad. Sci. Hungar.  \textbf{31} (1978), 125--149.

\bibitem[Sch1916]{schur}
I.\ Schur, \emph{\"{U}ber die Kongruenz $x^m+y^m\equiv z^m\pmod p$}, Jahresber.\ Dtsch.\ Math.-Ver. \textbf{25} (1916), 114--117.


\bibitem[Var59]{varnavides} P.~Varnavides, \emph{On certain sets of positive density}, J. London Math. Soc. \textbf{34} (1959), 358--360.


\bibitem[Vau86]{VaughanCubes}
R. C. Vaughan, \emph{On Waring's problem for cubes}, J. Reine Angew. Math. \textbf{365} (1986), 122--170. 


\bibitem[Vau89]{Vau1989}
\bysame, \emph{A new iterative method in Waring's problem}, Acta Math. \textbf{162} (1989), 1--71.

\bibitem[Vau97]{vaughan97}
\bysame, \emph{The {H}ardy-{L}ittlewood method}, 2nd ed., Cambridge
Tracts in Mathematics, vol. {\bf 125,} Cambridge University Press, Cambridge, 1997.

\bibitem[VW91]{VW1991}
R. C. Vaughan and T. D. Wooley, \emph{On Waring's problem: some refinements}, Proc. Lond. Math. Soc. (3) \textbf{63} (1991), 35--68.

\bibitem[Woo92]{Woo1992}
T. D. Wooley, \emph{Large improvements in Waring's problem}, Ann. of Math.
(2) \textbf{135} (1992), 131--164.

\bibitem[Woo95]{Woo1995}
\bysame, \emph{Breaking classical convexity in Waring's problem: sums of cubes and quasi-diagonal behaviour}, Invent. Math. \textbf{122} (1995), 421--451. 

\bibitem[Zha17]{zhao}
L. Zhao, \emph{On translation invariant quadratic forms in dense sets}, IMRN (2017), 44 pp.





\end{thebibliography}
\end{document}